\documentclass[a4paper]{article}

\usepackage[T1]{fontenc}
\usepackage[utf8x]{inputenc}
\usepackage{amsmath}
\usepackage{amsthm}
\usepackage{amssymb}
\usepackage{centernot}
\usepackage{amsfonts}
\usepackage[polish,english]{babel}
\usepackage{graphicx}
\usepackage{wrapfig}
\usepackage{mathtools}
\usepackage{stmaryrd}
\usepackage{tikz-cd}
\usepackage{enumitem}
    \setlistdepth{9}
\usepackage{tipa}
\usepackage[margin = 2cm]{caption}
\usepackage{hyperref}
\usepackage{accents}
\usepackage[thinlines]{easytable}
\usepackage{pgf,tikz,pgfplots}
\usepackage{mathrsfs}
\usetikzlibrary{arrows}
\usepackage{tikz-cd}

\usepackage{cleveref}

\setlist[itemize]{noitemsep, topsep=-10pt}
\setlist[enumerate]{noitemsep, topsep=-10pt}

\theoremstyle{plain}
\newtheorem{thm}{Theorem}[section]
\newtheorem{lem}[thm]{Lemma}
\newtheorem{prop}[thm]{Proposition}
\newtheorem{cor}[thm]{Corollary}

\theoremstyle{definition}
\newtheorem{defn}[thm]{Definition}

\newtheorem{question}[thm]{Question}

\theoremstyle{remark}
\newtheorem{rem}[thm]{Remark}
\newtheorem*{note}{Note}

\DeclareMathOperator{\rank}{rank}
\DeclareMathOperator{\core}{Core}

\makeatletter
\renewenvironment{proof}[1][\proofname]{\par
  \vspace{-\topsep}% remove the space after the theorem
  \pushQED{\qed}%
  \normalfont
  \topsep0pt \partopsep0pt % no space before
  \trivlist
  \item[\hskip\labelsep
        \itshape
    #1\@addpunct{.}]\ignorespaces
}{%
  \popQED\endtrivlist\@endpefalse
  \addvspace{6pt plus 6pt} % some space after
}

\makeatother
\expandafter\let\csname pgf@arrow@code@cm |\endcsname\relax

\makeatother
\pgfqkeys{/tikz/commutative diagrams}{
  Maps to/.style={
    /tikz/commutative diagrams/Rightarrow,
    /tikz/commutative diagrams/maps to},
  Maps from/.style={
    /tikz/commutative diagrams/Leftarrow,
    /utils/exec=\pgfsetarrowsend{cm |}}
}

\title{Investigating transversals as generating sets for groups}
\author{Maurice Chiodo, Robert Crumplin, Oscar Donlan, Pawe\l{} Piwek}
\date{\today}

\begin{document}

\newcommand{\todo}[1]{\color{red}(\textbf{TODO: #1})\color{black}}
\newcommand{\think}[1]{\color{blue}(\textbf{THINK: #1})\color{black}}
\newcommand{\set}[1]{\{#1\}}
\newcommand{\h}{0}
\newcommand{\w}{0}

\newcommand{\gen}[1]{\langle #1 \rangle}

\renewcommand*{\le}{\leqslant}
\renewcommand*{\leq}{\leqslant}
\renewcommand*{\ge}{\geqslant}
\renewcommand*{\geq}{\geqslant}

\maketitle
        \begin{abstract}
        In \cite{transgen} is was shown that for any group $G$ whose rank (i.e., minimal number of generators) is at most 3,  and any finite index subgroup $H\leq G$ with index $[G:H]\geq \rank(G)$, one can always find a left-right transversal of $H$ which generates $G$. In this paper we extend this result to groups of rank at most 4. We also extend this to groups $G$ of arbitrary (finite) rank $r$ provided all the non-trivial divisors of $[G:\core_G(H)]$ are at least $2r-1$. Finally, we extend this to groups $G$ of arbitrary (finite) rank provided $H$ is malnormal in $G$.
        \end{abstract}
    
{\let\thefootnote\relax\footnotetext{2010 \textit{AMS Classification:} 20E99, 20F05}}
{\let\thefootnote\relax}
{\let\thefootnote\relax\footnotetext{The second author was partially funded by an Undergraduate Summer Internship Scheme from King's College, Cambridge, and a Summer Research in Mathematics Bursary from the Faculty of Mathematics, Cambridge.}}
{\let\thefootnote\relax\footnotetext{The third author was partially supported by the Beker Fund from Selwyn College, Cambridge, and a Summer Research in Mathematics Bursary from the Faculty of Mathematics, Cambridge.}}
{\let\thefootnote\relax\footnotetext{The fourth author was partially funded by an Undergraduate Summer Internship Scheme from King's College, Cambridge, a Summer Research in Mathematics Bursary from the Faculty of Mathematics, Cambridge, and a London Mathematical Society
Undergraduate Research Bursary in Mathematics.}}

    \section{Introduction}
    
Given a group $G$ and a subgroup $H \leq G$ of finite index, one can consider  sets $S \subset G$ which are a \emph{left} (resp. \textit{right}) \textit{ transversals} of $H$ in $G$. That is, a complete set of left (resp. right) coset representatives. It then follows that we can consider \emph{left-right transversals}; sets which are simultaneously a left transversal, and right transversal, for $H$ in $G$.  While it is immediate in the case of finite index subgroups that left (resp. right) transversals exist (the case of infinite index subgroups is more complicated; existence of left or right transversals is equivalent to the axiom of choice, as shown in \cite[Theorem 2.1]{choice}), one might then ask the following question:

\begin{question}
Given a finitely generated group $G$ and a finite index subgroup $H$, does there always exist a left-right transversal of $H$ in $G$?
\end{question}

It turns out that such a transversal always exists, by an application (see \cite{cosetgraph}) of Hall's Marriage Theorem to the \emph{coset intersection graph} $\Gamma_{H,H}^{G}$ of $H$ in $G$ (\Cref{defn:coset-intersection}); a graph whose vertex set is the disjoint union of the left and the right cosets of $H$ in $G$, with edges between vertices whenever the corresponding cosets intersect. $\Gamma_{H, H}^{G}$ is thus bipartite as the set of left (and right) cosets are mutually disjoint. A proof without using Hall's Theorem was given in \cite[Theorem 3]{cosetgraph}.

In the context of a finitely generated group $G$, by defining $\rank(G)$ to be the smallest size of a generating set for $G$,  one might ask the following question relating generating sets to transversals:

\begin{question}
Given a finitely generated group $G$ and a finite index subgroup $H$ with $\rank(G)\leq[G:H]$, does there always exist a left transversal of $H$ which generates $G$?
\end{question}

It is also true that such a transversal always exists, as shown in \cite[Theorem 3.7]{transgen}. And clearly the reverse statement is true: if there exists a left transversal of $H$ which generates $G$ then $\rank(G)\leq [G:H]$. So we now have a necessary and sufficient condition for a finite index subgroup $H$ of a finitely generated group $G$ to have a left transversal which generates $G$; namely, that $\rank(G)\leq [G:H]$. Moreover, by taking inverses, and noting that the element-wise inverse of a left transversal is a right transversal, we see that the same condition holds for the existence of right transversals as generating sets. Combining all the observations so far, one might now consider the following more general question, which forms the main motivation of this paper:

\begin{question}\label{ques: motivating}
Given a finitely generated group $G$ and a finite index subgroup $H$ with $\rank(G)\leq[G:H]$, does there always exist a left-right transversal of $H$ which generates $G$?
\end{question}

This question was first studied in \cite{transgen},  where it was shown in \cite[Theorem 3.11]{transgen} that such a transversal always exists under the additional hypothesis that $\rank(G)\leq 3$. To achieve this, the authors introduced a new technique called \textit{shifting boxes}. This involves using the transitive action of a group $G$ on the set of left (or right) cosets of a subgroup $H \leq G$ to apply Nielsen transformations to a generating set of $G$ in a way such that the resulting generators lie inside (or outside) particular desired cosets of $H$. A study of the graph $\Gamma_{H,H}^{G}$ was conducted in \cite{cosetgraph}, and it was shown that the components are always complete bipartite graphs. This description of $\Gamma_{H,H}^{G}$ gave rise to a combinatorial model of coset intersections known as `chessboards'. Chessboards provide an approach to answering versions of \Cref{ques: motivating}, as indeed was done in \cite{transgen}, and which we do throughout this paper. In particular, we use these techniques, as done in \cite[Theorem 3.11]{transgen}, to relax the hypothesis of $\rank(G)\leq 3$, up to $\rank(G)\leq 4$. That is, we show as the first main result of this paper that:

\begin{thm}\label{thm: rank4case}
Let $G$ be a group of rank $4$ with a finite index subgroup $H$, such that $[G:H] \geq 4$. Suppose $S$ is a set of $4$ elements which generate $G$, then there exists a sequence of Nielsen transformations taking $S$ to a new set $\tilde{S}$, such that the elements of $\tilde{S}$ may be extended to a left-right transversal of $H$ in $G$.
\\Hence, given a finitely generated group $G$ of rank $4$ and a finite index subgroup $H$, there exists a left-right transversal of $H$ which generates $G$ if and only if $\rank(G)\leq[G:H]$.
\end{thm}

Notice that in the above theorem we speak of a much stricter condition on the left-right transversal found. Rather than showing existence of \textit{some} left-right transversal which generates the group, we instead take a generating set $S$ and Nielsen-transform it to a new generating set $S'$ which \textit{extends} to a left-right transversal (that is, $S'$ is a subset of a left-right transversal). As it turns out, all of our results making progress on \Cref{ques: motivating} are of this form: taking a generating set $S$ and Nielsen-transforming it to one which extends to a left-right transversal. We discuss this in more detail in \Cref{sec:finite-setting}, where we re-phrase our motivating question as \Cref{main Q}.

The technique of shifting boxes, as applied to chessboards, also allows us to make progress on \Cref{ques: motivating} under different additional hypotheses, this time invoking divisibility conditions on various subgroup indices. Taking $\core_G(H) := \cap_{g \in G}gHg^{-1}$ to be the \textit{core} of the subgroup $H$ in $G$ (that is, the largest normal subgroup of $G$ contained in $H$), we state another main result of this paper which makes further progress on \Cref{ques: motivating}.

\begin{thm}
\label{thm: divisor_thm}
Let $G$ be a group of rank $r$ with a finite index subgroup $H$, such that $[G:H] \geq r$. Suppose that each non-trivial divisor of $[G:\core_G(H)]$ is at least $2r-1$. Then any generating set $S$ having size $r$ may be Nielsen-transformed to a set $\tilde{S}$ that may be extended to a left-right transversal of $H$ in $G$.
\end{thm}

    Unfortunately the method of shifting boxes becomes combinatorially intractable when the rank of the underlying group gets sufficiently large. Hence other techniques become necessary. In this paper we introduce one such technique, which we call \textit{L-spins} applied to \textit{configurations} of multisets. A configuration of a multiset $S\subset G$ is simply a selection of some rows and columns of a chessboard of $H\leq G$ and the corresponding elements of $S$ lying in the intersections of these rows and columns (\Cref{defn:configuration}); one might think of this as a ``minor'' of a chessboard. And an L-spin is a way to re-arrange, via Nielsen transformations, this multiset on the configuration (\Cref{defn:L-spin}). We use these new techniques to prove the following version of \Cref{ques: motivating} with the added hypothesis that $H\leq G$ is a malnormal subgroup:
    
    \begin{thm}
    \label{thm:malnormal}
    Given a malnormal $H\leq G$ of finite index and a generating multiset $S$ of size $[G:H]$, $S$ can be Nielsen transformed to a left-right transversal.
\end{thm}

    Of course, only finite groups have finite index malnormal subgroups, as discussed in \Cref{sec:applications}, so the above theorem is only relevant for finite groups.  Note that the case when $H$ is normal is easily resolved in the affirmative, as left and right cosets of normal subgroups match up so one can immediately apply \cite[Theorem 3.7]{transgen}. Moreover, we resolve the special case of \Cref{ques: motivating} where $H$ is very close to  normal (that is, when $[H:xHx^{-1}\cap H]\leq 2$ for all $x\in G$), again in the affirmative as \Cref{prop:almost-normal}. Thus we see that if $H\leq G$ is very close to normal or very far from normal, then our motivating question is resolved in the affirmative in each case. So it is the instances where $H$ is somewhere between normal and malnormal where further investigation could be done, as are the instances where $\rank(G)>4$.
    
    This paper is laid out as follows:
    \\In \Cref{sec:graphs} we give an overview of the structure results and techniques introduced in \cite{cosetgraph,  transgen} on coset intersection graphs and chessboards. In \Cref{sec:shifting} we generalise some of the shifting boxes techniques from \cite{transgen} to obtain further results which become our main approach to answering questions about left-right transversals as generating sets. We then use these new results to prove \Cref{thm: divisor_thm}. In \Cref{sec:rank4} we apply our new techniques on shifting boxes developed in \Cref{sec:shifting}, to prove the rank-4 case of our motivating question; this is \Cref{thm: rank4case}. This proof is done in the similar style as the proof of the rank-3 case given in \cite[Theorem 3.11]{transgen}; as a sequence of reductions to various sub-cases depending on where certain generators lie in the chessboards of $H\leq G$. In \Cref{sec:finite-setting} we focus on the special case of \Cref{ques: motivating} where $G$ is finite, looking at additional hypotheses on $H\leq G$. Namely, when $H$ is cyclic (\Cref{lem:cyclic-nielsen}) or isomorphic to $C_{p}\times C_{p}$ for $p$ prime (\Cref{lem:CpxCp}), answering \Cref{ques: motivating} in the affirmative in all of these cases. In \Cref{sec: configurations} we define configurations and L-spins, and then use these to show some normal form theorems for configurations. We use these results to answer \Cref{ques: motivating} under the additional hypothesis that $H$ is very close to normal in $G$ (\Cref{prop:almost-normal}) and when $H$ is malnormal in $G$ (\Cref{thm:malnormal}).

\section{Coset intersection graphs and transversals as generating sets}\label{sec:graphs}

The work in Sections \ref{sec:graphs},\ref{sec:shifting},\ref{sec:rank4} is a direct extension of work done by Button, Chiodo and Zeron-Medina in the two papers \cite{cosetgraph,  transgen}. In \cite{cosetgraph}, the coset intersection graphs associated to a group $G$ were studied:

\begin{defn}\label{defn:coset-intersection}
     For a group $G$ with finite index subgroups $H,K  \leq G$, the associated \emph{coset intersection graph}, $\Gamma_{H,K}^{G}$, is the bipartite graph on vertex set $V = \{gH \ | \ g \in G\}\sqcup\{Kg \ | \ g \in G\}$ (where we treat $gH$ and $Kg$ as different vertices even if $gH=Kg$ as sets), with an edge joining two cosets with non-empty intersection in $G$.
\end{defn}

\begin{thm}{{\cite[Theorem 3]{cosetgraph}}}
Let $G$ be a group and $H,K \leq G$ of finite index. Then $\Gamma_{H,K}^G$ is a disjoint union of components of the form $K_{s_i,t_i}$, where $\frac{t_i}{s_i}$ is independent of $i$.
\end{thm}

The main application of this result is to the case when $H = K$, where it says that every right $H$ coset contained in a double coset $HgH$ intersects every left $H$ coset in the double coset. Pictorially, one can consider $G$ as a disjoint union of square grids corresponding to double cosets of $H$, where the columns and rows of each grid respectively correspond to left and right $H$ cosets contained in the double coset, and each square in each grid corresponds to a non-empty intersection of the (cosets represented by) the associated row and column; see \Cref{fig:chessboard1}. In \cite{transgen}, these square grids are referred to as \textit{chessboards}. We will work extensively with chessboards in this paper, and from hereon in keeping with the notation in \cite{cosetgraph,  transgen} we will always take rows to represent right cosets and columns to represent left cosets.

\begin{figure}[h]
    \centering
    \begin{center}
    	\begin{tikzpicture}[line cap=round,line join=round,x=1.5cm,y=1.5cm]
    	\foreach \x in {0,1,...,3}
    	{
    		\draw[line width = 1pt] (1+\x,0) -- (1+\x, 3);
    	}
    	\foreach \y in {0,1,...,3}
    	{
    		\draw[line width = 1pt] (1, \y) -- (4, \y);
    	}
    	
    	\draw (0.6, 0.4) node {$Hg_3$};
    	\draw (0.6, 1.4) node {$Hg_2$};
    	\draw (0.6, 2.4) node {$Hg_1$};
    	\draw (1.5, 3.3) node {$g_1'H$};
    	\draw (2.5, 3.3) node {$g_2'H$};
    	\draw (3.5, 3.3) node {$g_3'H$};
    	\draw (1.5, 2.5) node[scale = 0.7] {$g_1'H \cap Hg_1$};
    	\draw (2.5, 2.5) node[scale = 0.7] {$g_2'H \cap Hg_1$};
    	\draw (3.5, 2.5) node[scale = 0.7] {$g_3'H \cap Hg_1$};
    	
    	\end{tikzpicture}
    \end{center}
    \caption{Example of chessboard representing $Hg_1H$.}
    \label{fig:chessboard1}
\end{figure}

The main results of this paper concern \textit{transversals}. 

\begin{defn}
     Given a group $G$ and $H\leq G$, a set $S \subset G$ is called a \emph{left} (resp. \emph{right}) \emph{transversal} for $H$ in $G$ if it is a collection of representatives of all of the left (resp. right) cosets in $G$. A \emph{left-right transversal} is a set which is simultaneously a left and a right transversal for $H$ in $G$. We say a transversal $S$ is a \emph{generating transversal} if it generates $G$ as a group.
\end{defn}

A well known application of Hall's marriage theorem implies that every finite index subgroup possesses a simultaneous left-right transversal. In \cite{cosetgraph}, it was noted that a left-right transversal of a finite index subgroup $H \leq G$ can be obtained directly, simply by choosing an element from each diagonal entry of each chessboard of $H$. The ease and directness of this observation suggests that chessboards may be a useful tool to answer related questions.

The question approached in this paper is whether a finitely generated group $G$ necessarily has a %minimal\footnote{Here `minimal' means that the set has the smallest possible cardinality over all sets generating $G$.} 
generating set contained in a left-right transversal of a given finite index subgroup $H$. The direction of approach is `constructive' via chessboards, in the style of the above paragraph. In particular, using chessboards, we study how the positions (relative to cosets of $H$) of elements of a multiset\footnote{The reason why multisets are used instead of sets is that a Nielsen transformation of a set may result in multiple occurrences of a group element, which we do not want to discard. We do not use $n$-tuples because the order of the elements plays no role.} $S \subset G$ change under Nielsen moves.

\begin{defn}
Given a multiset of elements $S \subset G$, a multiset $\tilde{S}$ is said to be obtained from $S$ via a \textit{Nielsen move} if it is obtained by either of the following procedures:
    \begin{enumerate}
        \item obtained by replacing some $g \in S$ with $g^{-1}$
        \item obtained by replacing some $g \in S$ with $hg$ or $gh$ for some $h \in S\smallsetminus\{g\}$
    \end{enumerate}
    Two multisets are \textit{Nielsen equivalent} if they can be obtained from one another via a finite sequence of Nielsen moves; such a finite sequence is referred to as a \textit{Nielsen transformation}.
\end{defn}

To illustrate how chessboards provide a method of viewing `transversality' of generating sets, consider the most basic result in this direction, established in {{\cite[Theorem 3.7]{transgen}}}. The proof is included because many arguments in this paper will be of a similar style. 

\begin{prop}
   \label{prop: left-clean}
    Let $G$ be a group with finite index subgroup $H$, and take a multiset $S \subset G$ whose elements generate $G$. If $[G:H] \geq |S|$, then $S$ is Nielsen equivalent to a multiset with entries in distinct left cosets of $H$ (in particular, there exists a left transversal of $H$ generating $G$).
\end{prop}

\begin{defn}
    We shall refer to a coset as \emph{empty} if it contains no elements of $S$.
\end{defn}

\begin{proof}
    Left cosets of $H$ correspond to columns in the chessboards. Two elements of $S$ belong to the same left coset of $H$ if they belong to the same column of a chessboard. Suppose $s, s' \in S $ belong to the same column. Then $s'$ may be replaced by $s^{-1}s'$ (corresponding to a pair of Nielsen moves), which now belongs to $H$. Hence we may suppose every element of $S$ either belongs to $H$ (a $1\times 1$ chessboard), or to a column which contains no other entries of $S$. An example of such a configuration is shown in \Cref{fig:Box_Shifting_sec2}, with the dots representing elements of $S$.

    \begin{figure}[h]
        \centering
        \begin{center}
        	\renewcommand{\w}{1}
        	\renewcommand{\h}{1}
        	\begin{tikzpicture}[line cap=round,line join=round,x=1.5cm,y=1.5cm]
        	\foreach \y in {0,1,...,\h}
        	{
        		\draw [line width = 1pt] (0,\y+1.5) -- (\w,\y+1.5);
        	}
        	\foreach \x in {0,1,...,\w}
        	{
        		\draw [line width = 1pt] (\x,1.5) -- (\x,\h+1.5);
        	}
        	
        	\renewcommand{\w}{3}
        	\renewcommand{\h}{3}
        	\foreach \y in {0,1,...,\h}
        	{
        		\draw [line width = 1pt] (2.5,\y+0.5) -- (\w+2.5,\y+0.5);
        	}
        	\foreach \x in {0,1,...,\w}
        	{
        		\draw [line width = 1pt] (\x+2.5,0.5) -- (\x+2.5,\h+0.5);
        	}
        	
        	\renewcommand{\w}{3}
        	\renewcommand{\h}{3}
        	\foreach \y in {0,1,...,\h}
        	{
        		\draw [line width = 1pt] (7,\y+0.5) -- (\w + 7, \y+0.5);     
        	}
        	\foreach \x in {0,1,...,\w}
        	{
        		\draw [line width = 1pt] (\x + 7, 0.5) -- (\x + 7, \h + 0.5);
        	}
        	\draw [fill = black] (0.5, 2.25) circle (3.5pt);
        	\draw [fill = black] (0.5, 1.75) circle (3.5pt);
        	\draw [fill = black] (3, 3) circle (3.5pt);
        	\draw [fill = black] (4, 3) circle (3.5pt);
        	\draw [fill = black] (5, 2) circle (3.5pt);
        	\draw [fill = black] (7.5, 1) circle (3.5pt);
        	\draw [fill = black] (8.5, 2) circle (3.5pt);
        	\end{tikzpicture}
        \end{center} 
        \caption{A configuration in which elements outside of $H$ lie in distinct columns of chessboards.}
        \label{fig:Box_Shifting_sec2}
    \end{figure}
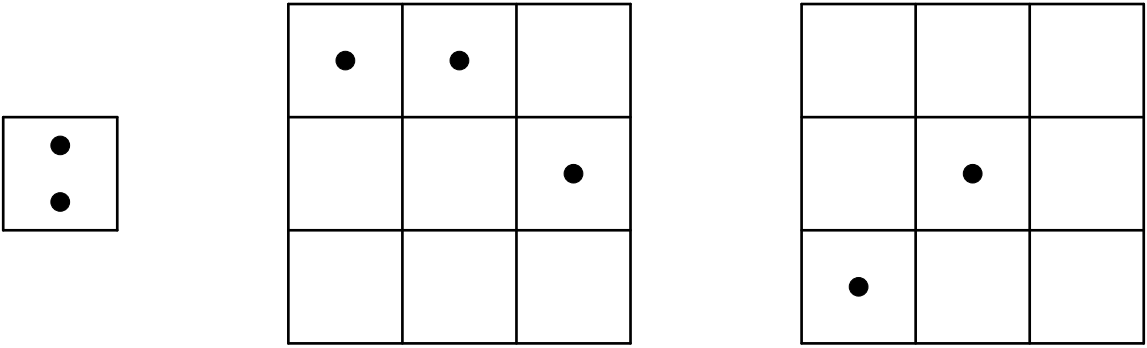
    
    The final stage of the proof is to `extract', via Nielsen moves, elements of $S\cap H$ into columns not yet containing elements of $S$. $S$ generates $G$ (this property is preserved under Nielsen equivalence) and $G$ acts transitively on the columns of the chessboards (i.e. on left cosets of $H$), and because $[G:H] \geq |S|$ there is at least one empty column provided $|S\cap H| \geq 2$ (if this is not the case, then the elements of $S$ already lie in distinct columns). 
    
    Suppose $gH$ is an empty column. It is possible to write $g = t_{1}^{\epsilon_1}\ldots t_{n}^{\epsilon_n}$, where $t_i \in S$, $\epsilon_i \in \{1,-1\}$ for each $i$. By taking $n$ minimal, it follows that there is an empty column of the form $s_2^{\epsilon}s_1H$, where $s_2,s_1 \in S$ (this works even when $n = 1$ because $S\cap H$ is non-empty). Because $|S\cap H| \geq 2$, we may find $h \in S\cap H$ such that $h \neq s_2$. 
    Then the replacement $h \mapsto s_2^{\epsilon}s_1h$ is a composition of Nielsen moves, and $s_2^{\epsilon}s_1h$ belongs to the empty column. This can be repeated until the multiset has only one entry in $H$.
    
    To obtain the parenthetical claim in the proposition, the set $S'$ obtained from $S$ by the above sequence of Nielsen moves still generates $G$, and may be extended to a left transversal of $H$ by adjoining elements from the remaining empty columns of the chessboards.
    \end{proof}
    
There were two main steps in the above proof. The first was the `cleaning' procedure: performing Nielsen moves until no two elements of $S$ shared a column outside of $H$. The second was the `extraction' procedure, Nielsen-moving elements of $S\cap H$ outside of $H$, preserving the property that no two elements share a column outside of $H$.

\begin{defn}
     Let  $H \leq G$ be of finite index, and take $S \subset G$. We say $S$ is \emph{left-clean} relative to $H$ if $S\smallsetminus H$ is non-empty and its elements lie in distinct left cosets of $H$, and so distinct columns of the chessboards. We similarly define \emph{right-clean}, and  \emph{left-right clean}.
\end{defn}

Given $H$, in addition to a multiset being clean we will also refer to a set of elements being \emph{diagonal}. A set is diagonal if no two elements belong to the same left or right coset of $H$. This terminology is motivated by chessboard diagrams: with appropriate ordering of columns and rows, a diagonal set can be drawn as belonging only to the diagonal of each board.

\begin{defn}
     Let $S \subset G$ be a multiset which is left-clean relative to $H$. A \emph{left-extraction} of an element of $S$ is a Nielsen transformation $S \rightarrow S'$ such that $|S'\cap H| = |S\cap H|-1$ and if $S$ is left-clean, then $S'$ is. We can define \emph{right-extraction} analogously. For the majority of this paper we will be interested in \emph{left-right extraction} and when this is clear from the context we will omit the `left-right'.  
\end{defn}

It is clear that \Cref{prop: left-clean} presents a procedure to Nielsen transform a given multiset into a left-clean set, and a general left-extraction procedure. Furthermore it is clear that the proposition holds with `right' replacing `left' throughout. In particular, given a multiset $S$, one may combine the left- and right- cleaning procedures (e.g. perform them in turn) to give a left-right cleaning procedure. However, it is not obvious how to combine the left- and right- extraction processes to give a left-right extraction process. Such a general procedure would, by modifying the argument in the proposition, result in a proof that any finitely generated group has a left-right transversal (of any subgroup $H$ with sufficiently large finite index) which generates the group.

The next section of the paper gives some general left-right extraction techniques. Before proceeding with this, it is noted that when $S$ has few elements, it is possible to explicitly devise ad-hoc left-right extraction procedures. Such procedures were presented in \cite{transgen} (the method of `shifting boxes'), and an example of a result obtained is given below.

\begin{thm}{{\cite[Theorem 3.11]{transgen}}} %(Theorem 3.11 in \cite{transgen}) 
Let $G$ be a group, $S$ a generating set for $G$ such that $|S| \leq 3$, and $H \leq G$ a subgroup of finite index such that $|S| \leq [G:H]$. Then $S$ is Nielsen-equivalent to a set $S'$ which is contained in a left-right transversal of $H$ in $G$.
\end{thm}

\section{Results on shifting boxes, and a proof of \Cref{thm: divisor_thm}}\label{sec:shifting}
In this section we present some left-right extraction procedures which are possible under certain conditions, as well as make some definitions which will be useful in the proof of \Cref{thm: rank4case}. To conclude the section, we apply our new extraction procedures to give a range of groups and subgroups for which any generating set of the group may be Nielsen transformed into a set which extends to a left-right transversal of the subgroup, finishing with a proof of \Cref{thm: divisor_thm}.

Taking $H\leq G$, our first definition here is motivated by viewing  chessboards as orbits of $H$ acting on its coset space $G/H$ by multiplication. We can visualise the left (respectively right) action of $H$ on $G/H$ (respectively $H \backslash G$) via the permutations it induces on the columns (respectively rows) of each chessboard. The same is true if we further restrict the action to $K \leq H$, however in this case whilst clearly the action still preserves the chessboards, it needn't be transitive on the columns/rows of a given chessboard. We will be interested in the case $K = \gen{S \cap H}$, where $S$ is a multiset of elements which we are trying to left-right extract from.  
\begin{defn}
     Let $K \leq H \leq G$, and suppose $S$ is a multiset of elements in $G$. With $K$ acting by left-multiplication on $G/H$, a $K$-orbit of a left coset of $H$ is \emph{sparse} if it contains a coset which doesn't contain an element of $S$. We also take the corresponding definition of sparseness for the right-multiplication action of $K$.
\end{defn}

For example, if we consider $K$ acting on the left, then saying $gH$ has a sparse $K$-orbit means we can find $k \in K$ such that $kgH$ is represented by a column (in the same chessboard as $gH$, as $K\leq H$) with no element of $S$ belonging to this new column.

The following lemma gives a left-right extraction technique which works provided $\gen{S \cap H}$ has left \emph{and} right sparse orbits in some chessboard.

\begin{rem}\label{rem: natural_condition}
To see why this is a natural condition, consider the simple scenario where there exists $g \in S\smallsetminus H$ and distinct $h$, $ h'\in S \cap H$ with $hgH$ and $Hgh'$ representing a column and row (respectively) which do not contain any element of $S$. Then combining the pair of Nielsen moves:
$$h \mapsto hg$$
$$hg \mapsto hgh'$$
constitutes a left-right extraction of an element (in this case $h$) from $S$. 
\end{rem}

\begin{lem}[Extraction Lemma]
\label{lem: EL}
     Let $G$ be a group, $H\leq G$ be of finite index, and $S\subset G$ a left-right clean multiset such that $|S| \leq [G:H]$. Suppose all of the following hypotheses hold: 

     \begin{enumerate}
         \item $|S \cap H|\geq 2$
         \item There exists $g \in S$ such that $gH$ and $Hg$ have sparse $\gen{S \cap H}$-orbits
     \end{enumerate}

     Then it is possible to left-right extract an element from $S \cap H$.
\end{lem}
\begin{proof} 
    Let $S\cap H = \set{h_1,\ldots,h_a}$ and $S\smallsetminus H = \set{g_1, \ldots, g_b}$ (note that $S$ is a multiset so $|S\cap H| = a \geq 2$, but distinct indices do not necessarily correspond to distinct elements). The first aim is to show that suitable Nielsen transformations reduce the hypotheses to a situation which is only slightly harder than the simple scenario presented in \Cref{rem: natural_condition}; that is, we are looking for an element $g \in S\smallsetminus H$, elements $h_l, h_k \in H\cap S$ and $\delta, \epsilon = \pm 1$ such that the cosets $h_{l}^{\delta}gH$,  $Hgh_{k}^{\epsilon}$ are empty. We do it in two stages: we first find $g'_1$, $h_k$, $\epsilon$ such that $Hg_1h_k^{\epsilon}$ is empty. Then we modify the element $g_1'$ to $g_1''$ by a Nielsen transformation, and find an $h_f$ and $\delta$ such that $h_l^{\delta}g_1''H$ is also empty.
    
    Taking $g_1$ satisfying the second condition of the lemma, then there exists $h_{i_{1}},\ldots, h_{i_{n}}$ (not necessarily distinct indices) and $\epsilon_{1},\ldots, \epsilon_{n} \in \set{-1, 1}$ such that $$Hg_{1}h_{i_{1}}^{\epsilon_{1}}\ldots h_{i_{n}}^{\epsilon_{n}}$$ does not contain any elements of $S$. We may assume $n$ is minimal. If $n=1$ we proceed to the next step. Suppose then that $n > 1$ so, by minimality we have $$Hg_{1}h_{i_{1}}^{\epsilon_{1}}\ldots h_{i_{n-1}}^{\epsilon_{n-1}} = Hg_{j} \ \ \ \ \textnormal{ for some } j \neq 1$$ And so now we make the pair of Nielsen transformations:
    $$ g_{1} \mapsto g_{1}h_{i_{1}}^{\epsilon_{1}}\ldots h_{i_{n-1}}^{\epsilon_{n-1}} =: g_{1}'$$
    $$ g_{j} \mapsto g_{j}h_{i_{n-1}}^{-\epsilon_{n-1}}\ldots h_{i_{1}}^{-\epsilon_{1}}=: g_{j}'$$ 
    labeling the new elements of $S$ as $g_{1}', g_{j}'$ respectively. To see pictorially what has happened (\Cref{fig:swap_rows}), the rows of the elements $g_{1}$, $g_{j}$ have been `swapped' and the columns have remained unchanged:
    
    \begin{figure}[h]
        \centering
        \begin{center}
        	\renewcommand{\w}{3}
        	\renewcommand{\h}{3}
        	\begin{tikzpicture}[line cap=round,line join=round,x=1.2cm,y=1.2cm]
        	\foreach \y in {0,1,...,\h}
        	{
        		\draw [line width = 1pt] (0,\y+0.5) -- (\w,\y+0.5);
        	}
        	\foreach \x in {0,1,...,\w}
        	{
        		\draw [line width = 1pt] (\x,0.5) -- (\x,\h+0.5);
        	}
        	\draw[->, line width = 2pt] (3.5,2) -- (4,2);
        	\renewcommand{\w}{3}
        	\renewcommand{\h}{3}
        	\foreach \y in {0,1,...,\h}
        	{
        		\draw [line width = 1pt] (4.5,\y+0.5) -- (\w + 4.5, \y+0.5);     
        	}
        	\foreach \x in {0,1,...,\w}
        	{
        		\draw [line width = 1pt] (\x + 4.5, 0.5) -- (\x + 4.5, \h + 0.5);
        	}
        	\draw (0.5,3) node {$g_1$};
        	\draw (1.5, 2) node {$g_j$};
        	\draw (5,2) node {$g_1'$};
        	\draw  (6,3) node {$g_j'$};
        	\end{tikzpicture}
        \end{center}
        \caption{Nielsen transformation preserving columns but swapping rows.}
        \label{fig:swap_rows}
    \end{figure}
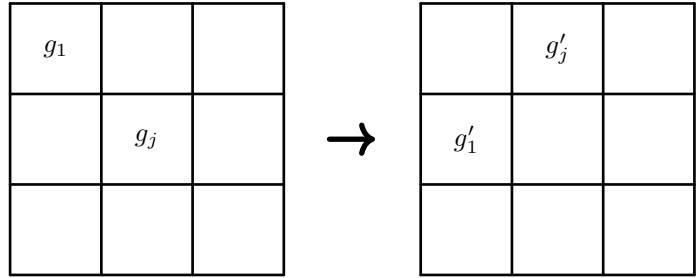

Note that this preserves left-right cleanliness, and the fact that the columns are unchanged is to say that the left cosets are unchanged and in particular the orbits of the left cosets are unchanged. So $g_{1}'H = g_{1}H$ has a sparse $\gen{S\cap H}$-orbit. As before this says that there are indices $j_{1},\ldots, j_{m}$ and signs $\delta_{1},\ldots,\delta_{m}$ such that
$$h_{j_{1}}^{\delta_{1}}\ldots h_{j_{m}}^{\delta_{m}}g_{1}'H$$
does not contain an element of $S$. We now repeat the previous argument but on the left. Suppose $m$ is minimal. If $m = 1$, we have found $h_{j_{1}}^{\delta_{1}}g_{1}'H$, $Hg_{1}'h_{i_{n}}^{\epsilon_{n}}$ both empty. Otherwise $m > 1$ and so
$$ h_{j_{2}}^{\delta_{2}}\ldots h_{j_{m}}^{\delta_{m}}g_{1}'H = g_{k}H$$
where $g_{k}$, $g_{1}'$ are distinct elements of the multiset. Then we make the following Nielsen transformations:
$$ g_{1}' \mapsto h_{j_{2}}^{\delta_{2}}\ldots h_{j_{m}}^{\delta_{m}}g_{1}' =:g_{1}''$$
$$ g_{k} \mapsto  h_{j_{m}}^{-\delta_{m}}\ldots h_{j_{2}}^{-\delta_{2}}g_{k}$$

Once again we can see that this doesn't alter rows, but swaps the columns. In particular it preserves left-right cleanliness.

Now, we know that the cosets $h_{j_1}^{\delta_1} g_1''H$ and $Hg_1''h_{i_n}^{\epsilon_n}$ are empty. If $j_1 \ne i_n$, then $h_{j_1}^{\delta_1}g_1''h_{i_n}^{\epsilon_n}$ lies in an empty row and empty column, so we can perform the desired extraction
$$h_{j_1} \mapsto h_{j_1}^{\delta_1}\cdot g_1'' \mapsto h_{j_1}^{\delta_1}g_1''\cdot h_{i_n}^{\epsilon_n}$$
% Denote the image of $g_{1}'$ by  $g_{1}''$, then letting $j_{1} = l $ , $\delta$ and $i_{k} = k$, $\epsilon_{k} = \epsilon$ we have that neither of:
%$$h_{l}^{\delta}g_{1}''H, \ \ \ Hg_{1}''h_{k}^{\epsilon}$$
%contain elements of $S$.
Thus, let's assume without loss of generality that $j_1 = i_n = 1$ and let's rename $\delta_1 =: \delta$, $\epsilon_n =: \epsilon$ and $g_1'' =: g_1$ (for simplicity).

Now the empty cosets are $h_{1}^{\delta}g_{1}H$ and $Hg_{1}h_{1}^{\epsilon}$. Take $h_2$ (we can do this since $|S\cap H| \ge 2$) and consider the element $h_{2}g_{1}h_{1}^{\epsilon}$. If this lies in a left coset representing an empty column, then we're done (since we extract $h_2 \mapsto h_2 \cdot g_1 \mapsto h_2g_1\cdot h_1^\epsilon$). Otherwise $h_{2}g_{1}H = g_{r}H$ for some $r$. This case requires two Nielsen transformations but the order depends on whether $r = 1$ or $r \neq 1$.

Both of the cases are illustrated with diagrams (Figures~\ref{fig:second_swap} and~\ref{fig:first_swap}), to help visualise the underlying process. The Nielsen transformations are represented with dashed arrows, whereas the elements in the chessboards which are not members of $S$ are underlined. We use this underlining notation from this point onwards in the rest of the paper.

If $r = 1$, we make the following Nielsen moves, as illustrated in \Cref{fig:second_swap}.
$$h_{2} \mapsto h_{2}\cdot g_{1} \mapsto h_{2}g_{1} \cdot h_{1}^{\epsilon}$$
$$g_{1} \mapsto h_{1}^{\delta} \cdot g_{1}$$
Here, since $Hg_1h_1^\epsilon$ is an empty row, it is one different from $Hg_1$. Thus $h_2g_1h_1^\epsilon$ actually lies in a row different to one that $g_1$ lies in. Similarly, $h_1^\delta g_1$ lies in a column different to one of $g_1$.

\begin{figure}[h]
    \centering
    \begin{center}
    	\renewcommand{\w}{3}
    	\renewcommand{\h}{3}
    	\begin{tikzpicture}[line cap=round,line join=round,x=1.2cm ,y=1.2cm]
    	\foreach \y in {0,1,...,\h}
    	{
    		\draw [line width = 1pt] (0,\y+0.5) -- (\w,\y+0.5);
    	}
    	\foreach \x in {0,1,...,\w}
    	{
    		\draw [line width = 1pt] (\x,0.5) -- (\x,\h+0.5);
    	}
    	\draw[->, line width = 2pt] (3.5,2) -- (4,2);
    	\renewcommand{\w}{3}
    	\renewcommand{\h}{3}
    	\foreach \y in {0,1,...,\h}
    	{
    		\draw [line width = 1pt] (4.5,\y+0.5) -- (\w + 4.5, \y+0.5);     
    	}
    	\foreach \x in {0,1,...,\w}
    	{
    		\draw [line width = 1pt] (\x + 4.5, 0.5) -- (\x + 4.5, \h + 0.5);
    	}
    	\draw[->, line width = 2pt] (8, 2) -- (8.5, 2);
    	\renewcommand{\w}{3}
    	\renewcommand{\h}{3}
    	\foreach \x in {0,1,...,\w}
    	{
    		\draw [line width = 1pt] (9 + \x, 0.5) -- (9 + \x, 0.5 + \h);
    	}
    	\foreach \y in {0,1,...,\h}
    	{
    		\draw [line width = 1pt] (9, \y+0.5) -- (9 + \w, \y  +0.5);
    	}
    	
    	\draw (0.5, 3) node {$g_1$};
    	\draw (0.5, 2) node {\underline{$h_2g_1h_1^\epsilon$}};
    	\draw (1.5, 3) node {\underline{$h_1^\delta g_1$}};
    	\draw (-0.5, 1) node {$h_2$};
    	\draw [dashed, ->, line width = 1.5pt, color = black] (-0.5+0.2,1) .. controls (-0.5+0.2+0.3,1) and (0.5,2-0.3-0.2) .. (0.5,2-0.2);
    	
    	\draw (5, 3) node {$g_1$};
    	\draw (5, 2) node {$h_2g_1h_1^\epsilon$};
    	\draw (6, 3) node {\underline{$h_1^\delta g_1$}};
    	\draw [dashed, ->, line width = 1.5pt, color = black] (5, 3+0.2) .. controls (5+0.3, 3+0.2+0.3) and (6-0.3-0.1, 3+0.2+0.3) .. (6-0.1, 3+0.2);
    	
    	\draw (5+4.5, 2) node {$h_2g_1h_1^\epsilon$};
    	\draw (6+4.5, 3) node {$h_1^\delta g_1$};
    	
    	\end{tikzpicture}
    \end{center}
    \caption{Nielsen moves yielding left-right extraction when $r = 1$.}
    \label{fig:second_swap}
\end{figure}
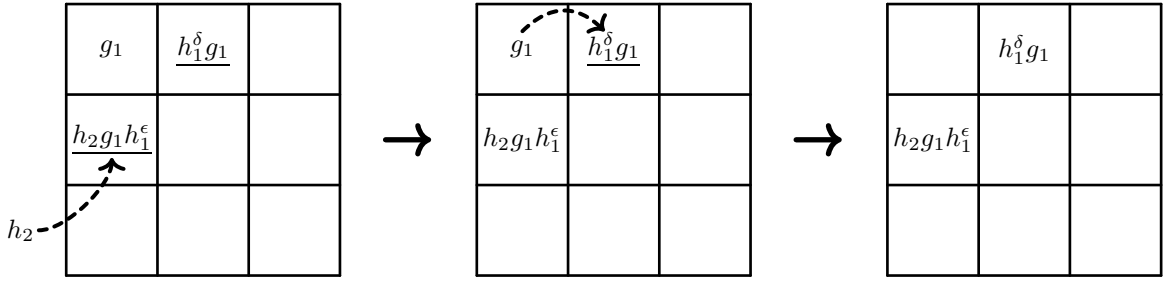

In the second case, $r \neq 1$, we put for simplicity $r=2$. We can make the pair of Nielsen transformations:
$$g_{2} \mapsto h_{2}^{-1} \cdot g_{2} \mapsto h_{1}^{\delta} \cdot h_{2}^{-1}g_{2}$$
$$h_{2} \mapsto h_{2}\cdot g_{1} \mapsto h_{2}g_{1} \cdot h_{1}^{\epsilon}$$
This is a left-right extraction, as shown pictorially in \Cref{fig:first_swap}. In this picture, $g_1$ and $g_2$ are diagonal, the row $Hg_1h_1^\epsilon$ is different to $Hg_1$ and $Hg_2$ because it is empty. Similarly, $h_1^\delta h_2^{-1}g_2H$ is different to $g_1H$ and $g_2H$ because $h_1^\delta g_1 h_1^\epsilon$ lies in an empty row and column.

\begin{figure}[h!]
    \centering
    \begin{center}
    	\renewcommand{\w}{3}
    	\renewcommand{\h}{3}
    	\begin{tikzpicture}[line cap=round,line join=round,x=1.3cm ,y=1.3cm]
    	\foreach \y in {0,1,...,\h}
    	{
    		\draw [line width = 1pt] (0,\y+0.5) -- (\w,\y+0.5);
    	}
    	\foreach \x in {0,1,...,\w}
    	{
    		\draw [line width = 1pt] (\x,0.5) -- (\x,\h+0.5);
    	}
    	\draw[->, line width = 2pt] (3.5,2) -- (4,2);
    	\renewcommand{\w}{3}
    	\renewcommand{\h}{3}
    	\foreach \y in {0,1,...,\h}
    	{
    		\draw [line width = 1pt] (4.5,\y+0.5) -- (\w + 4.5, \y+0.5);     
    	}
    	\foreach \x in {0,1,...,\w}
    	{
    		\draw [line width = 1pt] (\x + 4.5, 0.5) -- (\x + 4.5, \h + 0.5);
    	}
    	\draw[->, line width = 2pt] (8, 2) -- (8.5, 2);
    	\renewcommand{\w}{3}
    	\renewcommand{\h}{3}
    	\foreach \x in {0,1,...,\w}
    	{
    		\draw [line width = 1pt] (9 + \x, 0.5) -- (9 + \x, 0.5 + \h);
    	}
    	\foreach \y in {0,1,...,\h}
    	{
    		\draw [line width = 1pt] (9, \y+0.5) -- (9 + \w, \y  +0.5);
    	}
    	
    	\draw (0.5, 3) node {$g_1$};
    	\draw (1.5, 2) node {$g_2$};
    	\draw (1.5, 1) node {\underline{$h_2g_1h_1^\epsilon$}};
    	\draw (2.5, 2) node {\underline{$h_1^\delta h_2^{-1}g_2$}};
    	\draw (2.5, 1) node {\underline{$h_1^\delta g_1 h_1^\epsilon$}};
    	\draw [dashed, ->, line width = 1.5pt, color = black] (1.5,2+0.2) .. controls (1.5+0.3,2+0.3+0.2) and (2.5-0.3,2+0.3+0.2) .. (2.5,2+0.2);

    	\draw (5, 3) node {$g_1$};
    	\draw (7, 2) node {$h_1^\delta h_2^{-1}g_2$};
    	\draw (6, 1) node {\underline{$h_2g_1h_1^\epsilon$}};
    	\draw (4, 1) node {$h_2$};
    	\draw [dashed, ->, line width = 1.5pt, color = black] (4,1+0.2) .. controls (4+0.3,1+0.3+0.2) and (6-0.3,1+0.3+0.2) .. (6,1+0.2);
    	
    	\draw (9.5, 3) node {$g_1$};
    	\draw (11.5, 2) node {$h_1^\delta h_2^{-1}g_2$};
    	\draw (10.5, 1) node {$h_2g_1h_1^\epsilon$};
    	\end{tikzpicture}
    \end{center}
    \caption{Nielsen moves yielding left-right extraction when $r \neq 1$.}
    \label{fig:first_swap}
\end{figure}
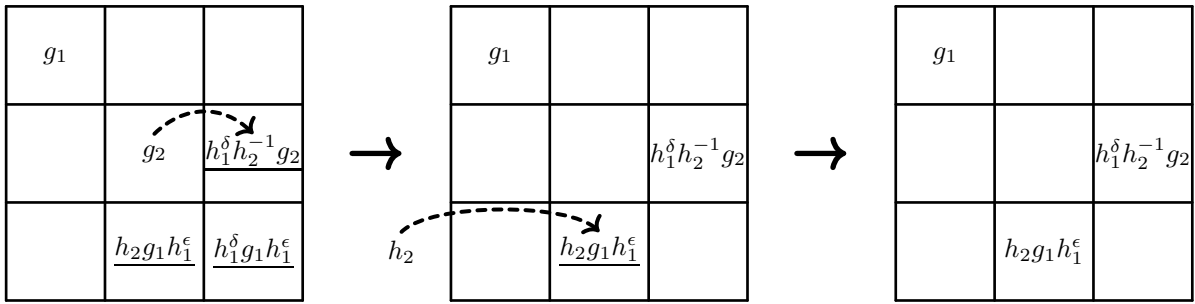

And so in any case, we have left-right extracted an element of $S\cap H$.
\end{proof}
  
This lemma will be very useful in shortening the arguments needed in the case-by-case analysis done in the proof of \Cref{thm: rank4case}, which we do in \Cref{sec:rank4}. More importantly it is the first extraction algorithm which can be applied to a general $H
\leq G$ and (no constraints on the index $[G:H]$), with only combinatorial hypotheses on $S$. For example, below is a neat situation where the hypotheses of \Cref{lem: EL} hold. 

\begin{cor}
    \label{cor: DPL}
    Let $G$, $H$ and $S$ be as in \Cref{lem: EL}. Suppose that there is an element $s \in S$ such that $s^{n}$ lies diagonal to every element of $S$ for some $n$. Then, if $|S \cap H|\geq 2$, a left-right extraction is possible. 
\end{cor}

\begin{proof}
    Suppose $h \in S \cap H$. Consider the element $hs^{n}$. This lies in $Hs^{n}$, which contains no other elements of $S$ by assumption. If $hs^{n}H$ contains no elements of $S$ then we may perform $n$ Nielsen moves to extract $h \mapsto hs^{n}$ and we're done. Otherwise, we have 
    $gH = hs^{n}H$ for some $g \in S\smallsetminus H$. We break this up into two cases.
    \\Case 1: If $s$, $g$ are distinct elements, then we may make the pair of Nielsen moves: $g \mapsto h^{-1}g$, followed by $h \mapsto hs^{n}$ (noting that $g \notin H$ so $g_1$, $h$ are distinct entries). This is the desired left-right extraction because we have moved $g$ to an empty row without changing its column, and then have extracted $h$ to the empty row $Hs^{n}$ and the (newly) empty column $gH$. 
    \\Case 2: $sH = hs^{n}H$. The previous sequence of Nielsen moves would no longer be valid. Instead we can note that the equation implies that $sH$ has a sparse orbit. Hence if we repeat the argument instead considering the right action, with $s^{n}h$ instead of $hs^{n}$, either we're done or we encounter the same barrier, i.e. $Hs = Hs^{n}h$. In this case both $sH$ and $Hs$ both have sparse orbits, and so we can extract by \Cref{lem: EL}.
\end{proof}

Before using this corollary to prove \Cref{thm: divisor_thm}, one further result of elementary group theory is needed. 

\begin{defn}
    For $G$ an arbitrary group and $H$ a finite index subgroup of $G$, the \emph{$H$-exponent} of $g \in G$ is the smallest positive integer $e$ such that $g^e \in H$.
\end{defn}

We will use $H$-exponents to enumerate cases of the proof in the next section. Below we will show that they are divisors of $[G:\core_G(H)]$, which will then allow us to prove \Cref{thm: divisor_thm}. The subgroup $\core_G(H)$ is defined as: 
\begin{defn}\label{defn:core_sg}
    If $G$ is a group with subgroup $H$, the \emph{core subgroup} of $H$ in $G$ is $$\core_G(H) := \bigcap\limits_{g \in G}gHg^{-1}$$
\end{defn}
The core is the largest normal subgroup of $G$ contained in $H$. If $H$ has index $n$, $\core_G(H)$ is the kernel of the homomorphism $G \rightarrow S_n$ induced by the left-multiplication action of $G$ on $H$. The first isomorphism theorem implies that $G/\core_G(H)$ is isomorphic to a subgroup of $S_n$, and in particular $\core_G(H)$ has finite index in $G$. This index is connected to $H$-exponents by the following proposition:

\begin{prop}
    The $H$-exponent of $g\in G$ divides $[G:\core_G(H)]$.
\end{prop}
\begin{proof}
    Let $m = [G:\core_G(H)]$, and $e$ be the $H$-exponent of $g$. If the proposition is false, then we may write $m = eq + r$ where $q$, $r$ are integers such that $1 \leq r < e $. But now, because $g^m, g^e \in H$, we have that: $$g^r = g^{m -eq} = g^m(g^e)^{-q} \in H$$ contradicting minimality of $e$. 
\end{proof}

We conclude with the proof \Cref{thm: divisor_thm}, which relies on everything mentioned in this section.

\vspace{\topsep}
\begin{proof}[Proof of \Cref{thm: divisor_thm}]
    Let $d$ be the minimal non-trivial divisor of $[G:\core_G(H)]$. Suppose that $S$ is a generating set of $G$ with size $r$. If $|S\cap H| = 1$ then $S$ is already diagonal, and we may extend it to a left-right transversal of $H$ in $G$. Otherwise, suppose $|S \cap H| \geq 2$ with $S$ left-right clean. It is sufficient to show that a left-right extraction is possible. 
    
    Choose $s \in S\smallsetminus H$, and let $e$ be its $H$-exponent. The previous proposition implies $e \geq d$. Minimality of $e$ implies that the set $A = \set{s^2, \ldots, s^{e-1}}$ is diagonal (every pair of elements lie in distinct rows and columns), and that $A \subset G\smallsetminus H$.
    
    Now suppose that the hypothesis of \Cref{cor: DPL} does not apply. That is, that every element of $A$ belongs to the same row or column of some element of $S \smallsetminus H$. Under this assumption, a chessboard which contains $k$ elements of $S$ can contain at most $2k$ elements of $A$ (counting over each row, and then each column). If we count the elements of $A$ per chessboard, this leads to the inequality $$e - 2 \leq 2|S\smallsetminus H|$$ And because $|S\cap H| \geq 2$, $$e - 2  \leq 2(r-2) $$ But $d$ is the smallest non-trivial divisor of $[G:\core_G(H)]$, so we deduce from the above that $d \leq 2r -2 $. The contrapositive assertion shows that if $d \geq 2r - 1$ we can indeed extract via \Cref{cor: DPL}, as required. 
\end{proof}

\section{Proof of \Cref{thm: rank4case}; transversals as generating sets for rank~4 groups}\label{sec:rank4}
 
In this section we prove \Cref{thm: rank4case}, by performing a case-by-case analysis of where elements in $S$ lie relative to chessboards of $H\leq G$. This is an extension of the main theorem of \cite{transgen} to groups of rank~4. The proof in \cite[Theorem 3.11]{transgen} for groups of rank at most $3$ uses the technique of shifting boxes, but also relies heavily on the fact that having just $3$ elements in the chessboards makes the case-by-case analysis of Nielsen moves tractable. Without further insight into shifting boxes as a general technique, attempting to mimic the arguments in \cite{transgen} for the case when the group has rank 4 quickly leads to case analysis which is too large to carry out. The results of \Cref{sec:shifting}  are enough to allow the argument style of \cite{transgen} to extend the proof technique of \cite[Theorem 3.11]{transgen} to the case where $G$ has rank 4.

In addition to the results of \Cref{sec:shifting}, we will also make use of the inversion map on chessboards.
\begin{defn}\label{defn:inversion}
     Consider the boards $HgH$ and $Hg^{-1}H$. The \emph{inversion map} is the standard map on group elements $x \mapsto x^{-1}$. This naturally induces:
     \begin{itemize}
         %\item A map on group elements as defined above,
         \item A map on chessboards via $HgH \mapsto Hg^{-1}H$,
         \item A map between the set of left cosets and the set of right cosets via $gH \mapsto Hg^{-1}$.
     \end{itemize}
If a board is mapped to itself via the inversion map, $HgH = Hg^{-1}H$, we say the board is \emph{self-inverse}.
\end{defn}

\begin{note}
    A chessboard is self-inverse if and only if it contains an element of $H$-exponent $2$. We make use of this in \Cref{sec:self-inverse-chessboards}.
\end{note}

As a map between left and right cosets of $H$, the inversion map commutes with the multiplication action of $G$. What we mean by this is that the diagram in \Cref{fig:diagram} commutes. This simple observation makes the inversion map a useful way to connect the orbits of left and right cosets, and hence we will make use of it to apply \Cref{lem: EL}.

\begin{figure}
    \centering
    \begin{tikzcd}[column sep = 4em, row sep = 4em]
        gH \rar[mapsto]{\textrm{act by } x} \dar[mapsto]{\textrm{inv.}} & xgH \dar[mapsto]{\textrm{inv.}} \\
        Hg^{-1} \rar[mapsto]{\textrm{act by } x} & Hg^{-1}x^{-1}
    \end{tikzcd}
    \caption{The inversion map commutes with the right action of $G$ on $G/H$ and $H\backslash G$.}
    \label{fig:diagram}
\end{figure}

Throughout the proof of \Cref{thm: rank4case}, the term \emph{box} will be used to refer to the intersection of a left and right coset in a double coset (forming a `box' in the corresponding chessboard). Also, we will always have a generating set $S$ of elements whose Nielsen transformations we are considering. Recall a coset or a box is referred to as \emph{empty} if it doesn't contain any elements of $S$. We are now ready to prove \Cref{thm: rank4case}.
%\begin{thmcustom}
%Let $G$ be a group of rank $4$ with a finite index subgroup $H$, with $[G:H] \geq 4$. Suppose $S$ is a set of 4 elements which generate $G$, then there exists a sequence of Nielsen transformations taking $S$ to a new set $\tilde{S}$, such that the elements of $\tilde{S}$ may be extended to a left-right transversal of $H$ in $G$. 
%\end{thmcustom}
\vspace{\topsep}
\begin{proof}[Proof of \Cref{thm: rank4case}]
 
Given the left-right cleaning procedure given in \Cref{sec:graphs}, we may suppose $S$ is left-right clean with respect to $H$. There are either 1,2 or 3 elements of $S$ remaining in $H$.

If there is only 1 element of $S$ in $H$, then $S$ is already diagonal and may be immediately extended to a left-right transversal.

If there are 3 elements of $S$ inside $H$, we may extract one as follows: Suppose $h_1, h_2, h_3$ are the elements of $S$ belonging to $H$, and $g$ is the entry not in $H$. Firstly, if $g$ has $H$-exponent (simply referred to as \emph{exponent} subsequently) at least $3$, then $g^{2}$ lies outside of $H$, and diagonally to $g$. By \Cref{cor: DPL}, we may extract one of the $h_i$ from $H$. Now suppose $g$ has exponent $2$. By the argument given in the proof of  \Cref{prop: left-clean}, there is an empty right coset of the form $Hs_is_j^{\epsilon}$ with $s_i$, $s_j$ elements of $S$. Since each $Hg^{\pm1} = Hg$ and each $h_i \in H$ we see the empty coset must be of the form $Hgh_i^{\epsilon} $ for some $i$, some $\epsilon \in \{\pm1\} $. This implies that the orbit of $Hg$ is sparse under the action of $\gen{h_1, h_2, h_3}$. 
Because $g^2 \in H$, the inversion map preserves the chessboard $HgH$ (that is, its columns are mapped to its rows and vice versa). Furthermore, because $g$ is the only entry of $S$ in the chessboard, and $gH \mapsto Hg$ under inversion, it follows that empty right cosets are mapped to empty left cosets under inversion. In particular $h_i^{-\epsilon}g^{-1}H = h_i^{-\epsilon}gH$ is empty, and so $g$ has a left and right sparse orbit, so we can extract by \Cref{lem: EL}. This process of extraction works for several more difficult configurations occurring later.

Hence we are left with the case where there are 2 elements of $S$ in $H$, and we can label $S$ by $S \cap H = \{h_1, h_2\}$ and $S \smallsetminus H  = \{g_1, g_2\}$. As we have seen above, the exponents of the elements outside of $H$ play an important role. Accordingly the remaining cases are indexed by the exponents of $g_1$ and $g_2$ ($e_1$ and $e_2$ respectively), where $e_{1}, e_{2}>1$ as $g_{1}, g_{2}\notin H$. Reordering so that $e_1 \leq e_2$, the cases are:
\begin{itemize}
      \item[I.] $e_2 \geq 4$
      \item[II.] $e_1 = 2 < e_2$
      \item[III.] $e_1 = e_2 = 3$
      \item[IV.] $e_1 = e_2 = 2$
\end{itemize}

In diagrams, $g_1$ and $g_2$ will be depicted as belonging to the same chessboard. Whilst this is not necessarily the case, generally the arguments made apply equally well in both cases and so will not be made twice.

\subsection*{Case I.}

As $e_2 \geq 4$, we have that $\set{g_2, g_2^{2}, g_2^{3}}$ is a diagonal set lying outside of $H$. If either of $g_2^{2}$ or $g_2^{3}$ lies diagonally to $g_1$, then it necessarily lies diagonally to every entry of $S$. We may therefore extract by \Cref{cor: DPL}. Hence we need only consider the case when neither $g_2^{2}, g_2^{3}$ lies diagonally to $g_1$.  That is, they both lie in $g_1H \cup Hg_1$.

The powers of $g_{2}$ are mutually diagonal so in this case neither power can lie in $g_1H \cap Hg_1$. That is, exactly one of $g_2^{2}$, $g_2^{3}$ lies in the row of $g_1$ (and not the column), and the other lies in the the column (and not the row). Suppose $g_2^{i}H = g_1H$, then perform $g_1 \mapsto g_2^{-i}g_1 = h_3 \in H$. Now, within the chessboard $Hg_1H$ there is the diagonal set $\set{g_2^{2}, g_2^{3}}$ and potentially also $g_2$.

We have $3$ elements in $H$, and to achieve our goal we must extract $2$ of them from $H$. To do this we will extract one element to such a position that we still have a diagonal power of $g_2$, i.e. a power of $g_2$ is diagonal to everything in the Nielsen transformed multiset $S$. Then we can apply \Cref{cor: DPL} to extract the other.

Consider the elements $h_ig_2^{2}$ for $i = 1,2,3$. These each lie in $Hg_2^{2}$, and if one lies outside of the columns $g_2^{3}H$, $g_2H$ then we can extract it, noting that $g_2^{3}$ remains diagonal to the arrangement. If none of the $h_ig^{2}$ lie outside of the two columns $g_2^{3}H$, $g_2H$, then by the pigeonhole principle, without loss of generality $h_1g_2^{2}$, $h_2g_2^{2}$ lie in the same columns. Hence $h_1^{-1}h_2g_2^{2} \in g_2^{2}H$ and so $h_1^{-1}h_2g_2^{2} \in g_2^{2}H \cap Hg_2^{2}$. After extracting $h_2 \mapsto h_1^{-1}h_2g_2^{2}$, $g_2^{3}$ remains diagonal to the entire configuration, so we may extract again by \Cref{cor: DPL}. \Cref{fig:rank4_zeroth_fig} shows this extraction, assuming that $g_2^2 \in Hg_1$ and $g_2^2 \in g_1H$ (transposing $g_2^2$ and $g_2^3$ doesn't alter the argument).

\begin{figure}[h!]
    \centering
    \begin{center}
    	\renewcommand{\w}{3}
    	\renewcommand{\h}{3}
    	\begin{tikzpicture}[line cap=round,line join=round,x=1.4cm,y=1.4cm]
    	\foreach \x in {0,1,...,\w}
    	{
    		\draw [line width = 1pt] (\x-0.5, 0-0.5) -- (\x-0.5, 3-0.5);
    	}
    	\foreach \y in {0,1,...,\h} 
    	{
    		\draw [line width = 1pt] (0-0.5, \y-0.5) -- (3-0.5,\y-0.5);
    	}
    	
    	\draw (0, 0) node {\underline{$g_2^3$}};
    	\draw (0, 2) node {\underline{$g_1$}};
    	\draw (1, 1) node {$g_2$};
    	\draw (2, 2.2) node {\underline{$g_2^2$}};
    	\draw (2, 1.8) node {\underline{$h_1^{-1}h_2g_2^2$}};
    	\draw (-1, 1) node {$h_2$};
    	
    	\draw [->, dashed, line width = 1.5pt, color = black] (-0.8, 1) .. controls (-0.8+0.6, 1) and (1.5-0.4, 1.8) .. (1.5, 1.8);
    	
    	\end{tikzpicture}
    \end{center}
    \caption{The extraction of $h_2$ to $h_1^{-1}h_2g_2^{2}$. Note $g_2^{3}$ is diagonal to $\set{g_2, h_1^{-1}h_2g_2^2}$}
    \label{fig:rank4_zeroth_fig}
\end{figure}
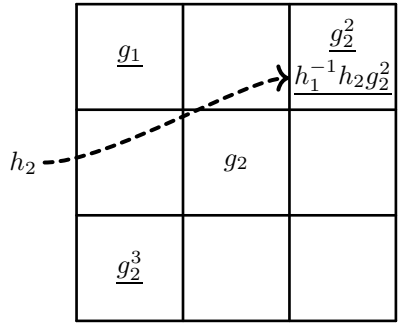

\subsection*{Case II.}
  
In this case $e_1 = 2$ and so the inversion map is useful. Because $e_2 \neq 2$, we have that $1 \not \equiv -1 \mod{e_2}$ and hence $\set{g_2, g_2^{-1}}$ is a diagonal set. The inversion map gives a bijection between left and right cosets of $H$, and because $Hg_1 \neq Hg_2$ we have $g_2^{-1}H \neq g_1^{-1}H = g_1H$ and similarly on the right. Hence $\set{g_1, g_2, g_2^{-1}}$ is diagonal, and so we may extract by \Cref{cor: DPL}.

\subsection*{Case III.}

We consider the set of inverses $\set{g_1^{-1}, g_2^{-1}}$. From remarks on the inversion map, this new pair of elements is diagonal. Furthermore, neither of the $g_i$ have exponent $2$, so $\set{g_i, g_i^{-1}}$ is diagonal. It follows that if $g_1^{-1}$ lies diagonally to $g_2$, or vice-versa, then we may extract by \Cref{cor: DPL}. Therefore, we may suppose that neither inverse lies diagonally to $\set{g_1, g_2}$ and so $g_1^{-1} \in g_2H \cup Hg_2$, and  $g_2^{-1} \in g_1H \cup Hg_1$. 

Within this configuration, we firstly consider the special case where $g_1^{-1} \in g_2H \cap Hg_2$. This says that the inverse of $g_i$ lies in the exact chessboard box of the other, so that $Hg_2^{-1} = Hg_1$ and $g_2^{-1}H = g_1H$ (by inverting we can see that the situation is symmetric with respect to $g_1$, $g_2$). Consider an empty right coset of the form $Hs_is_j^{\epsilon}$ where $\epsilon \in \{\pm 1\}$. In this subcase we have $$ Hg_ig_j^{\epsilon} = Hg_j^{\epsilon - 1} \in \set{H , Hg_1, Hg_2}$$ In particular the empty coset $Hs_is_j^{\epsilon}$ must be of the form $Hg_ih_j^{\epsilon}$. By symmetry, we may assume $Hg_1h_1$ is empty. Applying the same argument to the left shows that there is an empty left coset of the form $h_k^{\delta}g_jH$ for $\delta \in \{ \pm 1 \}$.

If $j = 1$, then $g_1$ has a left and right sparse orbit (under the $\gen{h_1, h_2}$ action) so we're done by \Cref{lem: EL}. So suppose $j = 2$. Hence $Hg_1h_1$, $h_k^{\epsilon}g_2H$ are the pair of empty cosets. This says that $g_1$ has a right-sparse orbit and $g_2$ has a left-sparse orbit.

If we are able to deduce that $g_1$, $g_2$ have the same left or right coset orbit (under the $\gen{h_1, h_2}$ action) then we're done, as this will imply that one of $\set{g_1, g_2}$ has a left and right sparse orbit.

To show this, consider the element $ x = g_1h_k^{\epsilon}g_2$. Immediately one can see that $x \notin g_1H \cup Hg_2$. Furthermore $x \notin H$, because otherwise we would have $h_k^{\epsilon}g_2H = g_1^{-1}H = g_2H$ which contradicts the emptiness of $h_k^{\epsilon}g_2H$. If $x$ lies diagonally to $\set{g_1, g_2}$, then we may extract via $h_k \mapsto g_1h_k^{\epsilon}g_2$. Else we have $x \in Hg_1$ or $x \in g_2H$ (considering the cases we have just eliminated). Examining both of these cases: $$ x \in Hg_1 \iff Hg_1h_k^{\epsilon}g_2 = Hg_1 \implies Hg_1h_k^{\epsilon} = Hg_1g_2^{-1} = Hg_2^{-2} $$
$$x \in g_2H \iff g_1h_k^{\epsilon}g_2H = g_2H \implies h_k^{\epsilon}g_2H = g_1^{-1}g_2H = g_1^{-2}H $$ Recalling that $g_1$ and $g_2$ have exponent $3$, the two equations at the end of each line may be written as: $$ Hg_1h_k^{\epsilon} = Hg_2$$ $$ h_k^{\epsilon}g_2H = g_1H$$
Either case implies that $g_1,g_2$ have the same $\gen{h_1,h_2}$-orbit as desired.
    
The above argument deals with the special case where neither inverse $g_1^{-1}$, $g_2^{-1}$ was diagonal to $S$, and the inverse of $g_2$ lies in $g_1H \cap Hg_1$. It remains to consider the case where $g_2^{-1} \in g_1H \Delta Hg_1$, where $\Delta$ denotes symmetric difference of sets. This implies that $g_1^{-1} \in g_2H \Delta Hg_2$. In particular: $$g_i^{-1} \in Hg_j \iff g_j^{-1} \in g_iH $$ Hence, without loss of generality we suppose $g_1^{-1} \in Hg_2 \smallsetminus g_2H$ and so $g_2^{-1} \in g_1H \smallsetminus Hg_1$; this is illustrated in \Cref{fig:rank4_first_fig}. This figure shows two conventions which we adopt from here on: by a suitable permutation of the rows and columns of the chessboards, the inversion map corresponds to reflection in the main diagonal (this clarifies arguments in \Cref{sec:self-inverse-chessboards}), and that the elements outside of $S$ (prior to any of the transformations which might also be drawn) are underlined.

\begin{figure}[h]
    \centering
    \begin{center}
    	\renewcommand{\w}{3}
    	\renewcommand{\h}{3}
    	\begin{tikzpicture}[line cap=round,line join=round,x=1.2cm,y=1.2cm]
    	\foreach \x in {0,1,...,\w}
    	{
    		\draw [line width = 1pt] (\x, 0) -- (\x, 3);
    	}
    	\foreach \y in {0,1,...,\h} 
    	{
    		\draw [line width = 1pt] (0, \y) -- (3,\y);
    	}
    	
    	\draw (0.5, 1.5) node {$g_1$};
    	\draw (0.5, 0.5) node {\underline{$g_2^{-1}$}};
    	\draw (1.5, 2.5) node {\underline{$g_1^{-1}$}};
    	\draw (2.5, 2.5) node {$g_2$};
    	\end{tikzpicture}
    \end{center}
    \caption{The configuration being considered so far (restricting to the appropriate $3\times 3$ minor after permuting rows and columns appropriately).}
    \label{fig:rank4_first_fig}
\end{figure}

The main aim of the next calculation is to reduce this subcase to the situation dealt with earlier in Case III, where  $g_1^{-1} \in g_2H \cap Hg_2$. Consider the element $y = g_2g_1^{-1}$. We note that $y \in Hg_1$, and $y \notin g_2H$. If $ y \notin g_1H$ then consider $g_1 \mapsto y$ (which is a Nielsen move). The set $\set{g_2, y}$ is still diagonal, but now $g_2^{-1}$ is diagonal to $\set{y, g_2}$, and so we're done by \Cref{cor: DPL}. This is shown in \Cref{fig:rank4_second_fig} for the case when $g_1$, $g_2$ belong to the same chessboard:

\begin{figure}[h]
    \centering
    \begin{center}
    	\renewcommand{\w}{4}
    	\renewcommand{\h}{4}
    	\begin{tikzpicture}[line cap=round,line join=round,x=1.2cm,y=1.2cm]
    	\foreach \x in {0,1,...,\w}
    	{
    		\draw [line width = 1pt] (\x, 0) -- (\x, 4);
    	}
    	\foreach \y in {0,1,...,\h} 
    	{
    		\draw [line width = 1pt] (0, \y) -- (4,\y);
    	}
    	
    	\draw (0.5, 2.5) node {$g_1$};
    	\draw (1.5, 3.5) node {\underline{$g_1^{-1}$}};
    	\draw (0.5, 1.5) node {\underline{$g_2^{-1}$}};
    	\draw (2.5, 3.5) node {$g_2$};
    	\draw (3.5, 2.5) node {\underline{$y$}};
    	
    	\draw [->, dashed, line width = 1.5pt, color = black] (0.5+0.2, 2.5+0.1) .. controls (0.5+0.2+0.3, 2.5+0.1+0.2) and (3.5-0.2-0.3, 2.5+0.1+0.2) .. (3.5-0.2, 2.5+0.1);
    	\end{tikzpicture}
    \end{center}
    \caption{Example configuration of various elements, where $y = g_2g_1^{-1}$. Note however that $y$ may be in the same column as $g_1^{-1}$.}
    \label{fig:rank4_second_fig}
\end{figure}

If, instead, $y$ belongs to $g_1H$ then $y \in g_1H\cap Hg_1$, and so $y^{-1} \in g_1^{-1}H \cap Hg_1^{-1}$. Making the Nielsen move $g_2 \mapsto y^{-1}$, we see that $\set{g_1, y^{-1}}$ is diagonal, and also $g_1^{-1} \in yH\cap Hy$, $y \in g_1H \cap Hg_1$ (in particular $y^{-1}$ doesn't have exponent $2$). Either $y^{-1}$ has exponent greater than $3$, in which case we can extract by appealing to Case I, or $y^{-1}$ has exponent $3$, and we are done by the exact argument used in the first paragraph of Case III in which the inversion map swaps the boxes of the two elements of $S$ lying outside of $H$. We illustrate this in \Cref{fig:rank4_third_fig}.

\begin{figure}[h]
    \centering
    \begin{center}
    	\renewcommand{\w}{3}
    	\renewcommand{\h}{3}
    	\begin{tikzpicture}[line cap=round,line join=round,x=1.3cm,y=1.3cm]
    	\foreach \x in {0,1,...,\w}
    	{
    		\draw [line width = 1pt] (\x-0.5, 0-0.5) -- (\x-0.5, 3-0.5);
    	}
    	\foreach \y in {0,1,...,\h} 
    	{
    		\draw [line width = 1pt] (0-0.5, \y-0.5) -- (3-0.5,\y-0.5);
    	}
    	
    	\draw (0, 1) node {$g_1$, \underline{$y^{-1}$}};
    	\draw (0, 0) node {\underline{$g_2^{-1}$}};
    	\draw (1, 2) node {\underline{$y$}, \underline{$g_1^{-1}$}};
    	\draw (2, 2) node {$g_2$};
    	
    	\draw [->, dashed, line width = 1.5pt, color = black] (2-0.1, 2+0.15) .. controls (2-0.1-0.2, 2+0.2+0.2) and (0.8+0.2, 2+0.2+0.2) .. (0.8, 2+0.1);
    	\end{tikzpicture}
    \end{center}
    \caption{Configuration which reduces to the second subcase of Case III.}
    \label{fig:rank4_third_fig}
\end{figure}

\subsection*{Case IV.}
Observe that an element $g_i$ of exponent $2$ satisfies $g_iH = g_i^{-1}H$, $Hg_i = Hg_i^{-1}$. From this it follows that having a one-sided sparse orbit is sufficient to have a left-right sparse orbit. For example if the orbit of $Hg_1$ is sparse, then $Hg_1x$ contains no entry of $S$ for some $x \in \gen{h_1, h_2}$. Then if we consider the image of this right coset under the inversion map $x^{-1}g_1^{-1}H = x^{-1}g_1H$, and note that the set of full cosets (those containing at least one element of $S$) is preserved by inversion, it follows that $x^{-1}g_1H$ is empty and hence $g_1$ has a right-and-left sparse orbit. This is shown in \Cref{fig:rank4_fourth_fig}. In light of this observation, we may suppose that no cosets of $g_1$, $g_2$ have sparse orbits (or else we may extract by \Cref{lem: EL}).

\begin{figure}[h]
    \centering
    \begin{center}
    	\renewcommand{\w}{3}
    	\renewcommand{\h}{3}
    	\begin{tikzpicture}[line cap=round,line join=round,x=1.3cm,y=1.3cm]
    	\foreach \x in {0,1,...,\w}
    	{
    		\draw [line width = 1pt] (\x-0.5, 0-0.5) -- (\x-0.5, 3-0.5);
    	}
    	\foreach \y in {0,1,...,\h} 
    	{
    		\draw [line width = 1pt] (0-0.5, \y-0.5) -- (3-0.5,\y-0.5);
    	}
    	
    	\draw (0, 0) node {\underline{$g_1x$}};
    	\draw (0, 2) node {$g_1$, \underline{$g_1^{-1}$}};
    	\draw (1, 1) node {$g_2$, \underline{$g_2^{-1}$}};
    	\draw (2, 2) node {\underline{$x^{-1}g_1$}};
    	\end{tikzpicture}
    \end{center}
    \caption{The empty row $Hg_1x$ is mapped to the empty column $x^{-1}g_1H$ under inversion.}
    \label{fig:rank4_fourth_fig}
\end{figure}

As always, there exists an empty right coset of the form $Hs_is_j^{\epsilon}$, where $s_i$, $s_j \in S$. It must be the case that $s_i \notin H$ (because $Hg_i^{\pm1} = Hg_i$ is clearly non-empty). Furthermore, because we have assumed to have no sparse orbits, it must also be the case that $s_j \notin \set{h_1, h_2}$. Hence there is an empty right coset of the form $Hg_ig_j^{\epsilon}$ (where $i$, $j$ are necessarily distinct). By symmetry between the $g_i$s and their inverses, we may suppose that $\epsilon = 1$, so that $Hg_1g_2$ is empty. 

By assumption, sparse orbits do not exist, so $Hg_1g_2h_1$ is also empty. If $g_1g_2H$ is also empty, then $g_1g_2h_1$ lies diagonally to $S$ and hence the Nielsen transformation $h_1 \mapsto g_1g_2h_1$ is an extraction. If instead, $g_1g_2H$ contains an entry of $S$ then it must be the case that $g_2H = g_1g_2H$. This is because $H \neq Hg_1g_2 $ implies $H \neq g_1g_2H$, and $g_2H \neq H$ implies $g_1g_2H \neq g_1H$. The configuration is shown in \Cref{fig:rank4_fifth_fig}. From the equations $g_1g_2H = g_2H = g_2^{-1}H$ it follows that $g_2g_1g_2 \in H$ and so $g_1(g_2g_1g_2) = (g_1g_2)^{2} \notin H$. It can be seen from \Cref{fig:rank4_fifth_fig}, the Nielsen transformation $\set{g_1, g_2} \mapsto \set{g_1, g_1g_2}$ preserves diagonality and increases an exponent: $(g_1g_2)^{2} \notin H$. Hence our ability to extract from this configuration is covered by one of the cases I, II or III.

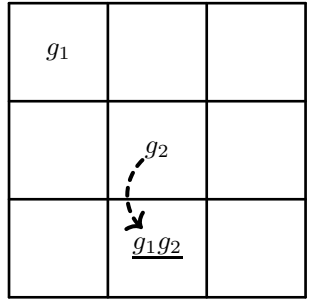
\begin{figure}[h]
    \centering
    \begin{center}
    	\renewcommand{\w}{3}
    	\renewcommand{\h}{3}
    	\begin{tikzpicture}[line cap=round,line join=round,x=1.3cm,y=1.3cm]
    	\foreach \x in {0,1,...,\w}
    	{
    		\draw [line width = 1pt] (\x-0.5, 0-0.5) -- (\x-0.5, 3-0.5);
    	}
    	\foreach \y in {0,1,...,\h} 
    	{
    		\draw [line width = 1pt] (0-0.5, \y-0.5) -- (3-0.5,\y-0.5);
    	}
    	
    	\draw (0,2) node {$g_1$};
    	\draw (1,1) node {$g_2$};
    	\draw (1,0) node {\underline{$g_1g_2$}};
    	
    	\draw [->, dashed, line width = 1.5pt, color = black] (1-0.15,1-0.1) .. controls (1-0.15-0.2,1-0.1-0.2) and (1-0.15-0.2,0+0.2+0.2) .. (1-0.15,0+0.2);
    	\end{tikzpicture}
    \end{center}
    \caption{The Nielsen move $g_2 \mapsto g_1g_2$ takes $\set{g_1, g_2}$ to the diagonal set $\set{g_1, g_1g_2}$.}
    \label{fig:rank4_fifth_fig}
\end{figure}

\end{proof}

\section{Results in the finite setting}\label{sec:finite-setting}

On closer inspection of \Cref{thm: rank4case} and \Cref{thm: divisor_thm}, we see that in each instance we are proving special cases of a slightly stronger question than \Cref{ques: motivating}. That is, it seems that the question we have been approaching so far is:

\begin{question}
\label{main Q}
  Let $S$ be a generating set of minimal size in a finitely generated group $G$, and $H$ a finite index subgroup with $[G:H] \geq |S|.$ Is $S$ Nielsen equivalent to a set which extends to a left-right transversal of $H$ in $G$?
\end{question}

In fact, by examining our statements and proofs of \Cref{thm: rank4case} and \Cref{thm: divisor_thm}, the question we are making progress on is even tighter. We do not need to work with a generating set of minimal size; our results allow us to Nielsen transform any generating set of the group, provided that set has size no greater than $[G:H]$.  That is:

\begin{question}
\label{general Q}
  Let $S$ be a finite multiset generating a group $G$, and $H$ a finite index subgroup with $[G:H] \geq |S|$. Is $S$ Nielsen equivalent to a set which extends to a left-right transversal of $H$ in $G$?
\end{question}

We note that \Cref{main Q} was raised in \cite{transgen}, where it was also noted that an affirmative answer for general $G$ would be implied by an affirmative answer for the case when $G$ is a free group of finite rank. In this section we make a reduction in the opposite direction, to the case where $G$ is a finite group, and use some of our techniques to make progress in this setting.

As noted following \Cref{defn:core_sg}, if $H$ is a finite index subgroup of $G$, then so is $N := \core_G(H)$. We provide the following two elementary lemmata without proof.

\begin{lem}\label{lem:lem1_np}
    Let $H\leq G$, and let $N:= \core_G(H)$. Because $N \leq H$, there is a natural bijection between left (right) cosets of $H/N$ in $G/N$, and left (right) cosets of $H$ in $G$. It follows that two elements of $G/N$ belong to the same left (right) coset of $H/N$ if and only if any two representatives of the elements lie in the same left (right) cosets of $H$ in $G$.
\end{lem}

\begin{lem}\label{lem:lem2_np}
     Let $N\trianglelefteq G$. Given a set $S$ in $G/N$, and a corresponding set of representatives $\tilde{S}\subset G$, any Nielsen transformation $S \rightarrow S_1 $ has a naturally corresponding Nielsen transformation $\tilde{S} \rightarrow \tilde{S_1}$, which commutes with the natural projection map $G \rightarrow G/N$. Explicitly, a Nielsen move on $S$ can be replicated on $\tilde{S}$ by performing the same move on the corresponding representatives.
\end{lem}

The reason that \Cref{lem:lem1_np} and \Cref{lem:lem2_np} do not allow us to reduce the question to the case where $G$ is finite is because the homomorphic image of minimal generating set is not necessarily a \emph{minimal} generating set. To get around this, we note that at no point in the previous sections do the arguments rely on the minimality of the generating set. That is, our results thus far make just as much progress towards the slightly more general \Cref{general Q}.

Clearly an affirmative answer to \Cref{general Q} would imply the same for \Cref{main Q}. It is enough to answer \Cref{general Q} in the case where $G$ is finite.

\begin{prop}
\label{finite reduction}
    \Cref{general Q} has an affirmative answer whenever $G$ is finite, if and only if it has an affirmative answer for all $G$. 
\end{prop}

\begin{proof}
    Suppose that the answer to \Cref{general Q} is `yes' whenever $G$ is finite. Let $G$ be arbitrary, $H$ a finite index subgroup and $S$ a generating multiset in $G$.
    
    Letting $\overline{G}$ denote quotient by $\core_G(H)$, $\overline{G}$ is a finite group containing the multiset $\overline{S}$ which generates $\overline{G}$. By assumption there is a Nielsen transformation $\overline{S} \mapsto \overline{S}_1$, such that $\overline{S}_1$ is a set of elements which each lie in distinct left and right cosets of $\overline{H}$. By \Cref{lem:lem2_np}, we may perform a corresponding Nielsen transformation $S \mapsto S_1$. Because this correspondence commutes with the map $G \rightarrow \overline{G}$, it follows from \Cref{lem:lem1_np} that the elements of $S_1$ each lie in distinct left and right cosets of $G$.   
\end{proof}

Thus, while we cannot reduce our motivating question (\Cref{main Q}) to the finite case, we can reduce the question that we would like to answer (\Cref{general Q}) to the finite case. With this reduction in mind, we now present two results for the case where $G$ is finite which give affirmative answers to \Cref{main Q} when $H$ is cyclic or the product of two prime-order cyclic groups. The first argument is completely elementary whilst the second relies on \Cref{lem: EL}.

\begin{lem}
	\label{lem:cyclic-nielsen}
	Suppose $H$ is a finite cyclic group of order $n$ and $\set{x,y} \subset H$. Then $\set{x,y}$ is Nielsen equivalent to a set containing the identity, $e$.
\end{lem}

\begin{proof}
    Assume that $H = \gen{g}$ and that $x = g^a, y = g^b$ for $a,b \in \set{0,1,\ldots, n-1}$ and $a\geq b$. Then we can perform Euclid's algorithm: if $b \ne 0$ take $q \in \mathbb{N}$ such that $|a-qb|< b$ and do the sequence of Nielsen moves:
    $$\set{x,y} \mapsto \set{xy^{-1},y} \mapsto \ldots \mapsto \set{xy^{-q},y}$$
    Then, we can change the roles of $x$ and $y$ with $y$ and $xy^{-q}$, and proceed with the algorithm.
	
	The minimal positive exponent of $g$ appearing in the set strictly decreases on each iteration, and so this process terminates at $\set{g^{\gcd(a,b)},g^0} = \set{z,e}$.
\end{proof}
We apply this lemma to answer \Cref{general Q} when $H$ is cyclic.

\begin{thm}
	Let $G$ be a finite group, and $H\leq G$ be a cyclic subgroup. Any generating multiset $S$ such that $|S|\leq [G:H]$ can be Nielsen transformed into a set contained in a left-right transversal of $H$ in $G$.
\end{thm}

\begin{proof}
	By left-right cleaning, we may suppose that $S$ is a disjoint union of: 1) a multiset whose elements are in $H$, and 2) a left-right diagonal set of elements outside of $H$.
	
	If $|S\cap H| = 1$, then the set is already diagonal with respect to $H$. Hence we may assume that there are at least two elements, $x, y \in S\cap H$. According to \Cref{lem:cyclic-nielsen}, we can Nielsen transform $\set{x,y}$ to a pair $\set{z,e}$ for some $z\in H$. Therefore we can assume $e \in S\cap H$.
	
	By assumption, $S$ generates $G$, so every element of $G$ can be written as a product of elements of $S$. Pick an element $g \in G$ that lies diagonally to $S$ (such an element exists by the index condition). $g$ can be written as a product of non-trivial elements of $S$, say $s_{i_1}s_{i_2}\ldots s_{i_n}$ Thus, the series
	$$e \mapsto e\cdot s_{i_1} \mapsto s_{i_1}\cdot s_{i_2} \mapsto \ldots \mapsto s_{i_1}s_{i_2}\ldots s_{i_{n-1}}\cdot s_{i_n}$$
	is a series of valid Nielsen moves. This sequence of Nielsen moves amounts to a left-right extraction of an element of $S\cap H$. We may repeat this until $|S \cap H| = 1$.
\end{proof}

It is immediate from the reduction method of \Cref{finite reduction} that the corresponding result which doesn't require $|G| < \infty$ is:
\begin{cor}
Let $G$ be a group with finite generating set $S$, and $H$ a finite index subgroup such that $[G:H] \geq |S|$. If $H/\core_G(H) \cong C_n$ for some $n \geq 1$, then $S$ is Nielsen equivalent to a set contained in a left-right transversal of $H$ in $G$.
\end{cor}
A similar corollary follows from the next proposition.

\begin{prop}\label{lem:CpxCp}
Let $G$ be a finite group, and $H \leq G$ such that $H \cong C_p \times C_p$ for some prime $p$. Suppose $S$ is a generating multiset in $G$ such that $[G : H] \geq |S|$. Then $S$ is a Nielsen-equivalent to a set contained in a left-right transversal of $H$ in $G$. 
\end{prop}

\begin{proof}
 
We use the following property of $C_p \times C_p$: any two non-trivial elements either generate the whole group or are powers of each other. To see this, take $x, y \in C_p\times C_p$ and suppose they are not both trivial. If they don't generate the entire group, they belong to a proper non-trivial subgroup. Such a subgroup is necessarily $C_p$ and so any non-trivial element (in particular one of $x$ and $y$) generates it.

As usual we may suppose $S$ is left-right clean. If $|S\cap H| = 1 $, we're done. Else, suppose we have a distinct pair $x, y \in S \cap H$. Suppose $\gen{x, y} \neq H$. Then by the first remark we (w.l.o.g) have $x = y^k$ for some $k \geq 0$, and we can perform a sequence of Nielsen moves resulting in $x \mapsto x y^{-k} = e$. As in the proof of \Cref{lem:cyclic-nielsen}, we can left-right extract $e$ from $S\cap H$. Hence we may suppose $\gen{x, y} = H$. It follows that $\gen{S\cap H}$ acts transitively on the rows and columns of any chessboard. 

Because $[G: H] \geq |S|$, there is a chessboard with a row and column not containing any element of $S$. Furthermore, we may suppose that there is such a chessboard which also contains an element of $S$ (otherwise, we may easily extract to the empty row of the form $Hs_is_j^{\epsilon}$ as in \Cref{prop: left-clean}, which would also be a left-right extraction). It follows that the element of $S$ in this chessboard has a sparse $\gen{S \cap H}$ orbit on the left and the right, and so we may perform a left-right extraction by \Cref{lem: EL}.

\end{proof}

\section{Further techniques and constructions}\label{sec: configurations}

Motivated by the method of shifting boxes,  we introduce some new techniques along a similar vein which help us resolve more special cases of \Cref{main Q} and \Cref{general Q}. In contrast to the extraction methods, which are useful in the situations of having relatively few elements of the multiset, these new techniques work well when the elements in chessboards are `packed densely'.

This section concerns a technique called an \emph{L-spin} of transforming particular configurations with a specific Nielsen transformation. The technique rests heavily on the fact that the inversion map induces a canonical bijection between the left and right cosets of a given subgroup as described in \Cref{defn:inversion}.

In this section we will always take $G$ to be an arbitrary group and $H\leq G$ a finite index subgroup. We will consider subsets of the chessboards of $H$ in $G$, given as follows:

\begin{defn}\label{defn:configuration}
    Given $H\leq G$ and a multiset $S$ of elements of the group $G$, we define a \emph{configuration} to be a table $T$ representing intersections of cosets $a_iH\cap Hb_j$ for some sets of cosets $\{a_1H,\ldots, a_lH\}$, $\{Hb_1, \ldots, Hb_k\}$, and a submultiset $R\subset S$, such that all elements of $R$ belong to the coset intersections described by the table $T$. We represent the elements by either writing them in the corresponding boxes in the table, or putting dots representing them in the boxes. Sometimes we also denote some arbitrary elements of $G$, not necessarily belonging to the multiset $R$, by writing them in the corresponding boxes underlined or represented by hollow dots.% \todo{is that the right word?}. YES IT IS.
\end{defn}

We can think of a configuration as ``choosing some rows and columns of a chessboard''; an example of what a configuration can be is given in \Cref{fig:configuration}. An important property of configurations is that if two elements of a chessboard lie in the same row/column and both of them feature in the configuration, then they lie in the same row/column in the configuration too. Similarly, if two elements lie in the same row/column of the configuration, then they lie in the same row/column of the chessboard.

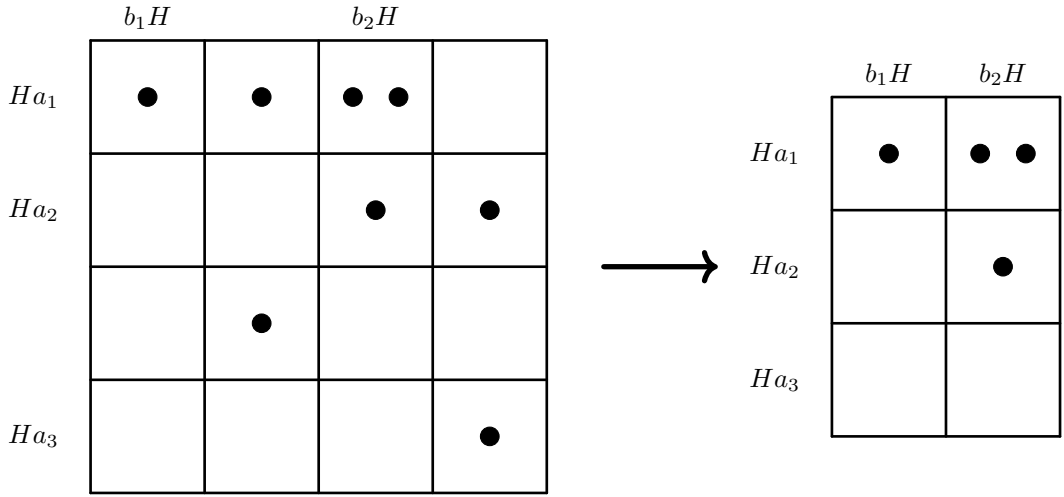
\begin{figure}[h!]
    \begin{center}
    	\renewcommand{\w}{4}
    	\renewcommand{\h}{4}
    	\begin{tikzpicture}[line cap=round,line join=round,x=1.5cm,y=1.5cm]
    	\foreach \y in {0,1,...,\h}
    	{
    		\draw [line width = 1pt] (0,\y) -- (\w,\y);
    	}
    	\foreach \x in {0,1,...,\w}
    	{
    		\draw [line width = 1pt] (\x,0) -- (\x,\h);
    	}
    	\draw[->, line width = 2pt] (4.5,2) -- (5.5,2);
    	\renewcommand{\w}{2}
    	\renewcommand{\h}{3}
    	\foreach \y in {0,1,...,\h}
    	{
    		\draw [line width = 1pt] (0+6.5,\y+0.5) -- (\w+6.5,\y+0.5);
    	}
    	\foreach \x in {0,1,...,\w}
    	{
    		\draw [line width = 1pt] (\x+6.5,0+0.5) -- (\x+6.5,\h+0.5);
    	}
    	\draw (-0.5,3.5) node {$Ha_1$};
    	\draw (-0.5,2.5) node {$Ha_2$};
    	\draw (-0.5,0.5) node {$Ha_3$};
    	\draw (0.5,4.2) node {$b_1H$};
    	\draw (2.5,4.2) node {$b_2H$};
    	
    	\draw (6,3) node {$Ha_1$};
    	\draw (6,2) node {$Ha_2$};
    	\draw (6,1) node {$Ha_3$};
    	\draw (7,3.7) node {$b_1H$};
    	\draw (8,3.7) node {$b_2H$};
    	
    	\draw [fill = black] (0.5,3.5) circle (3.5pt);
    	\draw [fill = black] (1.5,3.5) circle (3.5pt);
    	\draw [fill = black] (1.5,1.5) circle (3.5pt);
    	\draw [fill = black] (2.3,3.5) circle (3.5pt);
    	\draw [fill = black] (2.7,3.5) circle (3.5pt);
    	\draw [fill = black] (2.5,2.5) circle (3.5pt);
    	\draw [fill = black] (3.5,2.5) circle (3.5pt);
    	\draw [fill = black] (3.5,0.5) circle (3.5pt);
    	
    	\draw [fill = black] (7,3) circle (3.5pt);
    	\draw [fill = black] (7.8,3) circle (3.5pt);
    	\draw [fill = black] (8.2,3) circle (3.5pt);
    	\draw [fill = black] (8,2) circle (3.5pt);
    	\end{tikzpicture}
    \end{center}
    \caption{A entire chessboard, and a configuration obtained from that chessboard.}
    \label{fig:configuration}
\end{figure}

\subsection{L-spins}

We start with an elementary result on configurations, which allows us to define the L-spin -- a useful technique. It applies to the configuration shown in \Cref{fig:L-spin set-up}, where $a,b,c$ are elements of $G$ belonging to the multiset $S$.

\begin{figure}[h!]
	\begin{center}
		\renewcommand{\w}{2}
		\renewcommand{\h}{2}
		\begin{tikzpicture}[line cap=round,line join=round,x=1.5cm,y=1.5cm]
		\foreach \y in {0,1,...,\h}
		{
			\draw [line width = 1pt] (0,\y) -- (\w,\y);
		}
		\foreach \x in {0,1,...,\w}
		{
			\draw [line width = 1pt] (\x,0) -- (\x,\h);
		}
		\draw (0.5,1.5) node {$b$};
		\draw (0.5,0.5) node {$c$};
		\draw (1.5,1.5) node {$a$};
		\draw (1.5,0.5) node {\underline{$ab^{-1}c$}};
		\end{tikzpicture}
	\end{center}
    \caption{A configuration we may apply an L-spin to, defined below. The underlined element doesn't necessarily belong to the starting multiset.}
    \label{fig:L-spin set-up}
\end{figure}

\begin{lem}\label{lem:above}
     Given a configuration of a multiset $R$ in a table $T$, if there are elements $a,b,c\in R$ such that $a,b$ belong to the same row (right $H$ coset) and $b,c$ belong to the same column (left $H$ coset), then the element $ab^{-1}c$ lies in the same column as $a$ and the same row as $c$. Furthermore, we can perform Nielsen moves substituting any of the elements of the multiset $\{a,b,c\}$ with $ab^{-1}c$.
\end{lem}

\begin{proof}
    If $Ha = Hb$ and $bH=cH$, then $ab^{-1}\in H$ and $b^{-1}c\in H$. This means that $ab^{-1}cH = aH$ and that $Hab^{-1}c = Hc$. Thus indeed $ab^{-1}c$ lies in the column of $a$ and the row of $c$.
    Also, any of the following is a sequence of Nielsen moves:
    $$\{a,b,c\} \mapsto \{ab^{-1},b,c\} \mapsto \{ab^{-1}c,b,c\}$$
    $$\{a,b,c\} \mapsto \{a,b,b^{-1}c\} \mapsto \{a,b,ab^{-1}c\}$$
    $$\{a,b,c\} \mapsto \{a,b^{-1},c\} \mapsto \{a,ab^{-1},c\} \mapsto \{a,ab^{-1}c,c\}$$
\end{proof}

This motivates the following definition.

\begin{defn}\label{defn:L-spin}
     Given the configuration from \Cref{lem:above} (i.e. $Ha=Hb$, $bH=cH$), we call the transformation substituting an element of the multiset $\{a,b,c\}$ with the element $ab^{-1}c$ an \emph{L-spin}.
\end{defn}

Visually, an L-spin corresponds to rotating the L-shape formed in the table by the three elements $a,b,c$ in the configuration in question. This is shown in \Cref{fig:L-spin-visual}.

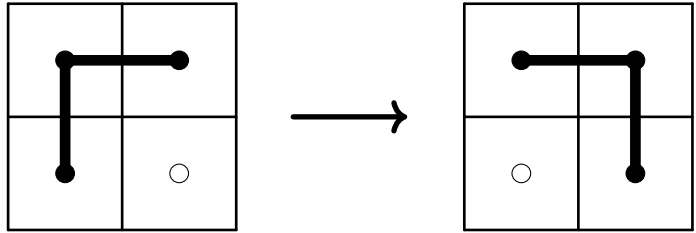
\begin{figure}[h!]
	\begin{center}
		\renewcommand{\w}{2}
		\renewcommand{\h}{2}
		\begin{tikzpicture}[line cap=round,line join=round,x=1.5cm,y=1.5cm]
		\foreach \y in {0,1,...,\h}
		{
			\draw [line width = 1pt] (0,\y) -- (\w,\y);
		}
		\foreach \x in {0,1,...,\w}
		{
			\draw [line width = 1pt] (\x,0) -- (\x,\h);
		}
		\draw [fill = black] (0.5,1.5) circle (3.5pt);
		\draw [fill = black] (0.5,0.5) circle (3.5pt);
		\draw [fill = black] (1.5,1.5) circle (3.5pt);
		\draw [] (1.5,0.5) circle (3.5pt);
		\draw [line width = 4pt] (0.5,0.5) -- (0.5,1.5);
		\draw [line width = 4pt] (0.5,1.5) -- (1.5,1.5);

		\draw [->, line width = 2pt] (2.5,1) -- (3.5,1);

		\foreach \y in {0,1,...,\h}
		{
			\draw [line width = 1pt] (0+4,\y) -- (\w+4,\y);
		}
		\foreach \x in {0,1,...,\w}
		{
			\draw [line width = 1pt] (\x+4,0) -- (\x+4,\h);
		}
		\draw [fill = black] (0.5+4,1.5) circle (3.5pt);
		\draw [] (0.5+4,0.5) circle (3.5pt);
		\draw [fill = black] (1.5+4,1.5) circle (3.5pt);
		\draw [fill = black] (1.5+4,0.5) circle (3.5pt);
		\draw [line width = 4pt] (0.5+4,1.5) -- (1.5+4,1.5);
		\draw [line width = 4pt] (1.5+4,0.5) -- (1.5+4,1.5);
		\end{tikzpicture}
	\end{center}
	\caption{Visual representation of an L-spin.}
	\label{fig:L-spin-visual}
\end{figure}

If some of $a,b,c$ happen to lie in the same box, we can still perform an operation of the above form. We call such a move a \emph{degenerate L-spin}, since the visual representation is somewhat different -- one of the dots in a box with two dots moves to a different non-empty box in the same column/row. This is shown in \Cref{fig:L-spin-degenerate}.

\begin{figure}[h!]
	\begin{center}
		\renewcommand{\w}{2}
		\renewcommand{\h}{2}
		\begin{tikzpicture}[line cap=round,line join=round,x=1.5cm,y=1.5cm]
		\foreach \y in {0,1,...,\h}
		{
			\draw [line width = 1pt] (0,\y) -- (\w,\y);
		}
		\foreach \x in {0,1,...,\w}
		{
			\draw [line width = 1pt] (\x,0) -- (\x,\h);
		}
		\draw [fill = black] (0.33,1.5) circle (3.5pt);
		\draw [fill = black] (0.33,0.5) circle (3.5pt);
		\draw [fill = black] (0.66,0.5) circle (3.5pt);
		\draw [] (0.66,1.5) circle (3.5pt);
		\draw [line width = 4pt] (0.33,0.5) -- (0.66,0.5);
		\draw [line width = 4pt] (0.33,1.5) -- (0.33,0.5);

		\draw [->, line width = 2pt] (2.5,1) -- (3.5,1);

		\foreach \y in {0,1,...,\h}
		{
			\draw [line width = 1pt] (0+4,\y) -- (\w+4,\y);
		}
		\foreach \x in {0,1,...,\w}
		{
			\draw [line width = 1pt] (\x+4,0) -- (\x+4,\h);
		}
		\draw [fill = black] (0.33+4,1.5) circle (3.5pt);
		\draw [fill = black] (0.33+4,0.5) circle (3.5pt);
		\draw [] (0.66+4,0.5) circle (3.5pt);
		\draw [fill = black] (0.66+4,1.5) circle (3.5pt);
		\draw [line width = 4pt] (0.33+4,0.5) -- (0.33+4,1.5);
		\draw [line width = 4pt] (0.33+4,1.5) -- (0.66+4,1.5);
		\end{tikzpicture}
	\end{center}
	\caption{Visual representation of a degenerate L-spin.}
	\label{fig:L-spin-degenerate}
\end{figure}
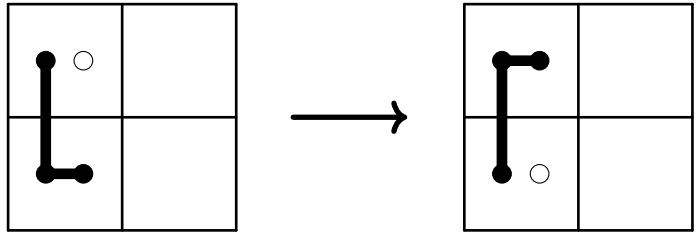

There is a similar, but more powerful construction if the chessboard containing the configuration is self-inverse. This is because it requires a configuration of only two elements instead of three.

\begin{lem}
    Suppose that the element $a$ has $H$-exponent $2$ %(i.e. $a\not \in H$, but $a^2\in H$, which is equivalent to $a$ and $a^{-1}$ lying in the same box outside of $H$)
    and that $b\in aH$. Then, the element $b^{-1}ab$ belongs to the row of $b$ and the column of $b^{-1}$. Furthermore, we can make a sequence of Nielsen moves transforming $\{a,b\}$ to $\{b^{-1}ab,b\}$.
    \label{lem:loop-shift}
\end{lem}

\begin{proof}
    The configuration is presented in \Cref{fig:inverse-L-spin-set-up}. Firstly, if $a$ is of exponent $2$, then $a^2\in H$ and so $a^2H = H$, which implies $aH=a^{-1}H$. Similarly for the right cosets: $Ha = Ha^{-1}$. Also, $a\notin H$, so the box of $a$ is definitely outside $H$.
    The configuration is presented in \Cref{fig:inverse-L-spin-set-up}. Looking at the element $b^{-1}ab$ we see that it belongs to $b^{-1}H$ since $a^{-1}$ and $b$ lie in the same column; it also belongs to $Hb$ as $a$ and $b$ lie in the same column. Then we can perform the Nielsen transformations:
    $$\{a,b\} \mapsto \{b^{-1} a,b\} \mapsto \{b^{-1}a^{-1}b,b\}$$
\end{proof}

\begin{figure}
	\begin{center}
		\renewcommand{\w}{2}
		\renewcommand{\h}{2}
		\begin{tikzpicture}[line cap=round,line join=round,x=1.5cm,y=1.5cm]
		\foreach \y in {0,1,...,\h}
		{
			\draw [line width = 1pt] (0,\y) -- (\w,\y);
		}
		\foreach \x in {0,1,...,\w}
		{
			\draw [line width = 1pt] (\x,0) -- (\x,\h);
		}
		\draw [fill = black] (0.33,1.5) circle (3.5pt);
		\draw [color = black] (0.33-0.16,1.5-0.2) node {$a$};
		\draw [] (0.66,1.5) circle (3.5pt);
		\draw [] (0.66+0.16,1.5-0.23) node {\underline{$a^{-1}$}};
		\draw [] (1.5,1.5) circle (3.5pt);
		\draw [] (1.5-0.16,1.5-0.23) node {\underline{$b^{-1}$}};
		\draw [fill = black] (0.5,0.5) circle (3.5pt);
		\draw [color = black] (0.5-0.16,0.5-0.16) node {$b$};
		\draw [] (1.5,0.5) circle (3.5pt);
		\draw [] (1.5,0.5-0.23) node {\underline{$b^{-1}a^{-1}b$}};

		\draw [->, line width = 2pt] (2.5,1) -- (3.5,1);

		\foreach \y in {0,1,...,\h}
		{
			\draw [line width = 1pt] (0+4,\y) -- (\w+4,\y);
		}
		\foreach \x in {0,1,...,\w}
		{
			\draw [line width = 1pt] (\x+4,0) -- (\x+4,\h);
		}
		\draw [] (0.33+4,1.5) circle (3.5pt);
		\draw [] (0.66+4,1.5) circle (3.5pt);
		\draw [] (1.5+4,1.5) circle (3.5pt);
		\draw [fill = black] (0.5+4,0.5) circle (3.5pt);
		\draw [fill = black] (1.5+4,0.5) circle (3.5pt);
		\end{tikzpicture}
	\end{center}
	\caption{A Nielsen transformation performed in a self-inverse chessboard.}
	\label{fig:inverse-L-spin-set-up}
\end{figure}
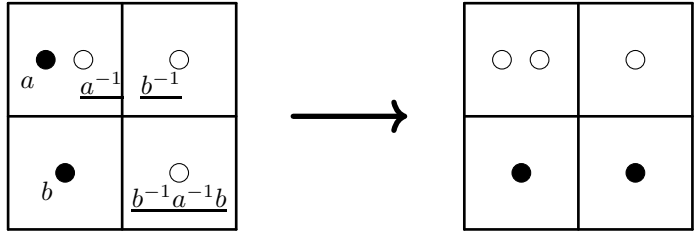

\subsection{Non self-inverse chessboards}
\label{sec:non-self-inverse-chessboards}

In this section we investigate subgroups which posses some non self-inverse chessboards. Recall \Cref{prop: left-clean} which says that a generating set of size equal to $[G:H]$ can be Nielsen transformed to a full left transversal.

\begin{defn}
    Let $S$ be a left transversal for $H$ (which is a set of elements in distinct columns). Given $g \in G$, we define a \emph{$g$-section of $S$} to be the set $S_g := HgH \cap S$.
\end{defn}

The reason we have defined $g$-sections is that if $HgH$ isn't a self-inverse chessboard (i.e. $HgH \neq Hg^{-1}H$), given a left transversal $S$ we can look at the set $S_{g^{-1}}^{-1}$,  interpreted as $(S_{g^{-1}})^{-1}$. Its elements are inverses of the elements of the $g^{-1}$-section $S_{g^{-1}}$ of $S$. Then, since $S_{g^{-1}}$ is left-diagonal (i.e. no two elements in the same left coset), the set $S_{g^{-1}}^{-1}$ is a right-diagonal set inside $HgH$.

The following definition will be useful in the proof of the next lemma, as well as in later sections.

\begin{defn}
    Given a configuration $S$ we define the \emph{graph of $S$} to be a graph whose vertices are elements of $S$ (with multiplicity). The edges are defined in the following way: for each row put an edge between every pair of elements lying in that row (call these \emph{horizontal edges}) and for each column put an edge between every pair of elements lying in that column (call them \emph{vertical edges}).
    We denote the graph by $\Sigma_S$.
    An example of how such a graph is created can be found in \Cref{fig:graph-example}.
\end{defn}

\begin{lem}
    \label{lem: solvable}
    Let $HgH \ne Hg^{-1}H$. Let $S$ be a multiset of elements in $HgH$ such that each row and each column contains exactly two elements of $S$. Then we can partition $S$ into two multisets $A$ and $B$ such that $A \cup B^{-1}$ is a left-right diagonal set with respect to $H \leq G$.
\end{lem}

\begin{figure}
	\begin{center}
		\renewcommand{\w}{4}
		\renewcommand{\h}{4}
		\begin{tikzpicture}[line cap=round,line join=round,x=1.5cm,y=1.5cm]
		\foreach \y in {0,1,...,\h}
		{
			\draw [line width = 1pt] (0,\y) -- (\w,\y);
		}
		\foreach \x in {0,1,...,\w}
		{
			\draw [line width = 1pt] (\x,0) -- (\x,\h);
		}
		\draw [fill = black] (0.5,3.5) circle (3.5pt);
		\draw [fill = black] (1.5,3.5) circle (3.5pt);
		\draw [fill = black] (1.5,2.5) circle (3.5pt);
		\draw [fill = black] (1.5,0.5) circle (3.5pt);
		\draw [fill = black] (2.5,2.5) circle (3.5pt);
		\draw [fill = black] (3.33,1.5) circle (3.5pt);
		\draw [fill = black] (3.66,1.5) circle (3.5pt);
		
		\draw [color = black] (0.5-0.16,3.5-0.16) node {$v_1$};
		\draw [color = black] (1.5-0.16,3.5-0.16) node {$v_2$};
		\draw [color = black] (1.5-0.16,2.5-0.16) node {$v_3$};
		\draw [color = black] (1.5-0.16,0.5-0.16) node {$v_4$};
		\draw [color = black] (2.5-0.16,2.5-0.16) node {$v_5$};
		\draw [color = black] (3.33-0.16,1.5-0.16) node {$v_6$};
		\draw [color = black] (3.66+0.16,1.5-0.16) node {$v_7$};

		\draw [line width=2pt, color=black] (0.5+6-0.5-1,3.5) -- (1.5+6-1,3.5);
		\draw [line width=2pt, dashdotted] (1.5+6-1,3.5) -- (1.5+6-0.5-1,2.5);
		\draw [line width=2pt, dashdotted] (1.5+6-1,3.5) -- (1.5+6-1,0.5);
		\draw [line width=2pt, dashdotted] (1.5+6-0.5-1,2.5) -- (1.5+6-1,0.5);
		\draw [line width=2pt, color=black] (1.5+6-0.5-1,2.5) -- (2.5+6-1,2.5);
		\draw [line width=2pt, color=black] (3.33+6-1,1.5) .. controls (3.33+6+0.2-1,1.5+0.2) and (3.66+6.5-0.2-1,1.5+0.2) .. (3.66+6.5-1,1.5);
		\draw [line width=2pt, dashdotted] (3.33+6-1,1.5) .. controls (3.33+6+0.2-1,1.5-0.2) and (3.66+6.5-0.2-1,1.5-0.2) .. (3.66+6.5-1,1.5);
		
		\draw [fill = black] (0.5+6-0.5-1,3.5) circle (3.5pt);
		\draw [color = black] (0.5-0.16+6-0.5-1,3.5-0.16) node {$v_1$};
		\draw [fill = black] (1.5+6-1,3.5) circle (3.5pt);
		\draw [color = black] (1.5+0.16+6-1,3.5-0.16) node {$v_2$};
		\draw [fill = black] (1.5+6-0.5-1,2.5) circle (3.5pt);
		\draw [color = black] (1.5-0.16+6-0.5-1,2.5-0.16) node {$v_3$};
		\draw [fill = black] (1.5+6-1,0.5) circle (3.5pt);
		\draw [color = black] (1.5-0.16+6-1,0.5-0.16) node {$v_4$};
		\draw [fill = black] (2.5+6-1,2.5) circle (3.5pt);
		\draw [color = black] (2.5+0.16+6-1,2.5-0.16) node {$v_5$};
		\draw [fill = black] (3.33+6-1,1.5) circle (3.5pt);
		\draw [color = black] (3.33-0.16+6-1,1.5-0.16) node {$v_6$};
		\draw [fill = black] (3.66+6.5-1,1.5) circle (3.5pt);
		\draw [color = black] (3.66+0.16+6.5-1,1.5-0.16) node {$v_7$};
		\end{tikzpicture}
	\end{center}
	\caption{An example of $\Sigma_{S}$. The vertical edges are dashed, the horizontal ones are solid.}
	\label{fig:graph-example}
\end{figure}
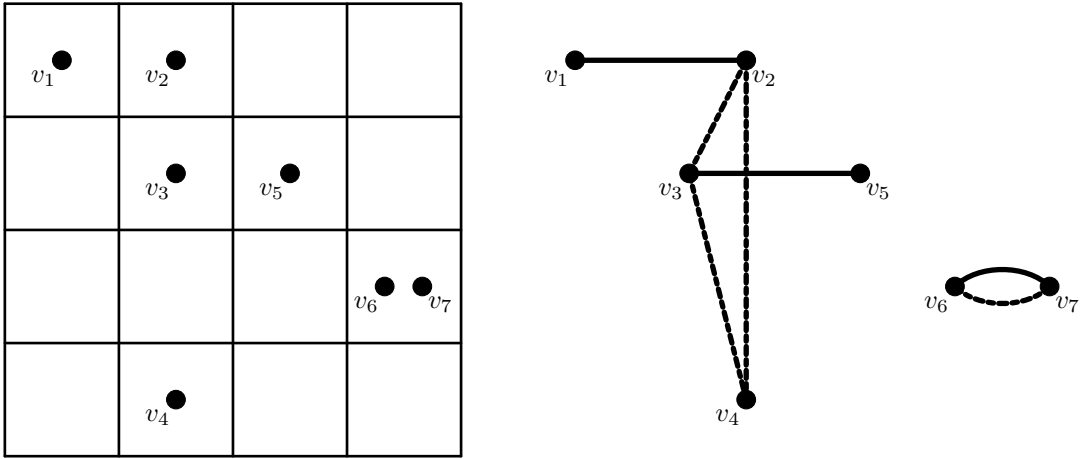

\begin{proof}
    Firstly, the graph $\Sigma_{S}$ is a union of disjoint cycles. This is because every column and every row of $HgH$ contains exactly two elements of $S$, so for every element of $S$ there exists precisely one element in the same column and precisely one element in the same row. Any graph in which every vertex is of degree $2$ is a union of disjoint cycles.
    
    Now, since each column and row contains only two elements, two consecutive edges can't be both horizontal or both vertical. Thus, all of the cycles of $\Sigma_{S}$ are of even length, so the graph is bipartite. Let's take $A$ and $B$ to be the sets of elements of $S$ corresponding to the bipartite components of $\Sigma_{S}$.
    
    In such a setting, in each of the rows and columns there is exactly one element of $A$ and one element of $B$, which means that both $A$ and $B$ are left-right diagonal sets. This in turn means that $A\cup B^{-1}$ is a left-right diagonal set, since $A \subset HgH$ and $B^{-1} \subset Hg^{-1}H$ while because of the assumption that $HgH \neq Hg^{-1}H$, we have $HgH \cap Hg^{-1}H = \emptyset$.
\end{proof}

This gives the following corollary.

\begin{cor}
    Suppose that $G$ has no self-inverse chessboards with respect to $H$ other than $H$ itself, and that $S$ is a left transversal such that for each $g \in G$ not in $H$, $S' = S_g \cup S_{g^{-1}}^{-1}$ satisfies the following condition: each row and column of $HgH$ contains exactly two elements of $S'$.
    Then, $S$ is Nielsen equivalent to a left-right transversal for $H$.
    \label{cor:trans-solv}
\end{cor}

\begin{proof}
    Let $S$ be as in the statement of the corollary. For each pair of inverse chessboards $HgH, Hg^{-1}H$, we can move all of the elements in these chessboards to just one of them using the inversion map, transforming $S_g\cup S_{g^{-1}}$ to $S_g \cup S_{g^{-1}}^{-1}$, which by the assumption is a configuration with two elements in each row and two elements in each column. By \Cref{lem: solvable} this can be partitioned into two multisets $A$ and $B$ such that $A\cup B^{-1}$ is diagonal with respect to $H$. After transforming $S_{g}\cup S_{g^{-1}} = A\cup B$ to $A\cup B^{-1}$, we obtain a diagonal set in the two chessboards $HgH$ and $Hg^{-1}H$. Doing this for each pair of inverse chessboards, we get a full left-right diagonal set (i.e. a left-right transversal).
\end{proof}

\subsubsection{Normal form of configurations}

\label{sec:normal-form}

In this section we will present a normal form of a configuration in a chessboard, meaning a representative of the equivalence class of configurations obtainable from a given one by performing L-spins. We start with a definition.

\begin{defn}
    Let $S$ be a configuration. The connected components of its graph $\Sigma_S$ are called \emph{connected components} (or just \emph{components}) of the configuration. We call a configuration \emph{connected} if it has only one connected component.
\end{defn}

It is clear that two elements of the configuration lying in the same row or column are in the same connected component, so the sets of rows and columns are partitioned as $R=\cup_{i=1}^k R_i$, $C=\cup_{i=1}^k C_i$ (where $R_i$ is a set of rows for each $i$, similarly $C_i$ is a set of columns) with:
\begin{itemize}
    \item The elements of $S$ lying in $R_i\times C_i$ form a connected component.
    \item There are no elements in $R_i\times C_j$ for $i\neq j$,
\end{itemize}
where by $R_i\times C_j$ we mean the union of boxes lying in the rows belonging to $R_i$ and at the same time lying in the columns belonging to $C_j$.

We can perform permutations on the rows and columns to get a situation where $R_i$'s and $C_j$'s consist of consecutive rows and columns and the top-left boxes of $R_i\times C_i$ are not empty. An example showing the connected components is in \Cref{fig:connected-components-example} there, the partition is as follows.
$$R_1 = \set{1,2}, \ R_2 = \set{3,4}, \ R_3 = \set{5}, \  C_1 = \set{1}, \  C_2 = \set{2,3,4}, \  C_3 = \set{5}$$

\begin{figure}[h!]
	\begin{center}
		\renewcommand{\w}{5}
		\renewcommand{\h}{5}
		\begin{tikzpicture}[line cap=round,line join=round,x=1.5cm,y=1.5cm]
		\foreach \y in {0,1,...,\h}
		{
			\draw [line width = 1pt] (0,\y) -- (\w,\y);
		}
		\foreach \x in {0,1,...,\w}
		{
			\draw [line width = 1pt] (\x,0) -- (\x,\h);
		}
		
		\draw [fill = black] (0.5,4.5) circle (3.5pt);
		\draw [fill = black] (0.33,3.5) circle (3.5pt);
		\draw [fill = black] (0.66,3.5) circle (3.5pt);
		\draw [fill = black] (1.5,2.5) circle (3.5pt);
		\draw [fill = black] (2.5,2.5) circle (3.5pt);
		\draw [fill = black] (3.5,2.5) circle (3.5pt);
		\draw [fill = black] (2.5,1.5) circle (3.5pt);
		\draw [fill = black] (4.5,0.5) circle (3.5pt);
		
		\draw [line width = 3pt] (0,5) -- (0,3) -- (1,3) -- (1,5) -- (0,5);
		\draw [line width = 3pt] (1,3) -- (1,1) -- (4,1) -- (4,3) -- (1,3);
		\draw [line width = 3pt] (4,1) -- (4,0) -- (5,0) -- (5,1) -- (4,1);
		\end{tikzpicture}
	\end{center}
    \caption{Example of a configuration's  connected components.}
    \label{fig:connected-components-example}
\end{figure}
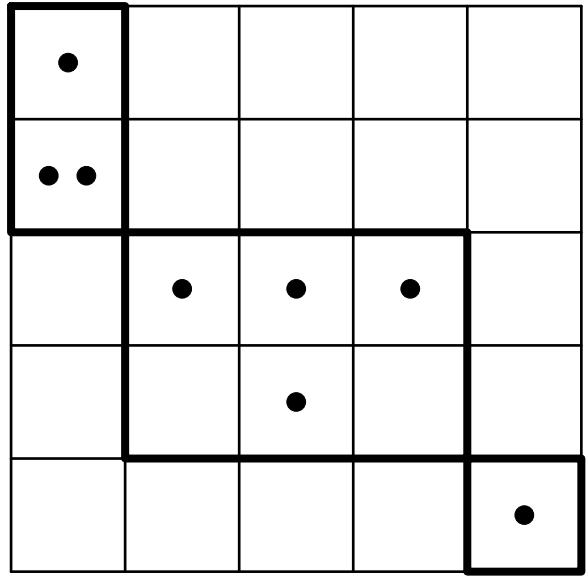

The following is an important observation.

\begin{prop}
    \label{prop:connectedness L-spins}
    L-spins don't change the connectedness of configurations in a chessboard.
\end{prop}

\begin{proof}
    We use the notation as in \Cref{fig:L-spin set-up}.
    In the case of a degenerate L-spin the result is clear since non-empty boxes before the L-spin are precisely those that are non-empty after the L-spin. In the case of a non-degenerate L-spin, note that the image of $c$ lies in the same column of $a$. Thus if our configuration were connected initially, the component containing the image of $c$ contains both $a$ and $b$, and so it remains connected.
\end{proof}

We can therefore focus our attention on the connected components. We want to transform a given configuration into a particularly ordered one.

Let's consider a single connected component $D$, and let $s$ be one of its elements. By permuting the columns and rows, we can assume that it is in the top left box. Then, if in the graph $\Sigma_D$ of $D$ there are any vertices of distance $2$ (in the graph $\Sigma_{D}$) from $s$, we can transform $D$ by an L-spin, decreasing the distance from $s$ to $1$, as shown in \Cref{fig:move-closer}.

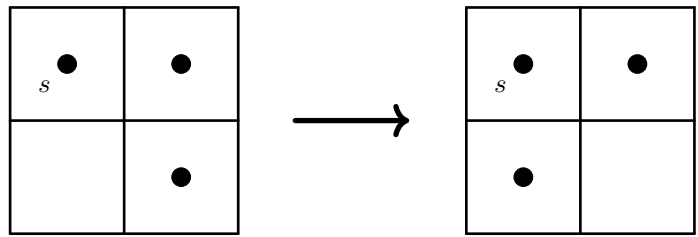
\begin{figure}[h!]
	\begin{center}
		\renewcommand{\w}{2}
		\renewcommand{\h}{2}
		\begin{tikzpicture}[line cap=round,line join=round,x=1.5cm,y=1.5cm]
		\foreach \y in {0,1,...,\h}
		{
			\draw [line width = 1pt] (0,\y) -- (\w,\y);
		}
		\foreach \x in {0,1,...,\w}
		{
			\draw [line width = 1pt] (\x,0) -- (\x,\h);
		}
		\draw [fill = black] (0.5,1.5) circle (3.5pt);
		\draw (0.5-0.2,1.5-0.2) node {$s$};
		\draw [fill = black] (1.5,1.5) circle (3.5pt);
		\draw [fill = black] (1.5,0.5) circle (3.5pt);
		
		\draw [->, line width = 2pt] (2.5,1) -- (3.5,1);

		\foreach \y in {0,1,...,\h}
		{
			\draw [line width = 1pt] (0+4,\y) -- (\w+4,\y);
		}
		\foreach \x in {0,1,...,\w}
		{
			\draw [line width = 1pt] (\x+4,0) -- (\x+4,\h);
		}
		\draw [fill = black] (0.5+4,1.5) circle (3.5pt);
		\draw (0.5-0.2+4,1.5-0.2) node {$s$};
		\draw [fill = black] (0.5+4,0.5) circle (3.5pt);
		\draw [fill = black] (1.5+4,1.5) circle (3.5pt);
		\end{tikzpicture}
	\end{center}
    \caption{Performing an L-spin to move an element closer to the top left corner.}
    \label{fig:move-closer}
\end{figure}
    
As the process doesn't increase the distance from $s$ of any elements, applying this repeatedly we can get to a point where all elements are of distance $1$ away from each other in $\Sigma_D$. Then, if there is a box, other than the top-left corner, such that there are at least two elements in it, we perform a degenerate L-spin to put the additional element in the top-left corner as shown in the \Cref{fig:normal-form-final-step}.

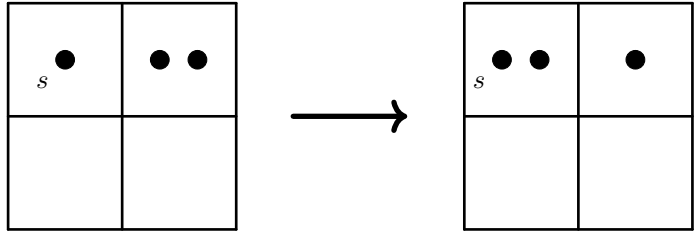
\begin{figure}[h!]
	\begin{center}
		\renewcommand{\w}{2}
		\renewcommand{\h}{2}
		\begin{tikzpicture}[line cap=round,line join=round,x=1.5cm,y=1.5cm]
		\foreach \y in {0,1,...,\h}
		{
			\draw [line width = 1pt] (0,\y) -- (\w,\y);
		}
		\foreach \x in {0,1,...,\w}
		{
			\draw [line width = 1pt] (\x,0) -- (\x,\h);
		}
		\draw [fill = black] (0.5,1.5) circle (3.5pt);
		\draw (0.5-0.2,1.5-0.2) node {$s$};
		\draw [fill = black] (1.33,1.5) circle (3.5pt);
		\draw [fill = black] (1.66,1.5) circle (3.5pt);
		
		\draw [->, line width = 2pt] (2.5,1) -- (3.5,1);

		\foreach \y in {0,1,...,\h}
		{
			\draw [line width = 1pt] (0+4,\y) -- (\w+4,\y);
		}
		\foreach \x in {0,1,...,\w}
		{
			\draw [line width = 1pt] (\x+4,0) -- (\x+4,\h);
		}
		\draw [fill = black] (0.33+4,1.5) circle (3.5pt);
		\draw (0.33-0.2+4,1.5-0.2) node {$s$};
		\draw [fill = black] (0.66+4,1.5) circle (3.5pt);
		\draw [fill = black] (1.5+4,1.5) circle (3.5pt);
		\end{tikzpicture}
	\end{center}
    \caption{Using a degenerate L-spin to move multiple elements into corner box.}
    \label{fig:normal-form-final-step}
\end{figure}

This leads us to define the normal form of a connected configuration.

\begin{defn}\label{defn:config-nf}
    A connected configuration $S$ is said to be in its \emph{normal form} if all the elements of $S$ are in the first row/column, with at most $1$ element in each such box, aside from potentially the top left corner box. We also enforce that the non-empty boxes all lie next to each other (i.e. permute rows/columns such that there are no gaps).
\end{defn}

An example of a normal form is given in \Cref{fig:normal-form-example}.

\begin{figure}[h!]
	\begin{center}
		\renewcommand{\w}{3}
		\renewcommand{\h}{3}
		\begin{tikzpicture}[line cap=round,line join=round,x=1.5cm,y=1.5cm]
		\foreach \y in {0,1,...,\h}
		{
			\draw [line width = 1pt] (0,\y) -- (\w,\y);
		}
		\foreach \x in {0,1,...,\w}
		{
			\draw [line width = 1pt] (\x,0) -- (\x,\h);
		}
		\draw [fill = black] (0.25,2.5) circle (3.5pt);
		\draw [fill = black] (0.5,2.5) circle (3.5pt);
		\draw [fill = black] (0.75,2.5) circle (3.5pt);
		\draw [fill = black] (1.5,2.5) circle (3.5pt);
		\draw [fill = black] (2.5,2.5) circle (3.5pt);
		\draw [fill = black] (0.5,1.5) circle (3.5pt);
		\end{tikzpicture}
	\end{center}
    \caption{An example of a normal form.}
    \label{fig:normal-form-example}
\end{figure}
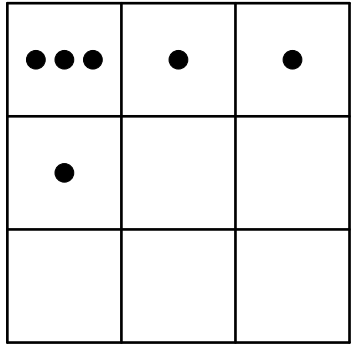

Finally, we introduce the following.

\begin{defn}
    Two configurations are said to be \emph{L-spin equivalent} if there exists a series of L-spins, and possibly permutation of rows and columns, transforming one configuration into the other.
\end{defn}

The next immediate result helps us in classifying the configuration.

\begin{prop}
    \label{prop:L-spin form determined}
    Let $C$ and $D$ be two L-spin equivalent configurations. Then they have the same:
    \begin{enumerate}
        \item number of elements,
        \item connected components,
        \item numbers of rows and columns.
    \end{enumerate}
\end{prop}

\begin{proof}
    1) is immediate from the definition of an L-spin. 2) is a consequence of \Cref{prop:connectedness L-spins}. 3) comes from the fact that if a column contains an element taking part in an L-spin, after the L-spin either the element in the column isn't transformed or it is transformed to a new element, but lying in the same column. The same applies to rows.
\end{proof}

\begin{prop}
    Every connected configuration is L-spin equivalent to a unique normal form.
\end{prop}

\begin{proof}
    We proved that there is some normal form that a configuration is L-spin equivalent to in the course of defining normal forms.
    On the other hand, a normal form is determined by the number of its columns, the number of its rows and the number of its elements, neither of which is changed by L-spins (by \Cref{prop:L-spin form determined}) or by permutation of rows or columns.
\end{proof}

\subsubsection{Solvable configurations}
\label{sec:solvable_non-self-inverse}

In light of \Cref{lem: solvable}, we make the following definition.

\begin{defn}
    A configuration is called \emph{solvable} if it is L-spin equivalent to a configuration with exactly two elements in each row and in each column.
\end{defn}

The term `solvable' is motivated by the observation in \Cref{lem: solvable} that if the left transversal $S$ has the property that for each $g \in G$ the multiset $S_g \cup S_{g^{-1}}^{-1}$ is in a solvable configuration, then $S$ can be Nielsen transformed to a left-right transversal.

The normal forms provide a useful criterion for determining when a configuration is solvable.

\begin{prop}
    \label{prop:solv}
    A connected configuration $S$ of $2n$ elements is solvable if and only if it has exactly $n$ rows and $n$ columns.
    In general, a configuration is solvable if and only if each of its connected components is solvable.
\end{prop}

From hereon we call connected solvable components \textit{square}.

\begin{proof}
    The `only if' direction is due to the fact that if the number of element in every row and in every column is $2$, then the total number of elements is equal to twice the number of rows and equal to twice the number of columns (which thus have to be equal).
    
    For the `if' direction, we will give an explicit transformation, constructed of the L-spins taking an element in top left corner to the bottom right corner, as shown in \Cref{fig:flip}.
    
    \begin{figure}[h!]
    	\begin{center}
    		\renewcommand{\w}{2}
    		\renewcommand{\h}{2}
    		\begin{tikzpicture}[line cap=round,line join=round,x=1.5cm,y=1.5cm]
    		\foreach \y in {0,1,...,\h}
    		{
    			\draw [line width = 1pt] (0,\y) -- (\w,\y);
    		}
    		\foreach \x in {0,1,...,\w}
    		{
    			\draw [line width = 1pt] (\x,0) -- (\x,\h);
    		}
    		\draw [fill = black] (0.5,1.5) circle (3.5pt);
    		\draw [fill = black] (1.5,1.5) circle (3.5pt);
    		\draw [fill = black] (0.5,0.5) circle (3.5pt);
    		
    		\draw [->, line width = 2pt] (2.5,1) -- (3.5,1);

    		\foreach \y in {0,1,...,\h}
    		{
    			\draw [line width = 1pt] (0+4,\y) -- (\w+4,\y);
    		}
    		\foreach \x in {0,1,...,\w}
    		{
    			\draw [line width = 1pt] (\x+4,0) -- (\x+4,\h);
    		}
    		\draw [fill = black] (1.5+4,0.5) circle (3.5pt);
    		\draw [fill = black] (0.5+4,0.5) circle (3.5pt);
    		\draw [fill = black] (1.5+4,1.5) circle (3.5pt);
    		\end{tikzpicture}
    	\end{center}
        \caption{The only type of L-spin used in the construction.}
        \label{fig:flip}
    \end{figure}
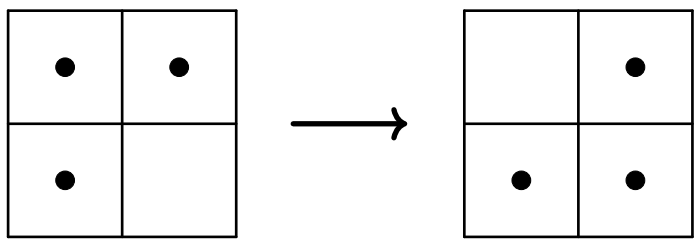
    
    For the descriptions of transformations we assume the convention that the top left corner is a box with coordinates $(1,1)$.
    
    In the first move we apply the described L-spin to move one element from box $(1,1)$ to box $(2,2)$ (we use the fact that the boxes $(1,2)$ and $(2,1)$ contain some elements).
    
    In the $i^{th}$ move (where $2\leq i \leq n-1$), we move the elements in boxes $(1,i), (2, i-1), \ldots, (i,1)$ to boxes $(2,i+1), (3, i+2), \ldots, (i+1,2)$. Because there are elements in boxes $(1,i+1),(2,i+2), \ldots, (i+1,1)$ we can perform these L-spins. This is shown pictorially in \Cref{fig:solv}.
    
    The configuration that we get in the end has exactly two elements in every row and in every column, which is what we needed. We provide an indicative illustration of this in \Cref{fig:solv}.
    
    \begin{figure}
    	\begin{center}
    		\renewcommand{\w}{4}
    		\renewcommand{\h}{4}
    		\begin{tikzpicture}[line cap=round,line join=round,x=1.25cm,y=1.25cm]
    		\foreach \y in {0,1,...,\h}
    		{
    			\draw [line width = 1pt] (0,\y) -- (\w,\y);
    		}
    		\foreach \x in {0,1,...,\w}
    		{
    			\draw [line width = 1pt] (\x,0) -- (\x,\h);
    		}
    		\draw [fill = black] (0.33,3.5) circle (3.5pt);
    		\draw [fill = black] (0.66,3.5) circle (3.5pt);
    		\draw [fill = black] (0.5,2.5) circle (3.5pt);
    		\draw [fill = black] (0.5,1.5) circle (3.5pt);
    		\draw [fill = black] (0.5,0.5) circle (3.5pt);
    		\draw [fill = black] (1.5,3.5) circle (3.5pt);
    		\draw [fill = black] (2.5,3.5) circle (3.5pt);
    		\draw [fill = black] (3.5,3.5) circle (3.5pt);
    		
    		\draw [dashed, ->, line width = 1.5pt, color = black] (0.33,3.5) .. controls (0.33,3.5-0.5) and (1.5-0.5,2.5) .. (1.5,2.5);
    		
    		\draw [->, line width = 2pt] (4.5,2) -- (5.5,2);
    		
    		\foreach \y in {0,1,...,\h}
    		{
    			\draw [line width = 1pt] (0+6,\y) -- (\w+6,\y);
    		}
    		\foreach \x in {0,1,...,\w}
    		{
    			\draw [line width = 1pt] (\x+6,0) -- (\x+6,\h);
    		}
    		\draw [fill = black] (0.5+6,3.5) circle (3.5pt);
    		\draw [fill = black] (1.5+6,2.5) circle (3.5pt);
    		\draw [fill = black] (0.5+6,2.5) circle (3.5pt);
    		\draw [fill = black] (0.5+6,1.5) circle (3.5pt);
    		\draw [fill = black] (0.5+6,0.5) circle (3.5pt);
    		\draw [fill = black] (1.5+6,3.5) circle (3.5pt);
    		\draw [fill = black] (2.5+6,3.5) circle (3.5pt);
    		\draw [fill = black] (3.5+6,3.5) circle (3.5pt);
    		
    		\draw [dashed, ->, line width = 1.5pt, color = black] (0.5+6,2.5) .. controls (0.5+6,2.5-0.5) and (1.5+6-0.5,1.5) .. (1.5+6,1.5);
    		\draw [dashed, ->, line width = 1.5pt, color = black] (1.5+6,3.5) .. controls (1.5+6,3.5-0.5) and (2.5+6-0.5,2.5) .. (2.5+6,2.5);
    		
    		\draw [->, line width = 2pt] (4.5+6,2) -- (5.5+6,2);
    		
    		\draw [->, line width = 2pt] (-1.5,2-5) -- (-0.5,2-5);
    		
    		\foreach \y in {0,1,...,\h}
    		{
    			\draw [line width = 1pt] (0,\y-5) -- (\w,\y-5);
    		}
    		\foreach \x in {0,1,...,\w}
    		{
    			\draw [line width = 1pt] (\x,0-5) -- (\x,\h-5);
    		}
    		\draw [fill = black] (0.5,3.5-5) circle (3.5pt);
    		\draw [fill = black] (1.5,1.5-5) circle (3.5pt);
    		\draw [fill = black] (2.5,2.5-5) circle (3.5pt);
    		\draw [fill = black] (0.5,1.5-5) circle (3.5pt);
    		\draw [fill = black] (0.5,0.5-5) circle (3.5pt);
    		\draw [fill = black] (1.5,2.5-5) circle (3.5pt);
    		\draw [fill = black] (2.5,3.5-5) circle (3.5pt);
    		\draw [fill = black] (3.5,3.5-5) circle (3.5pt);
    		
    		\draw [dashed, ->, line width = 1.5pt, color = black] (0.5,1.5-5) .. controls (0.5,1.5-5-0.5) and (1.5-0.5,0.5-5) .. (1.5,0.5-5);
    		\draw [dashed, ->, line width = 1.5pt, color = black] (1.5,2.5-5) .. controls (1.5,2.5-5-0.5) and (2.5-0.5,1.5-5) .. (2.5,1.5-5);
    		\draw [dashed, ->, line width = 1.5pt, color = black] (2.5,3.5-5) .. controls (2.5,3.5-5-0.5) and (3.5-0.5,2.5-5) .. (3.5,2.5-5);
    		
    		\draw [->, line width = 2pt] (4.5,2-5) -- (5.5,2-5);
    		
    		\foreach \y in {0,1,...,\h}
    		{
    			\draw [line width = 1pt] (0+6,\y-5) -- (\w+6,\y-5);
    		}
    		\foreach \x in {0,1,...,\w}
    		{
    			\draw [line width = 1pt] (\x+6,0-5) -- (\x+6,\h-5);
    		}
    		\draw [fill = black] (0.5+6,3.5-5) circle (3.5pt);
    		\draw [fill = black] (0.5+6,0.5-5) circle (3.5pt);
    		\draw [fill = black] (1.5+6,0.5-5) circle (3.5pt);
    		\draw [fill = black] (1.5+6,1.5-5) circle (3.5pt);
    		\draw [fill = black] (2.5+6,1.5-5) circle (3.5pt);
    		\draw [fill = black] (2.5+6,2.5-5) circle (3.5pt);
    		\draw [fill = black] (3.5+6,2.5-5) circle (3.5pt);
    		\draw [fill = black] (3.5+6,3.5-5) circle (3.5pt);
    		\end{tikzpicture}
    	\end{center}
	    \caption{A sequence of L-spins transforming the normal form to a solvable configuration}
	    \label{fig:solv}
	\end{figure}
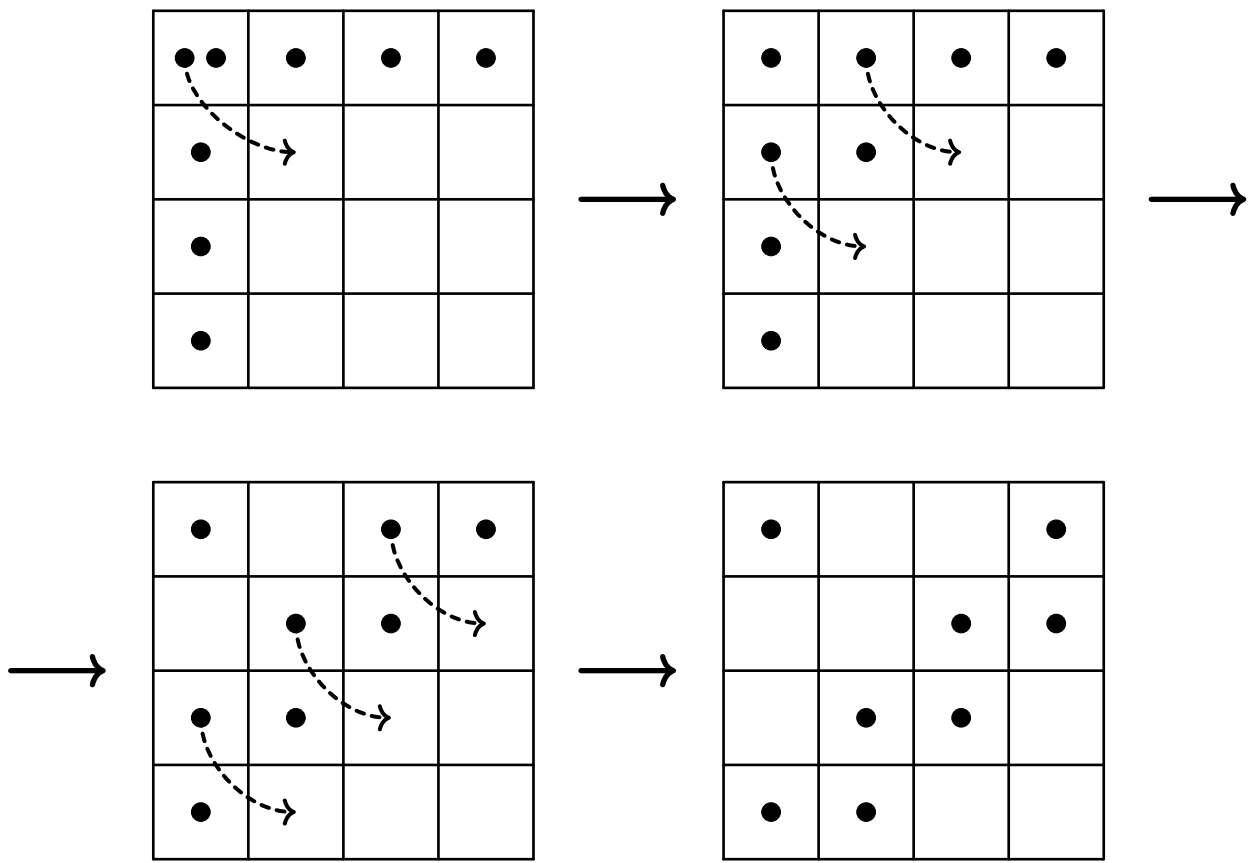
\end{proof}

\subsection{Self-inverse chessboards}
\label{sec:self-inverse-chessboards}

\begin{rem}\label{rem:self-inverse}
In the case of a self-inverse chessboard  $HgH=Hg^{-1}H$  we have a very different setting --- the inversion map won't enable us to move more elements into that chessboard. However, it gives us an additional operation in a given chessboard, apart from the L-spin. Note that without the inversion, when starting with a left transversal, we couldn't perform any L-spins --- this is since every column contained a single element.
\end{rem}

\subsubsection{Inverse-dual graphs}

We now give a convenient way of describing what happens with self-inverse chessboards. It turns out that again graphs are useful for this.

\begin{defn}
    Let $HgH = Hg^{-1}H$ and let $S$ be a configuration in this chessboard. According to \Cref{rem:self-inverse}, let the columns be $a_1H,a_2H,\ldots, a_nH$ and let the rows be $Ha_1^{-1},Ha_2^{-1}, \ldots, Ha_n^{-1}$. Then, we create a directed graph $\Theta_S$ as follows. Take $n$ vertices $v_1, \ldots, v_n$ and for each element $s$ of $S$, put a directed edge $v_iv_j$ if $s$ lies in the box of coordinates $(i,j)$ (that is, $s \in a_jH \cap Ha_i^{-1}$).
    
    We call this graph the \emph{inverse-dual graph} for a configuration $S$, reflecting the fact that, contrary to the previous graph of a configuration $\Sigma_S$, in $\Theta_S$ the elements of a configuration are represented with edges, not vertices (hence inverse-\textit{dual}); also, the graph is defined in the context of self-inverse chessboards (hence \textit{inverse}-dual).
\end{defn}

The construction for such a graph is shown in \Cref{fig:inverse-graph-example}. It is worth noting that the inverse-dual graph can contain both parallel edges (corresponding to elements lying in the same box) and loops (corresponding to elements lying on the diagonal).

\begin{figure}[h!]
	\begin{center}
		\renewcommand{\w}{4}
		\renewcommand{\h}{4}
		\begin{tikzpicture}[line cap=round,line join=round,x=1.5cm,y=1.5cm]
		\foreach \y in {0,1,...,\h}
		{
			\draw [line width = 1pt] (0,\y) -- (\w,\y);
		}
		\foreach \x in {0,1,...,\w}
		{
			\draw [line width = 1pt] (\x,0) -- (\x,\h);
		}
		\draw [fill = black] (0.5,3.5) circle (3.5pt);
		\draw [fill = black] (1.5,3.5) circle (3.5pt);
		\draw [fill = black] (1.5,2.5) circle (3.5pt);
		\draw [fill = black] (1.5,0.5) circle (3.5pt);
		\draw [fill = black] (2.5,2.5) circle (3.5pt);
		\draw [fill = black] (3.33,1.5) circle (3.5pt);
		\draw [fill = black] (3.66,1.5) circle (3.5pt);
		
		\draw [color = black] (0.5-0.16,3.5-0.16) node {$a$};
		\draw [color = black] (1.5-0.16,3.5-0.16) node {$b$};
		\draw [color = black] (1.5-0.16,2.5-0.16) node {$c$};
		\draw [color = black] (1.5-0.16,0.5-0.16) node {$d$};
		\draw [color = black] (2.5-0.16,2.5-0.16) node {$e$};
		\draw [color = black] (3.33-0.16,1.5-0.16) node {$f$};
		\draw [color = black] (3.66+0.16,1.5-0.16) node {$g$};

		\draw [fill = black] (1+5,1) circle (3.5pt);
		\draw [color = black] (1+5-0.2,1-0.2) node {$v_3$};
		\draw [fill = black] (3+5,1) circle (3.5pt);
		\draw [color = black] (3+5+0.2,1-0.2) node {$v_4$};
		\draw [fill = black] (1+5,3) circle (3.5pt);
		\draw [color = black] (1+5-0.2,3-0.2) node {$v_1$};
		\draw [fill = black] (3+5,3) circle (3.5pt);
		\draw [color = black] (3+5+0.3,3-0.1) node {$v_2$};
		
		\draw [latex-, line width = 2pt] (1+5,3) .. controls (1+5-0.5,3+0.5) and (1+5+0.5,3+0.5) .. (1+5,3);
		\draw [color = black] (1+5,3+0.5) node {$a$};
		\draw [-latex, line width = 2pt] (1+5,3) .. controls (1+5+0.5,3) and (3+5-0.5,3) .. (3+5,3);
		\draw [color = black] (2+5,3+0.2) node {$b$};
		\draw [-latex, line width = 2pt] (3+5,3) .. controls (3+5-0.5,3+0.5) and (3+5+0.5,3+0.5) .. (3+5,3);
		\draw [color = black] (3+5,3+0.5) node {$c$};
		\draw [-latex, line width = 2pt] (3+5,3) .. controls (3+5-1,3-0.5) and (1+5+0.5,1+1) .. (1+5,1);
		\draw [color = black] (2+5-0.05,2+0.05) node {$e$};
		\draw [-latex, line width = 2pt] (1+5,1) .. controls (1+5+0.5,1+0.5) and (3+5-0.3,1+0.3) .. (3+5,1);
		\draw [color = black] (2+5,1+0.5) node {$f$};
		\draw [-latex, line width = 2pt] (1+5,1) .. controls (1+5+0.5,1-0.5) and (3+5-0.3,1-0.3) .. (3+5,1);
		\draw [color = black] (2+5,1-0.5) node {$g$};
		\draw [-latex, line width = 2pt] (3+5,3) .. controls (3+5+0.5,3-0.5) and (3+5+0.3,1+0.3) .. (3+5,1);
		\draw [color = black] (3+5+0.1,2) node {$d$};
		\end{tikzpicture}
	\end{center}
    \caption{A configuration $S$ with associated inverse-dual graph $\Theta_S$.}
    \label{fig:inverse-graph-example}
\end{figure}
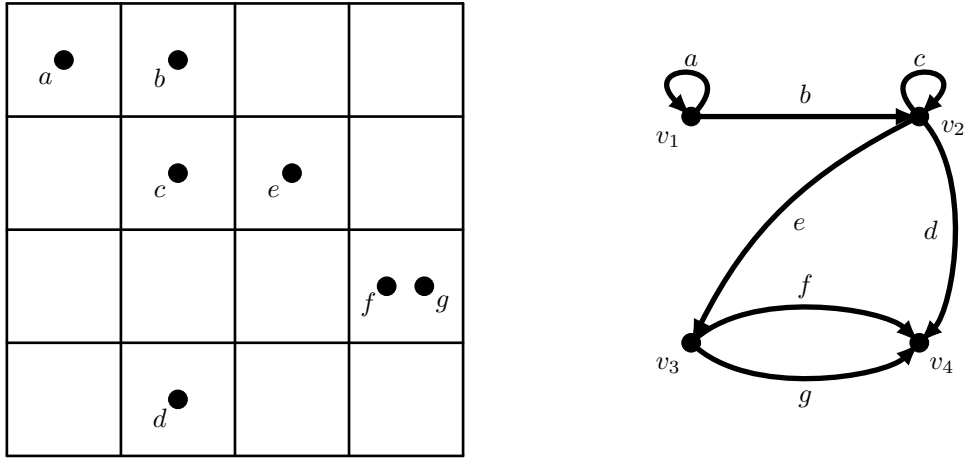

Now, we want to know what the inversions and L-spins correspond to in the context of inverse-dual graphs. The inversion map turns out to be simple.

\begin{prop}
    Let $S$ be a configuration in a self-inverse chessboard and $\Theta_S$ be its inverse-dual graph. Then, for $x\in S$, the Nielsen move $x \mapsto x^{-1}$ corresponds to changing the orientation of the edge corresponding to $x$.
\end{prop}

\begin{proof}
    Let $x$ be the edge $v_iv_j$, i.e. $x \in a_jH \cap Ha_i^{-1}$. Then $x^{-1} \in a_iH\cap Ha_j^{-1}$ and so it corresponds to the edge $v_jv_i$.
\end{proof}

Because of this, unless otherwise stated, from hereon  we will consider our inverse-dual graphs to be undirected, since changing direction of an edge is one of the Nielsen moves. Now we want to understand what the L-spins correspond to.

\begin{prop}
	\begin{enumerate}
		\item Let $a,b,c$ be elements of a configuration $S$ in a self-inverse chessboard $HgH$. Then: $a,b,c$ are in a configuration allowing an L-spin if and only if the three edges corresponding to these elements in the inverse-dual graph $\Theta_S$, when ordered appropriately satisfy: the first and the second go into the same vertex; the second and the third go out of the same vertex.
		\item If $a,b,c$ are indeed in a configuration allowing an L-spin, then an L-spin  on $\set{a,b,c}$ corresponds to exchanging one of the edges with the directed edge from the vertex at the beginning of the first edge to the vertex at the end of the third edge, as shown in \Cref{fig:L-spin.0}.
    \end{enumerate}
\end{prop}

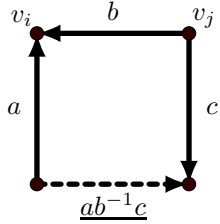
\begin{figure}[h!]
    \begin{center}
    	\definecolor{ttqqqq}{rgb}{0.2,0,0}
    	\begin{tikzpicture}[line cap=round,line join=round,x=1cm,y=1cm]
    	%        \clip(-2,-4) rectangle (10,2);
    	\draw [-latex, line width=2pt] (0,-2)-- (0,0);
    	\draw [latex-, line width=2pt] (0,0)-- (2,0);
    	\draw [-latex, line width=2pt] (2,0)-- (2,-2);
    	\draw [-latex, line width=2pt, dashed, color=black] (0,-2)-- (2,-2);
    	\draw [fill=ttqqqq] (0,-2) circle (2.5pt);
    	\draw [fill=ttqqqq] (0,0) circle (2.5pt);
    	\draw [fill=ttqqqq] (2,0) circle (2.5pt);
    	\draw [fill=ttqqqq] (2,-2) circle (2.5pt);
    	\draw [color=black] (-0.3,-1) node {$a$};
    	\draw [color=black] (1,0.3) node {$b$};
    	\draw [color=black] (2.3,-1) node {$c$};
    	\draw [color=black] (1,-2.3) node {\underline{$ab^{-1}c$}};
    	\draw [color=black] (-0.2,0.2) node {$v_i$};
    	\draw [color=black] (2.2,0.2) node {$v_j$};
    	\end{tikzpicture}
    \end{center}
    \caption{A situation in which an L-spin can be performed.}
    \label{fig:L-spin.0}
\end{figure}

\begin{proof}
    The elements $a,b,c$ are in a configuration allowing an L-spin if (in the context of a chessboard) two of them lie in the same column and two of them lie in the same row. Without loss of generality, let's suppose that $b, c$ lie in the same column $a_iH$ and $a, b$ in the same row $Ha_j^{-1}$ (like the situation in \Cref{fig:L-spin set-up}). This gives a situation shown in \Cref{fig:L-spin.0} -- the edges $a$ and $b$ both go into the vertex $v_i$, while the edges $b$ and $c$ both go out of the vertex $v_j$. This proves the first part of the proposition.
    
    For the second part, let's note that the new element created in the L-spin is $ab^{-1}c$ lying in the same row as $c$ and the same column as $a$. These correspond to vertices at the beginning and the end of the first and third edge, respectively. Thus, the element created in the L-spin corresponds to an edge from the beginning of the first edge to the end of the third edge (drawn dashed in \Cref{fig:L-spin.0}).
\end{proof}

 \Cref{fig:L-spin.1,fig:L-spin.2,fig:L-spin.3,fig:L-spin.4} on page~\pageref{fig:L-spin.1} show what the L-spins correspond to in the inverse-dual graph in situations differing by whether or not $\Theta_S$ involves parallel edges (i.e. two elements in the same box) and whether or not $\Theta_S$ involves loops (elements on the diagonal). The Nielsen moves are exactly the same as described above, the difference is purely visual.

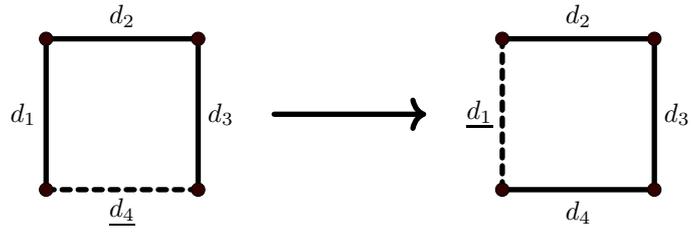
\begin{figure}[p]
    \begin{center}
    	\definecolor{ttqqqq}{rgb}{0.2,0,0}
    	\begin{tikzpicture}[line cap=round,line join=round,x=1cm,y=1cm]
    	%        \clip(-2,-4) rectangle (10,2);
    	\draw [line width=2pt] (0,-2)-- (0,0);
    	\draw [line width=2pt] (0,0)-- (2,0);
    	\draw [line width=2pt] (2,0)-- (2,-2);
    	\draw [line width=2pt, dashed, color=black] (0,-2)-- (2,-2);
    	\draw [fill=ttqqqq] (0,-2) circle (2.5pt);
    	\draw [fill=ttqqqq] (0,0) circle (2.5pt);
    	\draw [fill=ttqqqq] (2,0) circle (2.5pt);
    	\draw [fill=ttqqqq] (2,-2) circle (2.5pt);
    	\draw [color=black] (-0.3,-1) node {$d_{1}$};
    	\draw [color=black] (1,0.3) node {$d_{2}$};
    	\draw [color=black] (2.3,-1) node {$d_{3}$};
    	\draw [color=black] (1,-2.3) node {\underline{$d_{4}$}};
    	
    	\draw [<-, line width=2pt] (5,-1)-- (3,-1);
    	
    	\draw [line width=2pt, dashed, color=black] (6,-2)-- (6,0);
    	\draw [line width=2pt] (6,0)-- (8,0);
    	\draw [line width=2pt] (8,0)-- (8,-2);
    	\draw [line width=2pt] (6,-2)-- (8,-2);
    	\draw [fill=ttqqqq] (6,-2) circle (2.5pt);
    	\draw [fill=ttqqqq] (6,0) circle (2.5pt);
    	\draw [fill=ttqqqq] (8,0) circle (2.5pt);
    	\draw [fill=ttqqqq] (8,-2) circle (2.5pt);
    	\draw [color=black] (5.7,-1) node {\underline{$d_{1}$}};
    	\draw [color=black] (7,0.3) node {$d_{2}$};
    	\draw [color=black] (8.3,-1) node {$d_{3}$};
    	\draw [color=black] (7,-2.3) node {$d_{4}$};
    	\end{tikzpicture}
    \end{center}
    \caption{L-spin in $\Theta_S$.}
    \label{fig:L-spin.1}
\end{figure}

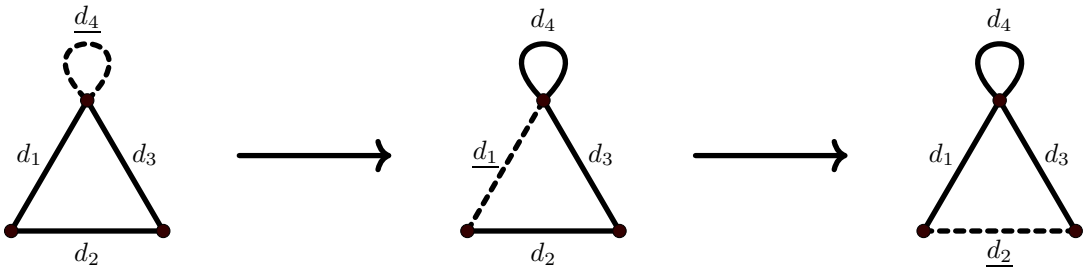
\begin{figure}[p]
	\begin{center}
		\definecolor{ttqqqq}{rgb}{0.2,0,0}
		\begin{tikzpicture}[line cap=round,line join=round,x=1cm,y=1cm]
		%        \clip(-2,-4) rectangle (10,2);
		\draw [line width=2pt] (0,0)-- (2,0);
		\draw [line width=2pt] (2,0)-- (1,1.732);
		\draw [line width=2pt] (1,1.732)-- (0,0);
		\draw [line width=2pt, dashed, color=black](1,1.732) .. controls (1+1,1.732+1) and (1-1,1.732+1) .. (1,1.732);
		\draw [fill=ttqqqq] (1,1.732) circle (2.5pt);
		\draw [fill=ttqqqq] (0,0) circle (2.5pt);
		\draw [fill=ttqqqq] (2,0) circle (2.5pt);
		\draw [color=black] (0.5-0.3*1.732/2,1.732/2+0.3/2) node {$d_{1}$};
		\draw [color=black] (1,-0.3) node {$d_{2}$};
		\draw [color=black] (1.5+0.3*1.732/2,1.732/2+0.3/2) node {$d_{3}$};
		\draw [color=black] (1,2.8) node {\underline{$d_{4}$}};
		
		\draw [<-, line width=2pt] (5,1)-- (3,1);
		
		\draw [line width=2pt] (0+6,0)-- (2+6,0);
		\draw [line width=2pt] (2+6,0)-- (1+6,1.732);
		\draw [line width=2pt, dashed, color=black] (1+6,1.732)-- (0+6,0);
		\draw [line width=2pt, color=black](1+6,1.732) .. controls (1+1+6,1.732+1) and (1-1+6,1.732+1) .. (1+6,1.732);
		\draw [fill=ttqqqq] (1+6,1.732) circle (2.5pt);
		\draw [fill=ttqqqq] (0+6,0) circle (2.5pt);
		\draw [fill=ttqqqq] (2+6,0) circle (2.5pt);
		\draw [color=black] (0.5-0.3*1.732/2+6,1.732/2+0.3/2) node {\underline{$d_{1}$}};
		\draw [color=black] (1+6,-0.3) node {$d_{2}$};
		\draw [color=black] (1.5+0.3*1.732/2+6,1.732/2+0.3/2) node {$d_{3}$};
		\draw [color=black] (1+6,2.8) node {$d_{4}$};
		
		\draw [<-, line width=2pt] (5+6,1)-- (3+6,1);        
		
		\draw [line width=2pt, dashed, color=black] (0+12,0)-- (2+12,0);
		\draw [line width=2pt] (2+12,0)-- (1+12,1.732);
		\draw [line width=2pt] (1+12,1.732)-- (0+12,0);
		\draw [line width=2pt, color=black](1+12,1.732) .. controls (1+1+12,1.732+1) and (1-1+12,1.732+1) .. (1+12,1.732);
		\draw [fill=ttqqqq] (1+12,1.732) circle (2.5pt);
		\draw [fill=ttqqqq] (0+12,0) circle (2.5pt);
		\draw [fill=ttqqqq] (2+12,0) circle (2.5pt);
		\draw [color=black] (0.5-0.3*1.732/2+12,1.732/2+0.3/2) node {$d_{1}$};
		\draw [color=black] (1+12,-0.3) node {\underline{$d_{2}$}};
		\draw [color=black] (1.5+0.3*1.732/2+12,1.732/2+0.3/2) node {$d_{3}$};
		\draw [color=black] (1+12,2.8) node {$d_{4}$};
		\end{tikzpicture}
	\end{center}
    \caption{L-spin in $\Theta_S$ involving a loop.}
    \label{fig:L-spin.2}
\end{figure}

\begin{figure}[p]
	\begin{center}
		\definecolor{ttqqqq}{rgb}{0.2,0,0}
		\begin{tikzpicture}[line cap=round,line join=round,x=1cm,y=1cm]
		%        \clip(-2,-4) rectangle (10,2);
		\draw [line width=2pt, color=black](0,0) .. controls (0+0.5,0+0.5) and (2-0.5,0+0.5) .. (2,0);
		\draw [line width=2pt, color=black](0,0) .. controls (0+0.5,0-0.5) and (2-0.5,0-0.5) .. (2,0);
		\draw [line width=2pt, color=black](2,0) .. controls (2+0.5,0+0.5) and (4-0.5,0+0.5) .. (4,0);
		\draw [line width=2pt, dashed, color=black](2,0) .. controls (2+0.5,0-0.5) and (4-0.5,0-0.5) .. (4,0);
		\draw [fill=ttqqqq] (0,0) circle (2.5pt);
		\draw [fill=ttqqqq] (2,0) circle (2.5pt);
		\draw [fill=ttqqqq] (4,0) circle (2.5pt);
		\draw [color=black] (1,0.6) node {$d_1$};
		\draw [color=black] (1,-0.6) node {$d_2$};
		\draw [color=black] (3,0.6) node {$d_3$};
		\draw [color=black] (3,-0.6) node {\underline{$d_4$}};
		
		\draw [<-, line width=2pt] (7,0)-- (5,0);
		
		\draw [line width=2pt, dashed, color=black](0+8,0) .. controls (0+0.5+8,0+0.5) and (2-0.5+8,0+0.5) .. (2+8,0);
		\draw [line width=2pt, color=black](0+8,0) .. controls (0+0.5+8,0-0.5) and (2-0.5+8,0-0.5) .. (2+8,0);
		\draw [line width=2pt, color=black](2+8,0) .. controls (2+0.5+8,0+0.5) and (4-0.5+8,0+0.5) .. (4+8,0);
		\draw [line width=2pt, color=black](2+8,0) .. controls (2+0.5+8,0-0.5) and (4-0.5+8,0-0.5) .. (4+8,0);
		\draw [fill=ttqqqq] (0+8,0) circle (2.5pt);
		\draw [fill=ttqqqq] (2+8,0) circle (2.5pt);
		\draw [fill=ttqqqq] (4+8,0) circle (2.5pt);
		\draw [color=black] (1+8,0.6) node {\underline{$d_1$}};
		\draw [color=black] (1+8,-0.6) node {$d_2$};
		\draw [color=black] (3+8,0.6) node {$d_3$};
		\draw [color=black] (3+8,-0.6) node {$d_4$};
		\end{tikzpicture}
	\end{center}
    \caption{L-spin in $\Theta_S$ involving parallel edges.}
    \label{fig:L-spin.3}
\end{figure}
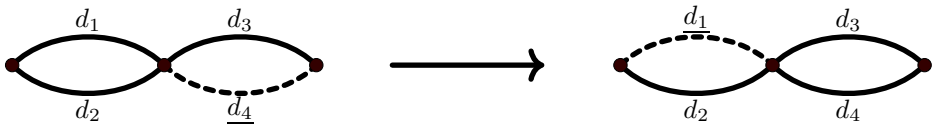

\begin{figure}[p]
    \begin{center}
    	\definecolor{ttqqqq}{rgb}{0.2,0,0}
    	\begin{tikzpicture}[line cap=round,line join=round,x=1cm,y=1cm]
    	%        \clip(-2,-4) rectangle (10,2);
    	\draw [line width=2pt, color=black](0,0) .. controls (0+0.5,0+0.5) and (2-0.5,0+0.5) .. (2,0);
    	\draw [line width=2pt, color=black](0,0) .. controls (0+0.5,0-0.5) and (2-0.5,0-0.5) .. (2,0);
    	\draw [line width=2pt, color=black](2,0) .. controls (2+0.5,0+0.5) and (2+0.5,0-0.5) .. (2,0);
    	\draw [line width=2pt, dashed, color=black](0,0) .. controls (-0.5,-0.5) and (-0.5,+0.5) .. (0,0);
    	\draw [fill=ttqqqq] (0,0) circle (2.5pt);
    	\draw [fill=ttqqqq] (2,0) circle (2.5pt);
    	\draw [color=black] (1,0.6) node {$d_1$};
    	\draw [color=black] (1,-0.6) node {$d_2$};
    	\draw [color=black] (2.7,0) node {$d_3$};
    	\draw [color=black] (-0.7,0) node {\underline{$d_4$}};
    	
    	\draw [<-, line width=2pt] (6,0)-- (4,0);
    	
    	\draw [line width=2pt, dashed, color=black](0+8,0) .. controls (0+0.5+8,0+0.5) and (2-0.5+8,0+0.5) .. (2+8,0);
    	\draw [line width=2pt, color=black](0+8,0) .. controls (0+0.5+8,0-0.5) and (2-0.5+8,0-0.5) .. (2+8,0);
    	\draw [line width=2pt, color=black](2+8,0) .. controls (2+0.5+8,0+0.5) and (2+0.5+8,0-0.5) .. (2+8,0);
    	\draw [line width=2pt, color=black](0+8,0) .. controls (-0.5+8,-0.5) and (-0.5+8,+0.5) .. (0+8,0);
    	\draw [fill=ttqqqq] (0+8,0) circle (2.5pt);
    	\draw [fill=ttqqqq] (2+8,0) circle (2.5pt);
    	\draw [color=black] (1+8,0.6) node {\underline{$d_1$}};
    	\draw [color=black] (1+8,-0.6) node {$d_2$};
    	\draw [color=black] (2.7+8,0) node {$d_3$};
    	\draw [color=black] (-0.7+8,0) node {$d_4$};
    	\end{tikzpicture}
    \end{center}
    \caption{L-spin in $\Theta_S$ involving a loop and parallel edges.}
    \label{fig:L-spin.4}
\end{figure}
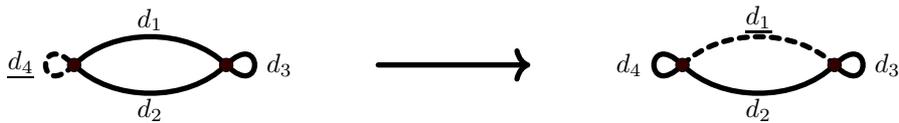

We now give an  operation similar to an L-spin, which can be performed with only two elements -- it was alluded to in \Cref{lem:loop-shift} in the context of chessboards.

\begin{prop}
    Suppose $HgH$ represents a self-inverse chessboard and $a \in S$ lies on its diagonal (so $a$ corresponds to a loop in $\Theta_S$). Let $b \in S$ lie in the column of $a$. Then the transformation $a \mapsto ab \mapsto b^{-1}ab$ corresponds to moving the loop $a$ from one vertex of the edge $b$ to the other.
    \label{prop:loop-shift}
\end{prop}

\begin{proof}
    Since $b\in aH$, we have $b^{-1}a \in H$, so $b^{-1}ab \in Hb$. At the same time $b\in a^{-1}H$ (since $a$ and $a^{-1}$ lie in the same box), so $ab \in H$ and so $b^{-1}ab \in b^{-1}H$. The box $Hb \cap b^{-1}H$ is on the diagonal, so $b^{-1}ab$ is a loop in $\Theta_S$ based at the vertex corresponding to $Hb$ and $b^{-1}H$.
\end{proof}
    
\begin{defn}
    We call a transformation described in \Cref{prop:loop-shift} a \emph{loop shift}, since it corresponds to shifting a loop from one vertex to another, as shown in \Cref{fig:loop-shift} on page~\pageref{fig:loop-shift}.
\end{defn}

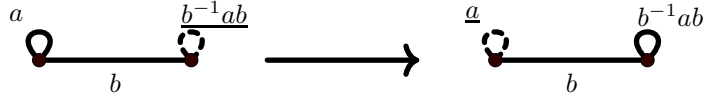
\begin{figure}[h]
    \begin{center}
    	\definecolor{ttqqqq}{rgb}{0.2,0,0}
    	\begin{tikzpicture}[line cap=round,line join=round,x=1cm,y=1cm]
    	%        \clip(-2,-4) rectangle (10,2);
    	\draw [line width=2pt, color=black] (0,0) .. controls (0+0.5,0+0.5) and (0-0.5,0+0.5) .. (0,0);
    	\draw [line width=2pt, dashed, color=black] (2,0) .. controls (2+0.5,0+0.5) and (2-0.5,0+0.5) .. (2,0);
    	\draw [line width=2pt, color=black] (0,0) -- (2,0);
    	\draw [fill=ttqqqq] (0,0) circle (2.5pt);
    	\draw [fill=ttqqqq] (2,0) circle (2.5pt);
    	\draw [color=black] (1,-0.3) node {$b$};
    	\draw [color=black] (-0.3,0.6) node {$a$};
    	\draw [color=black] (2.3,0.6) node {\underline{$b^{-1}ab$}};
    	
    	\draw [<-, line width=2pt] (5,0)-- (3.0,0);
    	
    	\draw [line width=2pt, dashed, color=black] (0+6,0) .. controls (0+0.5+6,0+0.5) and (0-0.5+6,0+0.5) .. (0+6,0);
    	\draw [line width=2pt, color=black] (2+6,0) .. controls (2+0.5+6,0+0.5) and (2-0.5+6,0+0.5) .. (2+6,0);
    	\draw [line width=2pt, color=black] (0+6,0) -- (2+6,0);
    	\draw [fill=ttqqqq] (0+6,0) circle (2.5pt);
    	\draw [fill=ttqqqq] (2+6,0) circle (2.5pt);
    	\draw [color=black] (1+6,-0.3) node {$b$};
    	\draw [color=black] (-0.3+6,0.6) node {\underline{$a$}};
    	\draw [color=black] (2.3+6,0.6) node {$b^{-1}ab$};
    	\end{tikzpicture}
    \end{center}
    \caption{A loop shift.}
    \label{fig:loop-shift}
\end{figure}

\subsubsection{Normal forms - octopuses and sweets}

As in \Cref{sec:normal-form}, we wish to classify the configurations which can be obtained from a given one by performing the operations listed above -- inversions, L-spins and loop shifts. We want to give representatives of each of the classes and show that they are not equivalent to each other.

\begin{defn}
	When given an inverse-dual graph of a configuration (in a self-inverse chessboard), we will call inversions, L-spins and loop shifts \emph{simple moves}. Two configurations (in a self-inverse chessboard) are called \emph{simply equivalent} if one can be obtained from the other with a series of simple moves.
\end{defn}

In fact, since we are considering the undirected inverse-dual graphs, the inversion doesn't do anything to the graph.

\begin{prop}
    \label{prop:sweets_inv}
    Let $S$ be a configuration in a self-inverse chessboard. Let $\Theta_S$ be its inverse-dual graph. Then, performing simple moves doesn't change:
    \begin{enumerate}
        \item the number of edges of $\Theta_S$,
        \item the connectedness of the (undirected) $\Theta_S$ (i.e. whether two vertices are connected by a path or not remains unchanged), which implies that we can restrict our attention to  connected components of inverse-dual graphs,
        \item whether or not a connected component of $\Theta_S$ is bipartite, and if it is:
        \item (if a connected component of $\Theta_S$ is indeed bipartite), the bipartite components of the connected component of $\Theta_S$ (and so, in particular, their sizes).
    \end{enumerate}
\end{prop}

\begin{proof}
    (1) Since Nielsen moves don't change the number of elements of a configuration and each  simple move corresponds to series of Nielsen moves, none of them changes the number of elements of $S$, i.e. the number of edges of $\Theta_S$.
    
    (2) In an L-spin all vertices involved are connected in the beginning and remain connected after the L-spin. This is also the case for the loop shift. Thus, simple moves don't change the connected components of $\Theta_S$.
    
    (3 \& 4) If a connected component of $\Theta_S$ is bipartite, then there are no loops (so, no loop shifts are possible) and then the walk (i.e. a path which may self-intersect) of length $3$ needed for performing the L-spin must have the vertices alternatingly in the two bipartite components. The new edge is between the first and last vertex, which must be in different connected components, and therefore the new edge is between the two bipartite components, not changing them.
\end{proof}

Now, we are going to show the possible `normal forms' of connected graphs that aren't bipartite.

\begin{defn}
    An undirected graph is called an \emph{octopus} if there exists a vertex $v_0$ (called the \emph{base}) such that for all vertices $v \ne v_0$ the degree of $v$ is $1$ and there exists a unique edge $vv_0$; these edges are called \emph{legs} of an octopus. The loops on $v_0$ are called \emph{heads} of an octopus.
\end{defn}

It is easy to see that the number of legs of an octopus is one less than the number of the vertices. An example of an octopus is given in \Cref{fig:octopus}.

\begin{figure}[h]
    \begin{center}
    	\definecolor{ttqqqq}{rgb}{0.2,0,0}
    	\begin{tikzpicture}[line cap=round,line join=round,x=1cm,y=1cm]
    	\draw [line width=2pt] (0,0)-- (-1.73,-1);
    	\draw [line width=2pt] (0,0)-- (-1,-1.73);
    	\draw [line width=2pt] (0,0)-- (0,-2);
    	\draw [line width=2pt] (0,0)-- (+1,-1.73);
    	\draw [line width=2pt] (0,0)-- (+1.73,-1);
    	\draw [line width=2pt, color=black](0,0) .. controls (-2,0) and (0,2) .. (0,0);
    	\draw [line width=2pt, color=black](0,0) .. controls (2,0) and (0,2) .. (0,0);
    	\draw [line width=2pt, color=black](0,0) .. controls (-1.5,1.5) and (1.5,1.5) .. (0,0);
    	\draw [fill=ttqqqq] (0,0) circle (2.5pt);
    	\draw [fill=ttqqqq] (-1.73,-1) circle (2.5pt);
    	\draw [fill=ttqqqq] (-1,-1.73) circle (2.5pt);
    	\draw [fill=ttqqqq] (0,-2) circle (2.5pt);
    	\draw [fill=ttqqqq] (+1,-1.73) circle (2.5pt);
    	\draw [fill=ttqqqq] (+1.73,-1) circle (2.5pt);
    	\end{tikzpicture}
    \end{center}
    \caption{An octopus.}
    \label{fig:octopus}
\end{figure}
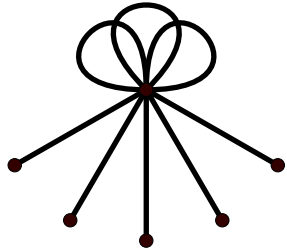

\begin{prop}
    \label{prop:normal_form_oct}
    Let $S$ be a configuration in a self-inverse chessboard, such that $\Theta_S$ is connected but not bipartite. Then, using simple moves, we can transform $S$ to some $S'$ such that $\Theta_{S'}$ is an octopus.
\end{prop}

\begin{proof}
    Since $\Theta_S$ isn't bipartite, there is an odd cycle in it, let's say it is $v_0v_1 \ldots v_{2n}$. We can use it to obtain a loop at $v_n$ by performing an L-spin on $v_1v_0,\:v_0v_{2n},\:v_{2n}v_{2n-1}$ interchanging $v_0v_{2n}$ with $v_1v_{2n-1}$, then with $v_2v_{2n-2}$ until getting $v_nv_n$, which is a loop at $v_n$. This procedure for $2n = 4$ is shown in \Cref{fig:oct.1}.
    
    \begin{figure}[h]
        \begin{center}
        	\definecolor{ttqqqq}{rgb}{0.2,0,0}
        	\begin{tikzpicture}[line cap=round,line join=round,x=1cm,y=1cm, scale=0.5]
        	%        \clip(-2,-4) rectangle (10,2);
        	\draw [fill=ttqqqq] (0,2) circle (3pt);
        	\draw [fill=ttqqqq] (-1.9,0.62) circle (3pt);
        	\draw [fill=ttqqqq] (-1.18,-1.62) circle (3pt);
        	\draw [fill=ttqqqq] (+1.9,0.62) circle (3pt);
        	\draw [fill=ttqqqq] (+1.18,-1.62) circle (3pt);
        	\draw [line width=2pt] (0,2) -- (-1.9,0.62);
        	\draw [line width=2pt] (0,2) -- (+1.9,0.62);
        	\draw [line width=2pt] (-1.9,0.62) -- (-1.18,-1.62);
        	\draw [line width=2pt] (+1.9,0.62) -- (+1.18,-1.62);
        	\draw [line width=2pt] (-1.18,-1.62) -- (+1.18,-1.62);
        	\draw [color=black] (0,2.3) node {$v_2$};
        	\draw [color=black] (-2.4,0.62) node {$v_1$};
        	\draw [color=black] (+2.4,0.62) node {$v_3$};
        	\draw [color=black] (-1.3,-2) node {$v_0$};
        	\draw [color=black] (+1.3,-2) node {$v_4$};
        	
        	\draw [<-, line width=2pt] (5,0)-- (3.0,0);
        	
        	\draw [fill=ttqqqq] (0+8,2) circle (3pt);
        	\draw [fill=ttqqqq] (-1.9+8,0.62) circle (3pt);
        	\draw [fill=ttqqqq] (-1.18+8,-1.62) circle (3pt);
        	\draw [fill=ttqqqq] (+1.9+8,0.62) circle (3pt);
        	\draw [fill=ttqqqq] (+1.18+8,-1.62) circle (3pt);
        	\draw [line width=2pt] (0+8,2) -- (-1.9+8,0.62);
        	\draw [line width=2pt] (0+8,2) -- (+1.9+8,0.62);
        	\draw [line width=2pt] (-1.9+8,0.62) -- (-1.18+8,-1.62);
        	\draw [line width=2pt] (+1.9+8,0.62) -- (+1.18+8,-1.62);
        	\draw [line width=2pt] (-1.9+8,0.62) -- (+1.9+8,0.62);
        	
        	\draw [<-, line width=2pt] (5+8,0)-- (3.0+8,0);        
        	
        	\draw [fill=ttqqqq] (0+16,2) circle (3pt);
        	\draw [fill=ttqqqq] (-1.9+16,0.62) circle (3pt);
        	\draw [fill=ttqqqq] (-1.18+16,-1.62) circle (3pt);
        	\draw [fill=ttqqqq] (+1.9+16,0.62) circle (3pt);
        	\draw [fill=ttqqqq] (+1.18+16,-1.62) circle (3pt);
        	\draw [line width=2pt] (0+16,2) -- (-1.9+16,0.62);
        	\draw [line width=2pt] (0+16,2) -- (+1.9+16,0.62);
        	\draw [line width=2pt] (-1.9+16,0.62) -- (-1.18+16,-1.62);
        	\draw [line width=2pt] (+1.9+16,0.62) -- (+1.18+16,-1.62);
        	\draw [line width=2pt] (16,2) .. controls (16-1,2+1) and (16+1,2+1) .. (16,2);
        	\end{tikzpicture}
        \end{center}
        \caption{Constructing a loop in the case of a non-bipartite graph.}
        \label{fig:oct.1}
    \end{figure}
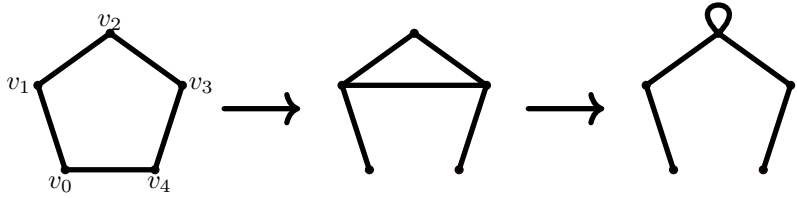
    
    Having obtained a loop at a vertex $v_0$, we can reduce the distance to $v_0$ of any vertex performing the following operation (which bears similarity to what we've done in the process of defining normal forms for non self-inverse chessboards). If $v_2$ is of distance 2 (in $\Theta_S$) from $v_0$, then let $v_0v_1v_2$ be a path. We can perform an L-spin on $v_0v_0, \: v_0v_1, \: v_1v_2$ interchanging $v_1v_2$ with $v_0v_2$ and thereby decreasing the distance of $v_2$ to $v_0$ without increasing the distance to $v_0$ of any other vertex. This procedure is illustrated in \Cref{fig:oct.2}.
    
    \begin{figure}[h]
	    \begin{center}
	    	\definecolor{ttqqqq}{rgb}{0.2,0,0}
	    	\begin{tikzpicture}[line cap=round,line join=round,x=1cm,y=1cm]
	    	
	    	\draw [line width=2pt] (0+6,0)-- (2+6,0);
	    	\draw [line width=2pt] (2+6,0)-- (1+6,1.732);
	    	\draw [line width=2pt, color=black](1+6,1.732) .. controls (1+1+6,1.732+1) and (1-1+6,1.732+1) .. (1+6,1.732);
	    	\draw [fill=ttqqqq] (1+6,1.732) circle (2.5pt);
	    	\draw [fill=ttqqqq] (0+6,0) circle (2.5pt);
	    	\draw [fill=ttqqqq] (2+6,0) circle (2.5pt);
	    	\draw [color=black] (-0.3+6,-0.2) node {$v_2$};
	    	\draw [color=black] (2.3+6,-0.2) node {$v_1$};
	    	\draw [color=black] (-0.4+1+6,1.732) node {$v_0$};
	    	
	    	\draw [<-, line width=2pt] (5+6,1)-- (3+6,1);        
	    	
	    	\draw [line width=2pt] (2+12,0)-- (1+12,1.732);
	    	\draw [line width=2pt] (1+12,1.732)-- (0+12,0);
	    	\draw [line width=2pt, color=black](1+12,1.732) .. controls (1+1+12,1.732+1) and (1-1+12,1.732+1) .. (1+12,1.732);
	    	\draw [fill=ttqqqq] (1+12,1.732) circle (2.5pt);
	    	\draw [fill=ttqqqq] (0+12,0) circle (2.5pt);
	    	\draw [fill=ttqqqq] (2+12,0) circle (2.5pt);
	    	\end{tikzpicture}
	    \end{center}
	    \caption{Decreasing the distances to $v_0$.}
	    \label{fig:oct.2}
	\end{figure}
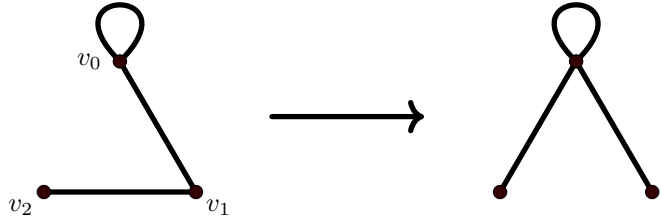
    
    Now we can assume that all vertices of $\Theta_S$ other than $v_0$ are of distance $1$ to $v_0$. Finally, if we get any edges between the vertices of distance $1$ to $v_0$, say an edge $v_1v_2$, we can move them to loops at $v_0$ via an L-spin on $v_1v_2, \:v_2v_0, \:v_0v_0$ interchanging $v_1v_2$ with $v_0v_0$ giving an additional head of the octopus. This procedure is illustrated in \Cref{fig:oct.3}.
    
    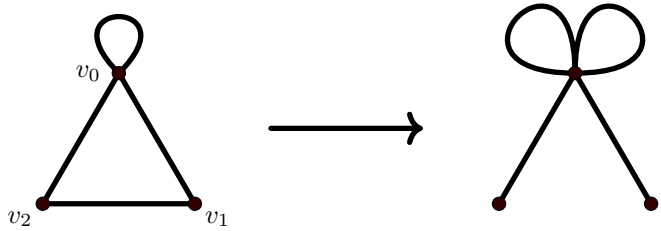
\begin{figure}[h]
        \begin{center}
        	\definecolor{ttqqqq}{rgb}{0.2,0,0}
        	\begin{tikzpicture}[line cap=round,line join=round,x=1cm,y=1cm]
        	%        \clip(-2,-4) rectangle (10,2);
        	
        	\draw [line width=2pt] (0+6,0)-- (2+6,0);
        	\draw [line width=2pt] (2+6,0)-- (1+6,1.732);
        	\draw [line width=2pt, color=black] (1+6,1.732)-- (0+6,0);
        	\draw [line width=2pt, color=black](1+6,1.732) .. controls (1+1+6,1.732+1) and (1-1+6,1.732+1) .. (1+6,1.732);
        	\draw [fill=ttqqqq] (1+6,1.732) circle (2.5pt);
        	\draw [fill=ttqqqq] (0+6,0) circle (2.5pt);
        	\draw [fill=ttqqqq] (2+6,0) circle (2.5pt);
        	\draw [color=black] (-0.3+6,-0.2) node {$v_2$};
        	\draw [color=black] (2.3+6,-0.2) node {$v_1$};
        	\draw [color=black] (-0.4+1+6,1.732) node {$v_0$};
        	
        	\draw [<-, line width=2pt] (5+6,1)-- (3+6,1);        
        	
        	%        \draw [line width=2pt, dashed, color=blue] (0+12,0)-- (2+12,0);
        	\draw [line width=2pt] (2+12,0)-- (1+12,1.732);
        	\draw [line width=2pt] (1+12,1.732)-- (0+12,0);
        	\draw [line width=2pt, color=black](1+12,1.732) .. controls (1-2+12,1.732) and (1+12,1.732+2) .. (1+12,1.732);
        	\draw [line width=2pt, color=black](1+12,1.732) .. controls (1+12,1.732+2) and (1+12+2,1.732) .. (1+12,1.732);
        	\draw [fill=ttqqqq] (1+12,1.732) circle (2.5pt);
        	\draw [fill=ttqqqq] (0+12,0) circle (2.5pt);
        	\draw [fill=ttqqqq] (2+12,0) circle (2.5pt);
        	\end{tikzpicture}
        \end{center}
        \caption{Final step towards an octopus.}
        \label{fig:oct.3}
    \end{figure}
    
    In the end we are left with a (possibly multiheaded) octopus.
\end{proof}

Now, we need to consider the case of the graph being bipartite.

\begin{defn}
    A graph is called a \emph{sweet} if there is an edge $v_0v_1$ such that every vertex $v$ other than $v_0$ and $v_1$ is connected by an edge to either $v_0$ or $v_1$, and such that every vertex other than $v_0$ and $v_1$ is of degree $1$. We call $v_0$ and $v_1$ the \emph{bases} of the sweet, the set of all edges $v_0v_1$ the \emph{core} of the sweet, while the other edges are called \emph{sticks}. 
\end{defn}

\begin{prop}
    \label{prop:normal_form_sweet}
    Let $S$ be a configuration in a self-inverse chessboard, such that $\Theta_S$ is connected and bipartite. Then, using simple moves, we can transform $S$ to some $S'$ such that $\Theta_{S'}$ is a sweet.
\end{prop}

\begin{proof}
    Let us then choose a particular edge $v_0v_1$, where $v_0$ and $v_1$ will be bases for the sweet. Now, if there is a vertex (say $v_3$) of distance to $\set{v_0,v_1}$ bigger than $1$, we can reduce it in a similar way to what was done in the proof of \Cref{prop:normal_form_oct}: Let $v_1v_2v_3$ be the path of length $2$. We perform an L-spin on $v_0v_1, \:v_1v_2,\:v_2v_3$ interchanging $v_2v_3$ with $v_0v_3$ and thereby reducing the distance of $v_3$ to $\set{v_0,v_1}$ without increasing the distances of any other vertices to it. A double usage of that procedure is shown in \Cref{fig:normal_form_sweet_red}, first applying that to $v_3$, then to $v_4$.
    
    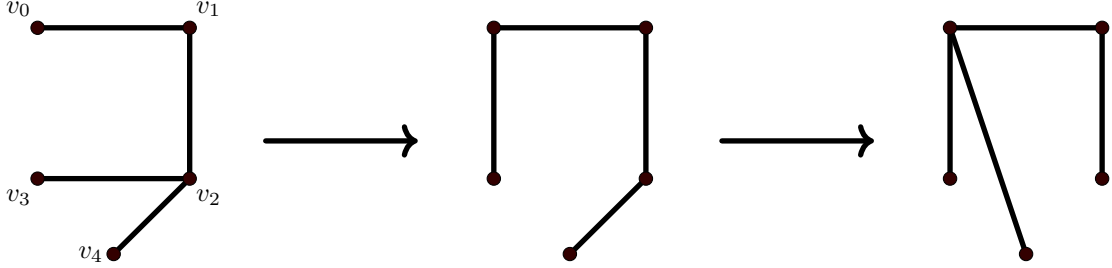
\begin{figure}[h]
        \begin{center}
        	\definecolor{ttqqqq}{rgb}{0.2,0,0}
        	\begin{tikzpicture}[line cap=round,line join=round,x=1cm,y=1cm]
        	\draw [line width=2pt] (0,0)-- (2,0);
        	\draw [line width=2pt] (2,0)-- (2,-2);
        	\draw [line width=2pt, color=black] (2,-2)-- (0,-2);
        	\draw [line width=2pt, color=black] (2,-2)-- (1,-3);
        	\draw [fill=ttqqqq] (0,0) circle (2.5pt);
        	\draw [fill=ttqqqq] (2,0) circle (2.5pt);
        	\draw [fill=ttqqqq] (2,-2) circle (2.5pt);
        	\draw [fill=ttqqqq] (0,-2) circle (2.5pt);
        	\draw [fill=ttqqqq] (1,-3) circle (2.5pt);
        	\draw (-0.25,0.25) node {$v_0$};
        	\draw (2.25,0.25) node {$v_1$};
        	\draw (-0.25,-2.25) node {$v_3$};
        	\draw (2.25,-2.25) node {$v_2$};
        	\draw (1-0.3,-3) node {$v_4$};;
        	
        	\draw [<-, line width=2pt] (5,-1.5)-- (3,-1.5);
        	
        	\draw [line width=2pt] (0+6,0)-- (2+6,0);
        	\draw [line width=2pt] (2+6,0)-- (2+6,-2);
        	\draw [line width=2pt, color=black] (0+6,0)-- (0+6,-2);
        	\draw [line width=2pt, color=black] (2+6,-2)-- (1+6,-3);
        	\draw [fill=ttqqqq] (0+6,0) circle (2.5pt);
        	\draw [fill=ttqqqq] (2+6,0) circle (2.5pt);
        	\draw [fill=ttqqqq] (2+6,-2) circle (2.5pt);
        	\draw [fill=ttqqqq] (0+6,-2) circle (2.5pt);
        	\draw [fill=ttqqqq] (1+6,-3) circle (2.5pt);
        	
        	\draw [<-, line width=2pt] (5+6,-1.5)-- (3+6,-1.5);
        	
        	\draw [line width=2pt] (0+6+6,0)-- (2+6+6,0);
        	\draw [line width=2pt] (2+6+6,0)-- (2+6+6,-2);
        	\draw [line width=2pt, color=black] (0+6+6,0)-- (0+6+6,-2);
        	\draw [line width=2pt, color=black] (0+6+6,0)-- (1+6+6,-3);
        	\draw [fill=ttqqqq] (0+6+6,0) circle (2.5pt);
        	\draw [fill=ttqqqq] (2+6+6,0) circle (2.5pt);
        	\draw [fill=ttqqqq] (2+6+6,-2) circle (2.5pt);
        	\draw [fill=ttqqqq] (0+6+6,-2) circle (2.5pt);
        	\draw [fill=ttqqqq] (1+6+6,-3) circle (2.5pt);
        	\end{tikzpicture}
        \end{center}
        \caption{Decreasing the distance of $v_3$ and $v_4$ to $\set{v_0,v_1}$.}
        \label{fig:normal_form_sweet_red}
    \end{figure}
    
    Now, we can assume that all of the vertices are of distance $1$ to $\set{v_0,v_1}$. Note that there can't be any edges within the neighbours of $v_0$ since the graph is bipartite, similarly with the neighbours of $v_1$. If there are any edges between the neighbours of $v_0$ and the neighbours of $v_1$, we can move these edges parallel to $v_0v_1$ obtaining a sweet. This is done via an L-spin. A double usage is shown in \Cref{fig:normal_form_sweet_core} where we first move the edge $v_2v_3$ and then $v_2v_4$.
    
    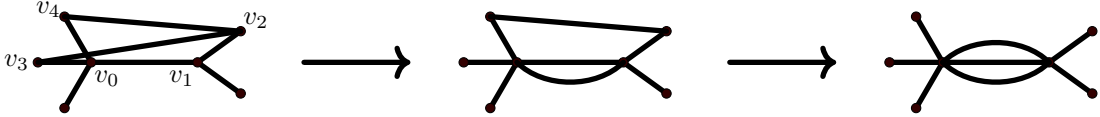
\begin{figure}[h]
		\begin{center}
			\definecolor{ttqqqq}{rgb}{0.2,0,0}
			\begin{tikzpicture}[line cap=round,line join=round,x=1cm,y=1cm, scale=0.7]
			%        \clip(-2,-4) rectangle (10,2);
			
			\draw [line width=2pt] (0,0)-- (2,0);
			\draw [line width=2pt, color=black] (2,0)-- (2+0.81,0+0.59);
			\draw [line width=2pt, color=black] (2,0)-- (2+0.81,0-0.59);
			\draw [line width=2pt, color=black] (0,0)-- (-0.5,1.732/2);
			\draw [line width=2pt, color=black] (0,0)-- (-0.5,-1.732/2);
			\draw [line width=2pt, color=black] (0,0)-- (-1,0);
			\draw [fill=ttqqqq] (0,0) circle (2.5pt);
			\draw [fill=ttqqqq] (2,0) circle (2.5pt);
			\draw [fill=ttqqqq] (2+0.81,0+0.59) circle (2.5pt);
			\draw [fill=ttqqqq] (2+0.81,0-0.59) circle (2.5pt);
			\draw [fill=ttqqqq] (-0.5,1.732/2) circle (2.5pt);
			\draw [fill=ttqqqq] (-0.5,-1.732/2) circle (2.5pt);
			\draw [fill=ttqqqq] (-1,0) circle (2.5pt);
			\draw [color=black] (3.1,0.8) node {$v_2$};
			\draw [color=black] (-1.4,0) node {$v_3$};
			\draw [color=black] (-0.8,1) node {$v_4$};
			\draw [color=black] (0.3,-0.3) node {$v_0$};
			\draw [color=black] (1.7,-0.3) node {$v_1$};
			
			\draw [line width=2pt, color=black] (-0.5,1.732/2)-- (2+0.81,0+0.59);
			\draw [line width=2pt, color=black] (-1,0)-- (2+0.81,0+0.59);
			
			\draw [<-, line width=2pt] (6,0)-- (4,0);
			
			\draw [line width=2pt] (0+8,0)-- (2+8,0);
			\draw [line width=2pt, color=black] (2+8,0)-- (2+0.81+8,0+0.59);
			\draw [line width=2pt, color=black] (2+8,0)-- (2+0.81+8,0-0.59);
			\draw [line width=2pt, color=black] (0+8,0)-- (-0.5+8,1.732/2);
			\draw [line width=2pt, color=black] (0+8,0)-- (-0.5+8,-1.732/2);
			\draw [line width=2pt, color=black] (0+8,0)-- (-1+8,0);
			\draw [fill=ttqqqq] (0+8,0) circle (2.5pt);
			\draw [fill=ttqqqq] (2+8,0) circle (2.5pt);
			\draw [fill=ttqqqq] (2+0.81+8,0+0.59) circle (2.5pt);
			\draw [fill=ttqqqq] (2+0.81+8,0-0.59) circle (2.5pt);
			\draw [fill=ttqqqq] (-0.5+8,1.732/2) circle (2.5pt);
			\draw [fill=ttqqqq] (-0.5+8,-1.732/2) circle (2.5pt);
			\draw [fill=ttqqqq] (-1+8,0) circle (2.5pt);
			
			\draw [line width=2pt, color=black] (-0.5+8,1.732/2)-- (2+0.81+8,0+0.59);
			%        \draw [line width=2pt, color=black] (-1+6,0)-- (2+0.81+6,0+0.59);
			\draw [line width=2pt, color=black] (0+8,0) .. controls (0.5+8,-0.5) and (1.5+8,-0.5) .. (2+8,0);
			
			\draw [<-, line width=2pt] (6+8,0)-- (4+8,0);
			
			\draw [line width=2pt] (0+8+8,0)-- (2+8+8,0);
			\draw [line width=2pt, color=black] (2+8+8,0)-- (2+0.81+8+8,0+0.59);
			\draw [line width=2pt, color=black] (2+8+8,0)-- (2+0.81+8+8,0-0.59);
			\draw [line width=2pt, color=black] (0+8+8,0)-- (-0.5+8+8,1.732/2);
			\draw [line width=2pt, color=black] (0+8+8,0)-- (-0.5+8+8,-1.732/2);
			\draw [line width=2pt, color=black] (0+8+8,0)-- (-1+8+8,0);
			\draw [fill=ttqqqq] (0+8+8,0) circle (2.5pt);
			\draw [fill=ttqqqq] (2+8+8,0) circle (2.5pt);
			\draw [fill=ttqqqq] (2+0.81+8+8,0+0.59) circle (2.5pt);
			\draw [fill=ttqqqq] (2+0.81+8+8,0-0.59) circle (2.5pt);
			\draw [fill=ttqqqq] (-0.5+8+8,1.732/2) circle (2.5pt);
			\draw [fill=ttqqqq] (-0.5+8+8,-1.732/2) circle (2.5pt);
			\draw [fill=ttqqqq] (-1+8+8,0) circle (2.5pt);
			
			%        \draw [line width=2pt, color=black] (-0.5+8+8,1.732/2)-- (2+0.81+8+8,0+0.59);
			%        \draw [line width=2pt, color=black] (-1+6+8,0)-- (2+0.81+6+8,0+0.59);
			\draw [line width=2pt, color=black] (0+8+8,0) .. controls (0.5+8+8,-0.5) and (1.5+8+8,-0.5) .. (2+8+8,0);
			\draw [line width=2pt, color=black] (0+8+8,0) .. controls (0.5+8+8,+0.5) and (1.5+8+8,+0.5) .. (2+8+8,0);
			
			\end{tikzpicture}
		\end{center}
        \caption{Final step towards a sweet.}
        \label{fig:normal_form_sweet_core}
    \end{figure}
\end{proof}

Now, we notice that a sweet is specified by the number of edges in total and the number of sticks on each side. This allows us to conclude the following.

\begin{prop}\label{prop:os-nf}
    Let $S$ be a configuration in a self-inverse chessboard with a connected inverse-dual graph $\Theta_S$. Then $\Theta_S$ can be transformed via simple moves to a (unique up to vertex relabelling) octopus or sweet, which we call the \emph{normal form} of $S$. Furthermore, the normal form is a sweet if and only if the inverse-dual graph $\Theta_S$ is bipartite.
\end{prop}

\begin{proof}
    The fact that $\Theta_S$ can be transformed to an octopus (if $\Theta_S$ isn't bipartite) and a sweet (if $\Theta_S$ is bipartite) is a consequence of Propositions~\ref{prop:normal_form_oct} and~\ref{prop:normal_form_sweet}.
    
    Now, as shown in \Cref{prop:sweets_inv}, the simple moves do not change whether or not $\Theta_S$ is bipartite, and, if it is bipartite, its bipartite components. So an octopus is possible if and only if $\Theta_S$ isn't bipartite, while a sweet only if $\Theta_S$ is bipartite. An octopus is determined by the number of elements of $S$ and number of its vertices (i.e. rows/columns that $S$ can occupy), so no two different octopuses are equivalent. A sweet is similarly determined by the number of its edges, its vertices and the sizes of its connected components. Since these are invariant by \Cref{prop:sweets_inv}, two sweets can be transformed by simple moves into each other only by relabelling of the vertices.
\end{proof}

\subsubsection{Solvable configurations}

Just like in \Cref{sec:solvable_non-self-inverse}, we want to be able to tell which of the inverse-dual graphs can be transformed into left-right diagonal sets. For that we first study what being left-right diagonal corresponds to in the inverse-dual graph.

%Again, as before, the way to know which of the configurations are solvable is to first realise what being in 'the solved state' looks like in the graph and then to observe which normal forms it corresponds to.

\begin{prop}
    Let $S$ be a configuration in a self-inverse chessboard. Then $S$ is left-right diagonal if and only if $\Theta_S$ is a union of disjoint (directed) cycles.
\end{prop}

\begin{proof}
    Being left-right diagonal means that there is exactly one element in each of the columns and in each of the rows. Coming back to the directed graphs, this means that every vertex has exactly one edge coming in to it and exactly one edge coming out of it. Such a graph is a union of disjoint directed cycles (possibly including loops and $2$-cycles).
\end{proof}

Now, we want to identify the normal forms that these have.

\begin{prop}
    \label{prop:normal-forms}
    In the context of inverse-dual graphs, the normal form of an odd cycle is a single-headed octopus with an even number of legs, while the normal form of an even cycle is a sweet with two edges in the core and equal number of sticks on either side. Also, all such forms correspond to configurations that can be transformed to left-right diagonal configurations via simple moves.
\end{prop}

\begin{proof}
    Odd cycles aren't bipartite, so their normal forms are octopuses by \Cref{prop:os-nf}. Since the number of edges and the number of vertices involved are equal, these octopuses must be single-headed. Odd length implies even number of legs.
    
    Even cycles are bipartite with the bipartite components of equal sizes, so their normal forms are sweets (\Cref{prop:os-nf}) with equal sizes of bipartite components. The number of edges equal to the number of vertices determines the size of the core to be $2$.
    
    Finally, for each number $n$, there are cycles of length $2n$ and $2n+1$, equivalent to respectively a sweet with $n-1$ sticks on each side and an octopus with $2n$ legs. Thus all such normal forms correspond to configurations that can be transformed to be left-right diagonal with simple transformations.
\end{proof}

The proposition motivates the following definition.

\begin{defn}
	A connected inverse-dual graph $\Theta_S$ is \emph{solvable} if its normal form is either an octopus with even number of legs, or a sweet with equal number of sticks on each side.
\end{defn}

Examples of such an octopus and a sweet are illustrated in \Cref{fig:ex-solv}.

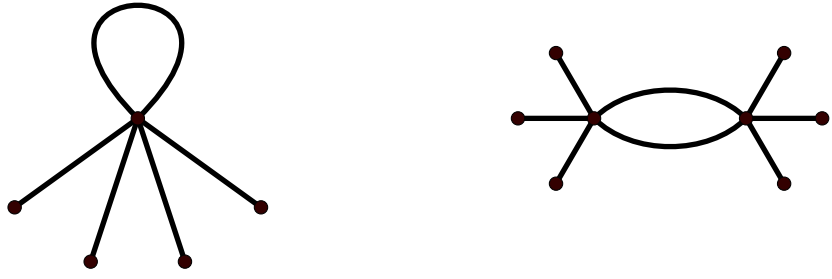
\begin{figure}
	\begin{center}
		\definecolor{ttqqqq}{rgb}{0.2,0,0}
		\begin{tikzpicture}[line cap=round,line join=round,x=1cm,y=1cm]
		%        \clip(-2,-4) rectangle (10,2);
		
		\draw [line width=2pt] (0,0)-- (-1.62,-1.18);
		\draw [line width=2pt] (0,0)-- (-0.62,-1.9);
		\draw [line width=2pt] (0,0)-- (+1.62,-1.18);
		\draw [line width=2pt] (0,0)-- (+0.62,-1.9);
		\draw [line width=2pt, color=black](0,0) .. controls (-2,2) and (2,2) .. (0,0);
		\draw [fill=ttqqqq] (0,0) circle (2.5pt);
		\draw [fill=ttqqqq] (-1.62,-1.18) circle (2.5pt);
		\draw [fill=ttqqqq] (-0.62,-1.9) circle (2.5pt);
		\draw [fill=ttqqqq] (+1.62,-1.18) circle (2.5pt);
		\draw [fill=ttqqqq] (+0.62,-1.9) circle (2.5pt);
		
		\draw [line width=2pt] (0+6,0) .. controls (0.5+6,0.5) and (1.5+6,0.5) .. (2+6,0);
		\draw [line width=2pt] (0+6,0) .. controls (0.5+6,-0.5) and (1.5+6,-0.5) .. (2+6,0);
		\draw [line width=2pt, color=black] (2+6,0)-- (2+0.5+6,0+1.732/2);
		\draw [line width=2pt, color=black] (2+6,0)-- (2+0.5+6,0-1.732/2);
		\draw [line width=2pt, color=black] (2+6,0)-- (2+1+6,0);
		\draw [line width=2pt, color=black] (0+6,0)-- (-0.5+6,1.732/2);
		\draw [line width=2pt, color=black] (0+6,0)-- (-0.5+6,-1.732/2);
		\draw [line width=2pt, color=black] (0+6,0)-- (-1+6,0);
		\draw [fill=ttqqqq] (0+6,0) circle (2.5pt);
		\draw [fill=ttqqqq] (2+6,0) circle (2.5pt);
		\draw [fill=ttqqqq] (2+0.5+6,0+1.732/2) circle (2.5pt);
		\draw [fill=ttqqqq] (2+0.5+6,0-1.732/2) circle (2.5pt);
		\draw [fill=ttqqqq] (2+1+6,0) circle (2.5pt);
		\draw [fill=ttqqqq] (-0.5+6,1.732/2) circle (2.5pt);
		\draw [fill=ttqqqq] (-0.5+6,-1.732/2) circle (2.5pt);
		\draw [fill=ttqqqq] (-1+6,0) circle (2.5pt);
		
		\end{tikzpicture}
	\end{center}
	\caption{Examples of solvable normal forms.}
	\label{fig:ex-solv}
\end{figure}

\Cref{prop:normal-forms} is restated in the following corollary.

\begin{cor}
    \label{cor:solv_self-inverse}
    If a configuration $S$ in a self-inverse chessboard is such that its inverse-dual graph $\Theta_S$ has all connected components solvable, then it is Nielsen equivalent to a left-right diagonal configuration in that chessboard.
\end{cor}

\subsubsection{Additional solvable configuration}
\label{sec:odd-octopi}
 We now know that octopuses with an odd number of legs and sweets with bipartite components of different sizes cannot be solved with simple moves (inversions, L-spins and loop shifts). However, we are able to weaken the condition of \Cref{cor:solv_self-inverse}.

For brevity, from now on we refer to configurations that can be transformed with simple moves to octopuses with even or odd number of legs, or sweets with equal number of sticks on each side as \emph{even octopi}, \emph{odd octopi} and \emph{equal sweets} respectively.

\begin{prop}
    \label{prop:additional-solv}
    Let $S$ be a union of a singleton configuration and a configuration $S'$ in a self-inverse chessboard $HgH$ which satisfy following conditions.
    \begin{enumerate}
        \item The singleton configuration consists of an element $h$ lying in $H$;
        \item $S'$ is such that all of its connected components are solvable or are odd octopuses;
        \item (possibly after inversions) $S'$ is left-diagonal.
    \end{enumerate}
    
    Then, $S$ is Nielsen-equivalent to a left-right diagonal configuration with one element in $H$ and the rest in $HgH$.
\end{prop}

\begin{proof}
    We are going to proceed by repeated use of an algorithm, which can be performed as long as there is some connected component in the form of an odd octopus. Each usage will decrease the value of the following counter:
    $$C(S) = \#(\textrm{connected components of $S$}) + 2\times \#(\textrm{connected components of $S$ being odd octopuses})$$
    thereby ensuring that at some point we are left with no odd octopuses. Then, by \Cref{cor:solv_self-inverse} the configuration can be transformed into a left-right diagonal form via simple moves.
    
    The condition that possibly after inversions $S$ is left-diagonal implies that no vertex in $\Theta_{S}$ has degree $0$. It also enforces the connected components to be either single-headed octopuses or sweets with two edges in the core.
    
    \textbf{Algorithm}
    
    Suppose that there is at least one connected component being an odd octopus, call it $S_0$. Take the element $x_0 \in S_0$ corresponding to the head of $S_0$ (i.e. the unique loop $v_0v_0$, where $v_0$ is the base of $S_0$) and the element $h\in H$, and perform the Nielsen move $h \mapsto x_0h$. Now, there are two options:
    \begin{enumerate}
        \item either $x_0h$ lies in one of the columns corresponding to vertices in $S_0$,
        \item or $x_0h$ lies in a column that doesn't correspond to a vertex in $S_0$.
    \end{enumerate}
    
    In the first case, we get a additional edge in $\Theta_{S_0}$, while the number of vertices remains unchanged, so the normal form of $S_0\cup\set{x_0h}$ is an octopus with two heads and an odd number of legs. The additional head can be shifted to the end (say $v_1$) of one of the legs via a loop shift and then this leg can be cut. Let $x_1$ be the element of $S_0$ corresponding to $v_0v_1$, we perform the Nielsen move $x_0 \mapsto x_0x_1^{-1} \in H$ (where $x_0$ corresponds to the loop) which leaves us with: two even octopuses and an element of $H$. Furthermore, none of the vertices (i.e. none of the rows of $HgH$) became empty, so we are in a situation satisfying the original hypotheses. This procedure is illustrated in \Cref{fig:cut_leg}.
    
    In the second case, we put an edge between two connected components of $\Theta_S$. Since we haven't done anything to the loop at the base of $S_0$, the new component must be a double-headed octopus. Let's transform it to that form and take $x_0, x_1$ to be elements corresponding to the two heads of the octopus. We perform the Nielsen move $x_0 \mapsto x_0x_1^{-1} \in H$, `cutting' one of the heads and thereby getting a single octopus and a single element in $H$. Also, no vertex became of degree $0$, so we have a configuration satisfying the original hypotheses.
    
    \textbf{Counter decreasing}
    
    The effect of both cases is summarised in \Cref{tab:counter}.
    
    \begin{table}[]
        \centering
        \begin{tabular}{c|ccc}
                    & no. components    & no. odd octopuses  & $C(S)$ \\ \hline
            Case 1. & $+1$  & $-1$  & $-1$  \\
            Case 2. & $-1$  & $0$ or $-1$  & $-1$ or $-3$
        \end{tabular}
        \caption{The effect of algorithm on the counter.}
        \label{tab:counter}
    \end{table}
    
    In the first case, the number of connected components increased by $1$ (since we `cut the octopuses into two'), while the number of odd octopuses decreased by $1$, so in total the  counter $C(S)$ decreased by $1$.
    
    In the second case, since the number of connected components decreased by $1$ and the number of odd octopuses remained the same (if we connected to a component of even number of elements) or decreased by $1$ (if we connected to a component with odd number of elements), the counter $C(S)$ decreased by $1$ or $3$.
    
    In both cases we observe the counter $C(S)$ decreasing, so at some point we get to a configuration where it is impossible to continue the algorithm, which is one that doesn't contain any odd octopuses. This is precisely what we wanted to get.
\end{proof}

\begin{figure}
    \begin{center}
    	\definecolor{ttqqqq}{rgb}{0.2,0,0}
    	\begin{tikzpicture}[line cap=round,line join=round,x=1cm,y=1cm, scale=0.5]
    	%        \clip(-2,-4) rectangle (10,2);
    	
    	\draw [line width=2pt] (0,0)-- (-1.62,-1.18);
    	\draw [line width=2pt] (0,0)-- (-0.62,-1.9);
    	\draw [line width=2pt] (0,0)-- (+1.62,-1.18);
    	\draw [line width=2pt] (0,0)-- (+0.62,-1.9);
    	\draw [line width=2pt] (0,0)-- (-2,0);
    	\draw [fill=ttqqqq] (-2,0) circle (2.5pt);
    	\draw [line width=2pt, color=black](0,0) .. controls (0,+2) and (-2,0) .. (0,0);
    	\draw [line width=2pt, color=black](0,0) .. controls (0,+2) and (+2,0) .. (0,0);
    	\draw [fill=ttqqqq] (0,0) circle (2.5pt);
    	\draw [fill=ttqqqq] (-1.62,-1.18) circle (2.5pt);
    	\draw [fill=ttqqqq] (-0.62,-1.9) circle (2.5pt);
    	\draw [fill=ttqqqq] (+1.62,-1.18) circle (2.5pt);
    	\draw [fill=ttqqqq] (+0.62,-1.9) circle (2.5pt);
    	
    	\draw [<-, line width=2pt] (4+1,0) -- (2+1,0);
    	
    	\draw [line width=2pt] (0+7+1+0.5,0)-- (-1.62+7+1+0.5,-1.18);
    	\draw [line width=2pt] (0+7+1+0.5,0)-- (-0.62+7+1+0.5,-1.9);
    	\draw [line width=2pt] (0+7+1+0.5,0)-- (+1.62+7+1+0.5,-1.18);
    	\draw [line width=2pt] (0+7+1+0.5,0)-- (+0.62+7+1+0.5,-1.9);
    	\draw [line width=2pt] (0+7+1+0.5,0)-- (-2+7+1+0.5,0);
    	\draw [fill=ttqqqq] (-2+7+1+0.5,0) circle (2.5pt);
    	\draw [line width=2pt, color=black](0+7+1+0.5,0) .. controls (0+7-2+1+0.5,+2) and (2+7+1+0.5,+2) .. (0+7+1+0.5,0);
    	\draw [line width=2pt, color=black](0+7-2+1+0.5,0) .. controls (0+7-2-2+1+0.5,+2) and (+2+7-2+1+0.5,2) .. (0+7-2+1+0.5,0);
    	\draw [fill=ttqqqq] (0+7+1+0.5,0) circle (2.5pt);
    	\draw [fill=ttqqqq] (-1.62+7+1+0.5,-1.18) circle (2.5pt);
    	\draw [fill=ttqqqq] (-0.62+7+1+0.5,-1.9) circle (2.5pt);
    	\draw [fill=ttqqqq] (+1.62+7+1+0.5,-1.18) circle (2.5pt);
    	\draw [fill=ttqqqq] (+0.62+7+1+0.5,-1.9) circle (2.5pt);
    	
    	\draw [<-, line width=2pt] (4+6+1+3,0) -- (2+6+1+3,0);
    	
    	\draw [line width=2pt] (0+7+7+4-0.5,0)-- (-1.62+7+7+4-0.5,-1.18);
    	\draw [line width=2pt] (0+7+7+4-0.5,0)-- (-0.62+7+7+4-0.5,-1.9);
    	\draw [line width=2pt] (0+7+7+4-0.5,0)-- (+1.62+7+7+4-0.5,-1.18);
    	\draw [line width=2pt] (0+7+7+4-0.5,0)-- (+0.62+7+7+4-0.5,-1.9);
    	% \draw [line width=2pt] (0+7+7+4-0.5,0)-- (-2+7+7+4-0.5,0);
    	\draw [fill=ttqqqq] (-2+7+7+4-0.5,0) circle (2.5pt);
    	\draw [line width=2pt, color=black](0+7+7+4-0.5,0) .. controls (0+7-2+7+4-0.5,+2) and (2+7+7+4-0.5,+2) .. (0+7+7+4-0.5,0);
    	\draw [line width=2pt, color=black](0+7-2+7+4-0.5,0) .. controls (0+7-2-2+7+4-0.5,+2) and (+2+7-2+7+4-0.5,2) .. (0+7-2+7+4-0.5,0);
    	\draw [fill=ttqqqq] (0+7+7+4-0.5,0) circle (2.5pt);
    	\draw [fill=ttqqqq] (-1.62+7+7+4-0.5,-1.18) circle (2.5pt);
    	\draw [fill=ttqqqq] (-0.62+7+7+4-0.5,-1.9) circle (2.5pt);
    	\draw [fill=ttqqqq] (+1.62+7+7+4-0.5,-1.18) circle (2.5pt);
    	\draw [fill=ttqqqq] (+0.62+7+7+4-0.5,-1.9) circle (2.5pt);
    	
    	\end{tikzpicture}
    \end{center}
    \caption{Cutting an odd octopus into two even octopuses.}
    \label{fig:cut_leg}
\end{figure}
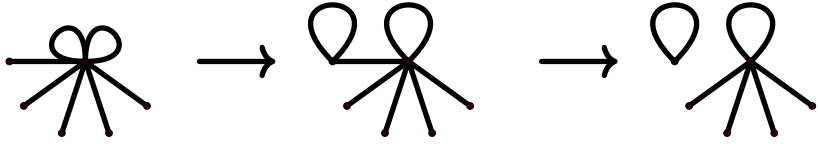

\subsection{Applications of configurations and inverse-dual graphs}\label{sec:applications}

The combined techniques of sections~\ref{sec:non-self-inverse-chessboards} and~\ref{sec:self-inverse-chessboards} give us the following proposition.

\begin{thm}
	\label{cor:final}
	Let $H\leq G$ be a subgroup and $S$ be a left transversal generating $G$. Then, if the following conditions are satisfied, $S$ is Nielsen-equivalent to a left-right transversal.
	\begin{enumerate}
		\item For $g\in G$ such that $HgH \neq Hg^{-1}H$, for $T_g = S_g\cup S_{g^{-1}}^{-1}$, the configuration $T_g$ has only square connected components.
		\item For $g\in G$ such that $HgH = Hg^{-1}H$, the graph $\Theta_{S_g}$ has connected components in the form of any octopuses or equal sweets.
	\end{enumerate}
\end{thm}

\begin{proof}
    We look one-by-one at the self-inverse chessboards and pairs of non-self-inverse chessboards which are inverses of each other.
    
    For each pair $HgH$ and $Hg^{-1}H$ of non-self-inverse chessboards, we note that the hypotheses of \Cref{prop:solv} are satisfied, since each connected component of $T_g$ has an equal number of rows and columns, while because $T_g = S_g\cup S_{g^{-1}}^{-1}$ and $S_g$ has exactly one element in each row, while $S_{g^{-1}}^{-1}$ has one element in each column, each connected component must have the number of elements equal to twice the number of rows (and columns).
    
    For a self-inverse chessboard, we can apply \Cref{prop:additional-solv} directly.
    
    After these procedures, for each $g$, the set $S_g$ is diagonal and therefore the whole set $S$ is diagonal.
\end{proof}

We now present two consequences of our results on configurations and inverse-dual graphs. One application is to a situation where the possible chessboards are small enough to not allow non-square connected components, i.e. when $[H:xHx^{-1}\cap H]\leq 2$ for all $x$ (so $H$ is very close to normal in $G$). The other application uses a technique similar to the proof of \Cref{lem:cyclic-nielsen} for finding left-right transversal generating sets for $H\leq G$ when $H\cong C_n$ , i.e. being able to Nielsen transform the trivial element $e$ to any other element of the group.

In each case, we make incremental progress on \Cref{main Q} or \Cref{general Q}, by adding additional hypotheses.

\begin{prop}\label{prop:almost-normal}
    Let $H\leq G$ be a subgroup of finite index such that $\forall x\in G$ we have $ [H:xHx^{-1}\cap H]\leq 2$. Then any generating multiset of size $[G:H]$ can be Nielsen transformed to a left-right transversal.
\end{prop}

\begin{proof}
    We start by taking a generating multiset $S$ of size $[G:H]$ and Nielsen transforming it to a left transversal. All of the chessboards are either $1\times 1$ or $2\times 2$ as by \cite[Proposition 9]{cosetgraph} the size of the chessboard containing $gH$ is $[G : gHg^{-1} \cap H]/[G:H]$. The $1\times 1$ chessboards correspond to the cosets of $H$ in the normaliser, and the elements of $S$ in them already form a diagonal set. Therefore, we'll only consider the $2\times 2$ chessboards from now on.
    
    For the self-inverse chessboards, we are either in the situation of having a diagonal set, or a configuration represented by a single-headed octopus with one leg. In particular, we cannot get a non-equal sweet.
    
    For the non-self inverse chessboards, after moving the elements from $Hg^{-1}H$ to $HgH$ with the inversion map, we get $4$ elements in a $2\times 2$ board, which can either form two connected components (each of size $1\times 1$) or form one connected component of size $2\times 2$. In both cases, these are square.
    
    Now, we can apply \Cref{cor:final} to conclude that $S$ can be Nielsen transformed to a left-right transversal.
\end{proof}

We now provide a proof for \Cref{thm:malnormal}. We note that the hypotheses of this theorem imply that $G$ must be finite. This is due to the fact that infinite groups can't have malnormal groups of finite index. Indeed, taking $H$ malnormal in $G$, if $[G:H]=[G:gHg^{-1}] < \infty$, then we also have $[G:H\cap gHg^{-1}] < \infty$, but $H\cap gHg^{-1} = \set{e}$, so this would imply that $|G|< \infty$.

\vspace{\topsep}
\begin{proof}[Proof of \Cref{thm:malnormal}]
    We start with the observation that for a malnormal subgroup $H\leq G$, the individual boxes in the chessboards other than $H$ correspond to singleton sets, since
    $$|xH\cap Hx| = |xHx^{-1}\cap H| = 1$$
    for $x\not \in H$. This means that whenever two dots end up in the same box, we are able to obtain the trivial element $e$, which can be Nielsen transformed to any other element. Similarly to the previous lemma, we start with a left generating transversal $S$.
    
    Firstly, given a pair $HgH\neq Hg^{-1}H$, we move all of the elements to one of these chessboards by inversion. Then, after transforming the configuration to its normal form (\Cref{defn:config-nf}) we get some number of connected components, each having two dots in its top left corner. We use this to first get $e$ in our multiset and then Nielsen transform $e$ (without using it in the Nielsen moves) to a box which will connect two distinct connected components in that chessboard. After a finite number of such operations we obtain just one connected component in $HgH$, which has to have equal vertical and horizontal dimensions. By \Cref{prop:solv} this is a solvable configuration, so by \Cref{lem: solvable} it is L-spin equivalent to one of the form $A\cup B^{-1}$, were $A \subset HgH$ and $B \subset Hg^{-1}H$ are diagonal sets.
    
    Secondly, given a self-inverse chessboard, i.e., a double coset $HgH=Hg^{-1}H$, by \Cref{prop:normal-forms} the normal forms of connected components of $S_g$ are either octopuses or sweets with core consisting of two edges. In the case of sweets, the two parallel edges in the core represent two elements in the same box, one of which can then be transformed to $e$, which then can be transformed to add a loop on one of the vertices of the inverse-dual graph of the component, thereby changing its normal form from a sweet to an octopus. Repeating the procedure for every self-inverse chessboard, for every connected component whose normal form is a sweet, we get to a situation, where all of components are actually octopuses. Now we can apply \Cref{cor:final} to conclude that $S$ is indeed equivalent to a left-right transversal.
\end{proof}

The procedure also gives us the following quantitative result.

\begin{thm}
	Let $H\leq G$ be a malnormal subgroup such that $\rank(G) \leq [G:H]$. Then:
	$$\rank(G) \leq [G:H] - \frac{\#(\emph{non-self-inverse chessboards})}{2}$$
\end{thm}

\begin{proof}
    According to \Cref{thm:malnormal}, we can take any generating set of size $[G:H]$ and Nielsen-transform it to a left-right transversal $S$.
    
    Now, for each pair $HgH,Hg^{-1}H$ of non-self inverse chessboards which are inverses to each other, we can use the inversion map to transform elements of the multiset from $Hg^{-1}H$ to $HgH$, forming configurations $S_g \cup S_{g^{-1}}^{-1}$ which are solvable. Now, when we transform these to their normal forms, \Cref{prop:solv} tells us that these normal forms have two dots in the top-left corner.
    
    Thus, denoting by $k$ the number of pairs $HgH,Hg^{-1}H$ of non-self-inverse chessboards, we Nielsen-transformed $S$ to a configuration with at least $k$ occurrences of two elements being in one box. Because $H$ is a malnormal subgroup, each box in a chessboard of $H$ (excluding the chessboard $H$ itself) contains precisely one group element. Thus a pair of elements of $S$ in the same box must be of the form \set{g,g}. Any such pair can be Nielsen transformed in the following way: $\set{g,g} \mapsto \set{g,e}$. Thus, we get $k$ occurrences of the trivial element $e$ in $S$, which don't play a role in generating $G$.
    
    This implies that the minimal size of a generating set, i.e. the rank of $G$, is less or equal to $|S|-k$, which is equal to
    $$[G:H] - \frac{\#(\textrm{non-self-inverse chessboards})}{2}$$
\end{proof}

\noindent \scriptsize{\textsc{King's College, Cambridge,
\\King's Parade, Cambridge, CB2 1ST, UK. 
\\mcc56@cam.ac.uk
\vspace{5pt}
\\King's College, Cambridge,
\\King's Parade, Cambridge, CB2 1ST, UK. 
\\rhc33@cam.ac.uk
\vspace{5pt}
\\Selwyn College, Cambridge,
\\Grange Road, Cambridge,  CB3 9DQ, UK. 
\\ojd30@cam.ac.uk
\vspace{5pt}
\\King's College, Cambridge,
\\King's Parade, Cambridge, CB2 1ST, UK. 
\\ppp24@cam.ac.uk}}

\end{document}